\documentclass{article}[12pt]
\usepackage{amssymb}
\usepackage{latexsym}
\usepackage{amsmath}
\usepackage{mathrsfs}
\usepackage[all]{xy}
\usepackage{enumerate}
\usepackage{amsthm}
\usepackage{ifthen}
\usepackage{tikz}
\usepackage{hyperref} 
\usetikzlibrary{arrows}
\hypersetup{colorlinks, citecolor=blue}

\newtheorem{thm}{Theorem}[section]
\newtheorem{lem}[thm]{Lemma}
\newtheorem{prop}[thm]{Proposition}
\newtheorem{cor}[thm]{Corollary}
\theoremstyle{definition}
\newtheorem{defn}[thm]{Definition}

\newtheorem{example}[thm]{Example}
\newtheorem{conv}[thm]{Convention}
\newtheorem{convention}[thm]{Convention}
\newtheorem{conj}[thm]{Conjecture}
\newtheorem{question}[thm]{Question}

\newcommand{\tr}{\textrm}

\newcommand{\mo}{{-1}}
\newcommand{\pmo}{{\pm 1}}
\renewcommand{\tilde}{\widetilde}
\newcommand{\FRS}{\ensuremath{F_{R(S)}}}
\newcommand{\Z}{\mathbb{Z}}
 \newcommand{\bra}{\langle}
\newcommand{\kett}{\rangle}
\newcommand{\ol}{\overline}
\newcommand{\rk}{\tr{rank}}
\newcommand{\F}{{\widetilde{F}}}

\newcommand{\x}{{\ol{x}}}
\newcommand{\z}{{\ol{y}}}
\newcommand{\y}{{\ol{y}}}

\newcommand{\B}{\mathcal{B}}

\newcommand{\stab}{\tr{stab}}

\newcommand{\R}{\mathcal{R}}

\newcommand{\FRSP}{F_{R(S')}}
\newcommand{\bfxk}{{\bra F, \x \kett}}

\newcommand{\brakett}[1]{{\bra #1 \kett}}
\newcommand{\fix}{\tr{fix}}
\newcommand{\te}{{e}}

\newcommand{\Hom}{\tr{Hom}_F}
\newcommand{\FRSi}{F_{R(S_i)}}
\newcommand{\FRSj}{F_{R(S_j)}}
\newcommand{\G}{\mathcal{G}}
\renewcommand{\B}{\mathcal{B}}
\newcommand{\axis}{\tr{Axis}}
\newcommand{\bk}{\brakett}
\renewcommand{\L}{\mathcal{L}}
\newcommand{\W}{\mathcal{W}}

\newcommand{\uhd}{\tr{{\bf uhd}}}
\newcommand{\uhdf}{\uhd_F}
\newcommand{\Haches}{\{h_1,\ldots,h_s\}}
\newcommand{\define}{\emph}
\newcommand{\relpres}[2]{\left\{\begin{array}{l}#1; \\
      #2 \end{array}\right.}
\newcommand{\Fh}{\widehat{F}}

\newcommand{\writeabove}[3]{\draw (#1.north) +(0,#3) node {#2}}
\newcommand{\dottedcircle}[2]{\draw[dotted,fill=black!10!white] (#1) circle (#2)}
\newcommand{\writeright}[3]{\draw (#1.east) +(#3,0.2) node {#2}}

\title{The fully residually $F$ quotients of $F*\bk{x,y}$} \author{Nicholas
  W.M. Touikan\footnote{Supported by NSERC PDF}\\Centre Interuniversitaire de Recherche en
  G\'eom\'etrie et Topologie\\UQ\`AM \\email: \texttt{nicholas.touikan@gmail.com}}
\begin{document}

\maketitle \abstract{We describe the fully residually $F$ groups; or
  limit groups relative to $F$; that are quotients of $F*\bk{x,y}$. We
  use the structure theory of finitely generated fully residually free
  groups to produce a finite list of possible types of cyclic JSJ
  decompositions modulo $F$ that can arise. We also give bounds on
  uniform hierarchical depth.}  \tableofcontents

\section{Introduction}
Systems of equations over free groups have been a very important and
fruitful subject of study in the field of combinatorial and geometric
group theory. A major achievement was the algorithm due to
Makanin and Razborov \cite{Makanin-1982, Razborov-1987} which produces
a complete description of the solution set of an arbitrary finite
system of equations over a free group. The method, however, is
algorithmic and uses surprisingly little algebra.

For a free group $F$, the classical algebraic geometry viewpoint given
by Baumslag, Myasnikov and Remeslennikov in \cite{BMR-1998}
established fully residually $F$ groups as the key algebraic
structures in the theory of systems of equations with coefficients
over $F$.

These groups still remained rather intractable until the work of
Kharlampovich and Myasnikov \cite{KM-IrredI,KM-IrredII,KM-JSJ}, and
independently Sela in \cite{Sela-DiophI}, which shows that finitely
generated fully residually free, or limit groups, in fact have a very
nice structure. We apply this structure theory to prove the following
result:

\begin{thm}[The Main Theorem]\label{thm:the-main-result} 
  Let $\FRS$, a fully residually $F$ quotient of $F*\bk{x,y}$, be
  freely indecomposable modulo $F$. The underlying graph $X$ of its
  cyclic JSJ decomposition modulo $F$ has at most 3 vertices, and $X$
  has at most two cycles. All vertex groups except maybe the vertex
  group $F\leq \F \leq \FRS$ are either free of rank 2 or free
  abelian. In all cases the vertex group $\F$ itself can be generated
  by $F$ and two other elements. Finally $\FRS$ has uniform
  hierarchical depth relative to $F$ at most 4.
\end{thm}

This result can be seen as a generalization of Chiswell and
Remeslennikov 's classification in \cite{CR-2000} of the fully
residually $F$ quotients of $F*\bk{x}$, which enabled them to finally
give a proof of a result on the solution sets of equations in one
variable over $F$ claimed independently by Appel and Lorenc
\cite{Appel-1969,Lorenc-1968}. The fully residually free groups
generated by at most three elements were classified in
\cite{FGMRS-1998}. We will give examples which show that the class of
fully residually $F$ quotients of $F*\bk{x,y}$ is considerably richer.

\subsection{Acknowledgements}
The author wishes to thank Olga Kharlampovich for numerous extremely
useful discussions and the anonymous referee who found many mistakes
and made good suggestions for improving the exposition.

\subsection{Definitions and notation}
We will denote the commutator $x^\mo y^\mo x y = [x,y]$. For
conjugation we will use the following notation:\begin{eqnarray*} x^w &
  = & w^{-1}xw\\ {}^wx & = & w x w^{-1}
\end{eqnarray*} We use this convention since
${}^x({}^yw)={}^{xy}w$. We shall also  denote by $\G(X)$ a graph of groups
with underlying graph $X$. We shall denote by $\rk(G)$ the minimal cardinality among all
generating sets of $G$.

\subsubsection{Fully residually $F$ groups}
Throughout this paper $F$ will denote a fixed free group of rank $N$.

\begin{defn}
  A group $G$ equipped with a distinguished monomorphism \[
  i:F\hookrightarrow G\] is called an \emph{$F$-group} we denote this
  $(G,i)$. Given $F$-groups $(G_1,i_1)$ and $(G_2,i_2)$, we define an
  $F$-homomorphism to be a homomorphism of groups $f$ such that the
  following diagram commutes:\[\xymatrix{ G_1 \ar[r]^f & G_2 \\ F
    \ar[u]^{i_1} \ar[ur]_{i_2}&}\] We denote by $\Hom(G_1,G_2)$ the
  set of $F$-homomorphisms from $(G_1,i_1)$ to $(G_2,i_2)$.
\end{defn}

In the rest of the paper the distinguished monomorphisms will not be
explicitly mentioned, seeing as the inclusions will always be obvious.

\begin{defn}
  Let $G$ and $H$ be groups. A collection of homomorphism $\Phi$ from $G$
  to $H$ \emph{discriminates} $G$ if for every finite subset $P
  \subset G$ there is some $f \in \Phi$ such that the restriction $f|_P$
  is injective.
\end{defn}

\begin{defn}
	A group $G$ is \emph{fully residually $F$} if $\Hom(G,F)$ discriminates $G$.
\end{defn}

This definition is slightly non-standard in that we are
\emph{requiring} fully residually $F$ groups to be discriminated by
\emph{$F$-homomorphisms}. More generally we shall say that a group $G$
is \emph{fully residually free} to mean that it is discriminated by
some arbitrary set of homomorphisms into some free group.
 
For the rest of the paper $\FRS$ shall denote a fully
residually $F$ quotient of $F*\bk{x,y}$. The notation reflects the
fact that $\FRS$ is the coordinate group of system of equations $S(x,y)$
over $F$. Since $\FRS$ is fully residually $F$ it is in fact the
coordinate group of an irreducible system of equations with
coefficients in $F$ and variables $x,y$. We refer the reader to
\cite{BMR-1998} for a complete treatment of the interpretation of 
fully residually $F$ groups in algebraic geometry.

The following facts will be used throughout this paper and follow
easily from the definitions:

\begin{thm}Let $G$ be fully residually free.
\begin{itemize}
\item $G$ is torsion free.
\item Two elements $g,h$ either commute or generate a free subgroup of rank 2.
\item $G$ is a CSA group, i.e. maximal abelian subgroups groups are malnormal (which implies commutation transitivity).
\end{itemize}
\end{thm}

\begin{conv}
  Throughout this paper $F$ shall denote a free group of rank $N$, and
  $\FRS$ will denote a fully residually $F$ quotient of
  $F*\bk{x,y}$. The elements $\x,\y \in \FRS$ denote the images of
  $x,y$ respectively via the epimorphism $F*\bk{x,y} \rightarrow
  \FRS$.
\end{conv}

\subsubsection{Generalized JSJ decompositions and uniform hierarchical
  depth}\label{sec:uhd-and-jsj}
We assume some familiarity with Bass-Serre theory and cyclic JSJ
theory for groups (as introduced in \cite{R-S-JSJ}.) If $F\neq\FRS$
then it follows, for example from Corollary \ref{cor:get-split}, that
it has an essential cyclic or free splitting modulo $F$. If $\FRS \neq
F$ is freely indecomposable modulo $F$ then it has a non-trivial
cyclic JSJ decomposition modulo $F$. As a matter of terminology, \define{MQH stands for ``maximal quadratically hanging''. QH
  (i.e. ``quadratically hanging'') subgroups are the surface-type
  subgroups that arise from hyperbolic-hyperbolic pairs of
  splittings.}

\begin{conv} Unless stated otherwise, instead of saying the
  \emph{cyclic JSJ decomposition of $\FRS$ modulo $F$}, we will simply
  say the \emph{JSJ of $\FRS$}. 
\end{conv}  

\begin{defn}\label{defn:gen-JSJ}
  Let $\FRS = H_0*H_1*\ldots*H_m*F(Y)$ be a Grushko decomposition of
  $\FRS$ modulo $F$, with $F \leq H_0$. A \emph{generalized JSJ} of
  $\FRS$ is a graph of groups decomposition $\FRS = \pi_1(\G(X))$
  modulo $F$, where $\G(X)$ is a graph of groups, such
  that:\begin{itemize}
  \item Edge groups are either trivial or cyclic. 
  \item If $e_1,\ldots,e_k$ are the the edges of $X$ with trivial edge
    group, and $X_0, X_1,\ldots,X_{+1}$ are the connected components
    of $X \setminus \big(e_1 \cup \ldots \cup e_k\big)$ then, up to
    reordering, each graph of groups $\G(X_i)$ is the JSJ of $H_i$ for
    $i=1,\ldots m$, $\pi_1(\G(X_{m+1})) = F(Y)$, and $\G(X_0)$ is the
    JSJ of $\F$ modulo $F$
\end{itemize}
\end{defn}

\begin{conv}
  The JSJ of $\FRS$ will always have a vertex group containing $F$, we
  shall always denote this vertex group as $\F$. 
\end{conv}

Further details on JSJ theory are deferred to Section
\ref{sec:JSJ}. Uniform hierarchical depth is a very natural notion,
however the definition is complicated by the fact that subgroups of
fully residually $F$ groups need not themselves be fully residually
$F$.
\begin{defn}
  For a finitely generated fully residually free group $G$ we define its
  \define{uniform hierarchical depth}, denoted $\uhd(G)$, as
  follows:\begin{itemize}
    \item If $G$ is trivial, abelian, free or a surface group, then
      $\uhd(G)=0$.
    \item Otherwise let $G_1,\ldots,G_n$ be the vertex groups of the
      generalized JSJ of $G$. We set\[\uhd(G) =
      \max\big(\uhd(G_1),\ldots,\uhd(G_n)\big) +1.\] 
\end{itemize} 
If $G$ is a finitely generated fully residually $F$  group then the vertex groups of its
generalized JSJ are not necessarily $F$-groups. We define its
\define{uniform hierarchical depth relative to $F$}, denoted
$\uhdf(G)$, as follows:\begin{itemize}
      \item If $G=F$ then $\uhdf=0$.
      \item Otherwise let $F\leq \F, H_1,\ldots,H_n$ be the vertex
        groups of $G$'s generalized JSJ modulo $F$, then \[\uhdf(G) =
        \max\big(\uhdf(\F),\uhd(H_1),\ldots,\uhd(H_n)\big)+1.\]
      \end{itemize}
\end{defn}

For a finitely generated fully residually free group $G$, $\uhd(G)$ is
essentially the number of levels of the \emph{canonical analysis
  lattice} defined in \S4 of \cite{Sela-DiophI}. The main difference
is that a free product of $\uhd$ 0 groups has $\uhd$ 1, whereas the
analysis lattice of such a group has only a 0-level. It follows that
$\uhd(G$) is at least $l$, the number of levels of an analysis
lattice, and at most $l+1$. By Theorem 4.1 of \cite{Sela-DiophI}
$\uhd(G) < \infty$ for a finitely generated fully residually free
group. $\uhdf(\FRS) < \infty$ is an easy consequence of Theorem 4 of
\cite{KM-IrredII} (it is restated in this paper as Theorem
\ref{thm:ICE}.)

\subsubsection{Relative presentations}
Let $G_1,\ldots G_n$ be groups with presentations $\bk{X_1 \mid
  R_1},\ldots, \bk{X_n \mid R_n}$ respectively and $t_1,\ldots, t_k$ a set
of letters.  Let $R$ denote a set of words in $\bigcup X_i^\pmo \cup
\{t_1,\ldots, t_k\}^\pmo$ then we will define the \emph{relative
  presentation}
\[\bk{G_1,\ldots,G_n,t_1,\ldots t_k\mid R}\] to be the group defined
by the presentation \[\bk{X_1,\ldots,X_n,t_1,\ldots t_k \mid
  R_1,\ldots R_n, R}\] We assume the reader is familiar with the relative
presentation that can be given to the fundamental group of a graph of
groups $\G(X)$ which depends on some maximal spanning tree $T\subset X$
(see \S5 of \cite{Serre-arbres} for details.)

\begin{convention}
  The ``generators'' of the relative presentations will be the vertex
  groups of $\G(X)$ and \define{stable letters} corresponding to edges
  of $X\setminus T$, where $T \subset X$ is a maximal spanning
  tree. Vertex groups will always be written using capital letters,
  and stable letters will always be denoted in lower case.

The ``relations'' will always involve stable letters. Moreover we
will abuse notation and, for example, abbreviate the HNN extension
associated to the isomorphism $\psi:A_1 \rightarrow A_2$, $A_1 \leq G
\geq A_2$, simply as $\bk{G,t \mid A_1^t = A_2}; A_1,A_2 \leq G$.

For each edge $e \in T$ we will consider the images of the
corresponding edge group in the vertex groups to be identified. This
means that the corresponding edge group will be given as the
intersection of two vertex groups.
\end{convention}

\section{The classification of the fully residually $F$ quotients
  of $F*\bk{x,y}$}\label{sec:descriptions}

\subsection{Auxiliary results}

So far the only comprehensive classification theorems of fully
residually free groups in terms of the number of generators are the
following, recall that the \emph{rank} of a group is the minimum
cardinality of its generating sets.

\begin{thm}\cite{FGMRS-1998}\label{thm:FGMRS}
If $G$ is fully residually free group, then\begin{itemize}
\item if $\rk(G)=1$ then $G$ is infinite cyclic.
\item If $\rk(G)=2$ then $G$ is free or free abelian of rank 2.
\item If $\rk(G)=3$ then $G$ is either free, free abelian of rank 3,
  or $G$ is isomorphic to a rank 1 centralizer extension
  of a free group of rank 2.
\end{itemize}
\end{thm}

\define{Let $Ab(x,y)$ denote the free abelian group with basis
  $\{x,y\}$.}

\begin{thm}\cite{CR-2000} \label{thm:onevariable} The fully residually $F$
  quotients of $F*\bk{x}$ are:\begin{equation}\label{eqn:onevariable}
    \FRS = \left\{
  \begin{array}{l} F \\ F*\brakett{x} \\ F*_\bk{\alpha}Ab(\alpha,r)
  \end{array}\right.\end{equation} 
\end{thm}

Our proofs build on these previous classifications. This next result
which is interesting in its own right is also important for
proving some of the corollaries of our classification.

\begin{prop}\label{prop:b1-complexity} Let $F$ be a free group of rank
  $N$ and let $G$ be a fully residually $F$ group. Then $b_1(G)=N$ if
  and only if $G=F$.
\end{prop} 

If $G$ is the fundamental group of a graph of groups $\G(X)$ with
cyclic edge groups, then it is easy to estimate $b_1(G)$, the first
Betti number of $G$, in terms of the first Betti numbers of its vertex
groups. We have the following lower bound: for $T\subset X$ a maximal
spanning tree we have\begin{equation}\label{eqn:b1-bound}b_1(G)\geq
  \sum_{v \in T} b_1(G_v) - E\end{equation} where $E$ is the number of
edges in $T$. If there is an epimorphism $G\rightarrow H$ then $b_1
(G)\geq b_1(H)$. This fact is used to derive many of the corollaries
regarding first Betti numbers.

\subsection{When $\FRS$ is freely decomposable modulo $F$}\label{sec:prop-decompo}
This proposition is proved in Section \ref{sec:decomp-classification}
\begin{prop}\label{prop:decomposable}
  If $\FRS$ is freely decomposable modulo $F$ then,\[ \FRS = \left\{
    \begin{array}{l} F*\bk{t}\\ F*H; \tr{~where~}$H$\tr{~is fully
        residually free of rank 2}\\ \big(F*_\bk{\alpha}Ab(\alpha,r)\big)*\brakett{s}\\
    \end{array}\right.\]
\end{prop}

\begin{cor}
  If $\FRS$ is freely decomposable modulo $F$ then $\uhdf(\FRS) \leq 1$.
\end{cor}

\subsection{When all the vertex groups of the JSJ of $\FRS$ except $\F$ are
  abelian}\label{sec:prop-abelian}

The proof of this proposition takes up Section
\ref{sec:abelian-classification} and Section
\ref{sec:abelian-2-gen-mod-F}.

\begin{prop}\label{prop:abelian}
  If JSJ of $\FRS$ has abelian vertex groups, but only the non-abelian
  vertex group $F\leq\F$, then the possible underlying graphs of the
  JSJ are:\[ \xymatrix{u\bullet \ar@{-}[r] &\bullet v},
  \xymatrix{v\bullet \ar@{-}[r] &\bullet u\ar@{-}[r]& \bullet w},
  \tr{~or~~~~~~} \xymatrix{u\bullet \ar@{-}[r] \ar@(ul,dl)@{-}
    &\bullet v}.\] Moreover if there are two abelian vertex groups, or
  one of the abelian vertex groups has rank 3 then $\F=F$. In all
  cases $\F$ is generated by $F$ and at most two other elements and
  $b_1(\F) < b_1(\FRS).$
\end{prop}

This corollary follows from Proposition \ref{prop:b1-complexity} and
an easy estimation of the first Betti number.

\begin{cor}\label{cor:abelian-low-betti-uhd}
  If $\FRS$ is as in Proposition \ref{prop:abelian} and
  $b_1(\FRS)=N+1$ then $\FRS = F*_\bk{\alpha} Ab(\alpha,r)$. In
  particular $\uhdf(\FRS) = 1$.
\end{cor}

This next proposition is proved in Section \ref{sec:one-edge-case}.

\begin{prop}\label{prop:abelian-uhd}
  If the JSJ of $\FRS$ is as in Proposition \ref{prop:abelian}, then
  $\uhd_F(\FRS)\leq 3$.
\end{prop}

\subsection{When the JSJ of $\FRS$ has at least two non-abelian vertex
  groups}\label{sec:2-nonab}
This situation is covered in Section \ref{sec:2-nonab-classification}.
We now give a list of the possible JSJs for $\FRS$. Throughout this
section let $\F = F*_\bk{u}Ab(u,r)$ be a rank 1 centralizer extension
of $F$ (see Definition \ref{defn:ICE}) and let $H$ be a free group of
rank 2.



\begin{defn}\label{defn:almost-conjugate}
  A collection of elements $\alpha_1,\ldots,\alpha_n \in H$ are
  \define{almost conjugate in $H$} if there exists a cyclic subgroup
  $\bk{\gamma} \leq H$ and elements $g_1,\ldots,g_n \in H$ such that
  $\bk{g_i^\mo \alpha_ig_i} \leq \bk{\gamma}$ for $i=1,\ldots,n$.
\end{defn}

\subsubsection{When all the vertex groups are free non-abelian}\label{sec:allfree}
{\bf A:} If the underlying graph of the JSJ has only \emph{one edge} the possibilities
are:
\begin{enumerate}
\item $\FRS = F*_\bk{\alpha} H$.
\item $\FRS = F*_\bk{\alpha} Q$ with $Q$ a QH subgroup so that in fact \[\FRS =
  \bk{F,s,t \mid [s,t]=\alpha}; \alpha \in F.\]
\end{enumerate}
{\bf B:} If the underlying graph of the JSJ has \emph{two edges} the possibilities
are:
\begin{enumerate}
\item \[\FRS =
  \relpres{\bk{F,H,t \mid \beta^t = \beta'}}{\beta,\beta' \in H;
    \bk{\alpha} = \F \cap H} \]where $\alpha$ is not almost
  conjugate to $\beta$ or $\beta'$ in $H$. The subgroup $\bk{H,t}$
  is also free of rank 2.
\item \[ \FRS = \relpres{\bk{F,H,t \mid \beta^t = \gamma}}{\beta
    \in \F, \gamma \in H, \bk{\alpha} = \F \cap H}. \] 
\end{enumerate}
{\bf C:} If the underlying graph of the JSJ has \emph{three edges} the possibilities
are:
\begin{enumerate}
\item  \[\FRS =  \relpres{\bk{F,H,t,s \mid \beta^t = \gamma, \delta^s = \delta'}}{\beta
        \in F; \gamma,\delta,\delta' \in H; \bk{\alpha} = F \cap H}\]
     where $\bk{H,s}$ is also free of rank 2, moreover 
     $\alpha$ and  $\gamma$ are not almost conjugate in $H$.
   \item \[\FRS = \relpres{\bk{F,H,t,s, \mid \beta^t = \delta,
         \gamma^s = \epsilon}}{\beta,\gamma \in F, \delta,\epsilon \in
       H, \bk{\alpha} = F \cap H}\] where $H$ is generated by
     $\alpha,\delta,\epsilon$.
\end{enumerate}

\subsubsection{When there is an abelian vertex group}\label{sec:abvertgr}
{\bf A:} If the underlying graph of the JSJ has \emph{two edges} the only
possibility is \[\FRS = F*_\bk{\alpha}H*_\bk{\beta}Ab(\beta,r).\] 
{\bf B:} If the underlying graph of the JSJ has \emph{three edges} the
possibilities are:\begin{enumerate}

\item \[\FRS = \relpres{\bk{F,H,Ab(p,r),t \mid \alpha^t =
        \alpha'}}{\alpha,\alpha' \in H, \bk{u} = \F\cap H \\ \bk{p} =
      H \cap Ab(p,r)}\] moreover $u,p$ are not almost conjugate
    to either $\alpha$ or $\alpha'$ in $H$.

  \item \[\FRS = \relpres{\bk{F,H, Ab(\delta,r),t \mid \beta^t =
        \gamma }}{\beta \in F, \gamma \in H, \bk{\alpha} = F \cap H,
      \bk{\delta} = H \cap Ab(\delta,r)}.\] Moreover $\gamma$ and
    $\alpha$ are not almost conjugate in $H$.
\end{enumerate}

\subsubsection{When $\F$ is a rank 1 centralizer extension of $F$}\label{sec:ICEICE}
The possibilities for the JSJ are \begin{enumerate}
\item $\FRS = \F*_\bk{\alpha} H$.
\item \[\FRS = \relpres{ \bk{\F,H,t \mid \beta^t =
      \beta'}}{\beta,\beta' \in H; \bk{\alpha} = \F \cap H} \] where
  $\alpha$ is not almost conjugate to $\beta$ or $\beta'$ in $H$. The
  subgroup $\bk{H,t}$ is also free of rank 2.
\item \[ \FRS = \relpres{\bk{\F,H,t \mid \beta^t = \gamma}}{\beta \in
    \F, \gamma \in H, \bk{\alpha} = \F \cap H.} \]
\end{enumerate}

\begin{prop}\label{prop:2v-class}
  If the JSJ of $\FRS$ has more than one one non-abelian vertex group,
  then its JSJ is one of those given in sections \ref{sec:allfree},
  \ref{sec:abvertgr} or \ref{sec:ICEICE}.
\end{prop}

\begin{cor}\label{cor:nonab-uhd}
  Suppose the JSJ of $\FRS$ has more than one non-abelian vertex
  group then:
  \begin{itemize}
    \item If the JSJ of $\FRS$ has three vertex groups then $\uhd_F(\FRS) = 1$.
    \item $\uhd_F(\FRS) = 2$ if and only if the vertex group $F \leq
      \F$ is a centralizer extension of $F$. In particular $\FRS$ must
      have a non-cyclic abelian subgroup.
    \item If $b_1(\FRS) = N+1$ then it has no non-cyclic
      abelian subgroups, so $\uhd_F(\FRS)=1$.
\end{itemize}
\end{cor}

\begin{cor}\label{cor:QH}
  If the JSJ of $\FRS$ has a QH subgroup, then $\FRS = \bk{F,s,t \mid
    [s,t]=\alpha}; \alpha \in F.$
\end{cor}

\subsection{When the JSJ of $\FRS$ has one vertex group}\label{sec:prop-1v}
The proof of this proposition takes up Section
\ref{sec:1v-classification}. It should also be noted that the
arguments rely heavily upon the results of Sections
\ref{sec:prop-decompo} through \ref{sec:2-nonab}.

\begin{prop}\label{prop:1v}
  If the JSJ of $\FRS$ has only one vertex group $\F$, then
  $\F\neq F$ and it is generated by $F$ and two other
  elements. Moreover we have the following possibilities:
  \begin{itemize}
  \item[(I)] The JSJ of $\FRS$ has two edges, $\F$ doesn't
    contain any non-cyclic abelian subgroups and $\uhd_F(\FRS) = 2$.
  \item[(II)] The JSJ of $\FRS$ has one edge and $\uhd_F(\FRS) \leq 4$.
  \end{itemize}
\end{prop}
\subsection{Proof of the main theorem}
\begin{proof}[proof of Theorem \ref{thm:the-main-result}]
  $\FRS$ must fall into one of the situations of Sections
  \ref{sec:prop-abelian} through \ref{sec:prop-1v}. It therefore
  follows that $\FRS$ must conform to one of the descriptions given by
  propositions \ref{prop:abelian}, \ref{prop:2v-class}, and
  \ref{prop:1v}.
\end{proof}

\subsection{Examples}
We give some examples of the groups given in Section
\ref{sec:descriptions}. We first note that it is very easy to
construct examples that are freely decomposable modulo $F$. Examples of
Proposition \ref{prop:abelian} are easy to construct by taking centralizer
extensions of $F$ or $F*\bk{r}$ or by taking iterated centralizer
extensions of height 2. The next few examples are more delicate.

\begin{example}[Example of a group in Section \ref{sec:ICEICE} of type
  1.] \label{eg:D-I-1} Let $F=F(a,b)$. It is proved in
  \cite{Touikan-2007} that the group \begin{eqnarray*}
    G&=&\brakett{F,x,y|[a^\mo b a[b,a][x,y]^2x,a] =1}\\ &&=
    \brakett{F,x,y,t | [x,y]^2x= [a,b]a^\mo b^\mo a t;
      [t,a]=1}\end{eqnarray*} is freely indecomposable modulo $F$. Let
  $u=[a,b]a^\mo b^\mo a t$, and $w(x,y) = [x,y]^2x $ then \[ G =
  \F*_\brakett{u=w(x,y)}\bk{x,y}\] where $\F$ is a rank 1 centralizer
  extension of $F$. Moreover $G$ is shown to be fully residually $F$
  by the $F$-embedding into the iterated centralizer
  extension \[F_2=\bk{F,t,s\mid [t,a]=1, [s,u]=1}\] via the mapping,
  $x \mapsto s^\mo(b^\mo t)s$ and $y \mapsto s^\mo(b^\mo a b)s$.
\end{example}

This example above is also interesting since it is a one relator fully
residually $F$ group that cannot embed in a rank 1 centralizer
extension of $F$. See Corollary \ref{cor:countereg}.

\begin{example}[Example of a group in Section \ref{sec:abvertgr} of type
  B.1.]\label{eg:D-III} Let $F=F(a,b)$ and consider the iterated
  centralizer extension \[F_2=\bk{F,s,t \mid [s,a]=1,
    [t,(a^2(b^\mo a b)^2)^s]=1}\] One can check that the subgroup
  $K\leq\bk{F,s^\mo b s, t}$ has induced splitting:
  \[K=\relpres{\bk{F,H,Ab(p,t),r \mid \gamma'^r = \gamma}}{
    \gamma,\gamma' \in H, \bk{a} = F \cap H, \bk{p} = H \cap
    Ab(p,t)}\]
  where $H = s^\mo\bk{a,b^\mo a b}s$, $\gamma=s^\mo a s$,
  $\gamma'=s^\mo b^\mo a b s$, $r=s^\mo b s,$ and $p=(a^2(b^\mo a
  b)^2)^s$. Moreover it is freely indecomposable, fully residually $F$
  and generated by two elements modulo $F$.
\end{example}

\begin{example}[Example of a group in Section \ref{sec:ICEICE} of type
  3.] We modify Example \ref{eg:D-I-1}. Let $F=F(a,b)$ and let
\[F_1= \brakett{F,s,t,r|[t,a]=1,[s,b^\mo a b]=1,[u,r]=1}\] where
$u=[a,b] a^\mo b^\mo a t$. $F_1$ in iterated centralizer
extension. Let $x'=b^\mo t, y'=b^\mo a b$ and let $G=\brakett{F,r^\mo
  x' r, sr}$. Let $H=r^\mo\brakett{x',y'}r$ and consider $G\cap H$. We
see that $(sr)^\mo b^\mo a b (sr) = r^\mo b^\mo a b r$ so $H \leq G$,
on the other hand letting $y=(sr)$ and $x=r^\mo x' r$ and by Britton's
lemma we have a splitting: \[G = \relpres{\bk{\F,H,y\mid (b^\mo a b)^y
    = \gamma}}{b^\mo a b \in \F, \gamma \in H, \bk{u} = \F \cap H}\]
with $\F=\brakett{F,t} = {F(a,b)*_\bk{a} Ab(a,t)}, \gamma = {r^\mo
  b^\mo abr}, u =[a,b] a^\mo b^\mo a t$, and $H$ free of rank 2 and not
freely decomposable modulo edge groups.
\end{example}

\subsection{A conjecture and a question}
Conspicuously absent from the list is an example of a fully residually
$F$ quotient of $F*\bk{x,y}$ whose JSJ has only one vertex
group. Which leads to the following conjecture:

\begin{conj}
There are no fully residually $F$ quotients of $F*\bk{x,y}$ whose JSJ has only one vertex group.
\end{conj}

A related question is the following:

\begin{question}\label{question:free-vg}
  Is there a finitely generated fully residually free group whose JSJ
  has a single vertex group that is free?
\end{question}

By Theorem \ref{thm:FGMRS} the answer to Question
\ref{question:free-vg} is ``no'' in the case of three-generated fully
residually free groups.

\section{Graphs of groups and folding processes}

\subsection{Graphs of groups}
The main result of Bass-Serre theory is that minimal actions of groups
(without inversions) on simplicial trees correspond to splittings as
fundamental groups of graphs of groups. The account we give here is
only to fix the notation. We refer the reader to \cite{Serre-arbres}
for a full treatment of the subject.

\begin{defn} A \emph{graph of groups} $\mathcal{G}(A)$ consists of a
  connected directed graph $A$ with vertex set $VA$ and edges $EA$.
  $A$ is directed in the sense that to each $e\in EA$ there are
  functions $i:EA\rightarrow VA, t:EA\rightarrow VA$ corresponding to
  the \emph{initial and terminal} vertices of edges. To $A$ we
  associate the following:\begin{itemize}
    \item To each $v \in VA$ we assign a \emph{vertex group} $A_v$.
    \item To each $e \in EA$ we assign an \emph{edge group} $A_e$.
    \item For each edge $e \in EA$ we have monomorphisms
      \[i_e:A_e\rightarrow A_{i(e)}, ~t_e:A_e\rightarrow
      A_{t(e)}\] we call the maps $i_e,t_e$ \emph{boundary
    monomorphisms} and the images of these maps \emph{boundary subgroups}.
\end{itemize} We also formally define the following expressions: for
each $e \in EA$ \[(e^\mo)^\mo = e, ~i(e^\mo)=t(e),~ t(e^\mo)=i(e),~ i_{e^\mo}=t_e,~
  t_{e^\mo}=t_e\]
\end{defn}
We denote by $\pi_1(\mathcal{G}(A))$ the fundamental group of a graph
of groups.

\begin{defn}\label{defn:splitting}
  We say that a group \emph{splits} as the fundamental group as a
  graph of groups if $G=\pi_1(\mathcal{G}(A))$ and refer to the data
  $D=(G,\mathcal{G}(A))$ as a \emph{splitting}. A \define{cyclic
    splitting} is a splitting such that the edge groups are all
  cyclic. A \define{$(\leq \Z)$-splitting} is a splitting
  whose edge groups are trivial or infinite cyclic.
\end{defn}

\begin{defn}\label{defn:ga-path} A sequence of the form \[ a_0,
  e_1^{\epsilon_1}, a_1, e_2^{\epsilon_2},\ldots e_n^{\epsilon_n},
  a_n\] where $e_1^{\epsilon_1},\ldots e_n^{\epsilon_n}$ is an edge
  path of $A$ and where $a_i \in A_{i(e_{i+1}^{\epsilon_i+1})} =
  A_{t(e_1^{\epsilon_i})}$ is called a $\G(A)$-path.
\end{defn}

\begin{defn}
  We denote by $\pi_1(\G(A),u)$ the group generated by $\G(A)$-paths
  based at the vertex $u$
  equipped with the obvious multiplication (i.e. concatenation and reduction) rules.
\end{defn}
As usual, if $A$ is connected the isomorphisms class of
$\pi_1(\G(A),u)$ doesn't depend on $u$.

\subsection{Induced splittings and $\mathcal{G}(A)$-graphs}

\begin{defn}
  Suppose that $G$ has a splitting $D$ as the fundamental group of a
  graph of groups and let $H$ be a subgroup of $G$. Then $G$ acts on a
  tree $T$ and $H$ acts on the minimal $H$-invariant subtree
  $T(H)\subset T$. Therefore $H$ also splits as a graph of groups.  We
  call this splitting of $H$ the \emph{induced splitting of $H$.}
\end{defn}

We now present the folding machinery developed in \cite{KMW-2005},
which is a more combinatorial version of
Stallings-Bestvina-Feighn-Dunwoody folding sequences. We will use it
to find induced splittings. This gives an alternative to normal
forms when dealing with fundamental groups of graphs of groups which
simplifies and unifies the arguments of Sections
\ref{sec:decomp-classification}, \ref{sec:2-nonab-classification} and
\ref{sec:1v-classification}.

\subsubsection{Basic definitions}
We follow \cite{KMW-2005}. 
\begin{defn}\label{defn:G(A)-graph} Let $\G(A)$ be a graph of
  groups. A $\G(A)$-graph $\B$ consists of an underlying graph $B$
  with the following data:\begin{itemize}
    \item A graph morphism $[.]:B \rightarrow A$
    \item For each $u \in VB$ there is a group $\B_u$ with $\B_u \leq
      A_{[u]}$, called a $\B$-vertex group. We give $u$ the \emph{label} $(\B_u,[u])$
    \item To each edge $e \in EB$ there are two associated elements
      $e_i \in A_{[i(f)]}$ and $e_t \in A_{[t(f)]}$. If we flip the
      orientation of $e$ we have the convention $(e^\mo)_i =
      (e_t)^\mo$. We give the edge $e$ the \emph{label}
      $(e_i,[e],e_t)$.
    \end{itemize}
\end{defn}

\begin{convention}We will usually denote $\G(A)$-graphs by $\B$ and
    will assume that the underlying graph of $\B$ is some graph $B$.
\end{convention}

\begin{defn} Let $\B$ be a $\G(A)$-graph and suppose that
  $e_1^{\epsilon_1},\ldots,e_n^{\epsilon_n}$; where $e_j \in EB,
\epsilon_j \in \{\pmo\}$; is an edge path of $B$. A sequence of the
form\[b_0,e_1^{\epsilon_1},b_1,e_2^{\epsilon_2},\ldots,e_n^{\epsilon_n},b_n\]
where $b_j \in \B_{t(e_j^{\epsilon_j})}$ is called a $\B$-path. To each
$\B$-path $p$ we associate a label\[\mu(p)=a_0[e_1]^{\epsilon_{1}} 
a_1 [e_2]^{\epsilon_{2}} \ldots [e_n]^{\epsilon_n} a_n\] where
$a_0=b_0(e_1^{\epsilon_1})_i, a_j=(e_j^{\epsilon_1})_t b_j
(e_{j+1}^{\epsilon_1})_i$ and $a_n=(e_n^{\epsilon_n})_t b_n$ which is
a $\G(A)$-path (see Definition \ref{defn:ga-path}.
\end{defn}

\begin{defn} Let $\B$ be a $\G(A)$-graph with a basepoint $u$. Then we
  define the subgroup $\pi_1(\B,u) \leq \pi_1(\G(A),[u])$ to be the
  subgroup generated by the $\mu(p)$ where $p$ is a $\B$-loop based at
  $u$.
\end{defn}

\begin{example}\label{eg:1}
  Let $G=\pi_1(\G(X),u)= A*_C E$. The underlying graph is
\[\xymatrix{u \bullet \ar@{-}[r]_e & \bullet v}\]
with $X_u = A, X_v=E$ and $X_e = C$. Let $g =a_1b_2a_3b_4a_5$ be a
word in normal form where $a_i \in A$ and $b_j \in E$. Then the
$\G(X)$-graph $\B$ given by:
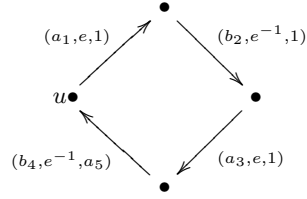
\[ \xymatrix{ & \bullet \ar[dr]^{(b_2,e^\mo,1)} & \\
  u \bullet\ar[ur]^{(a_1,e,1)} & &\bullet \ar[dl]^{(a_3,e,1)} \\ &
  \bullet \ar[ul]^{(b_4,e^\mo,a_5)}}\] whose $\B$-vertex groups are
all trivial is such that $\pi_1(\B,u)=\brakett{g}$
\end{example}

This example motivates a definition.
\begin{defn}\label{defn:gloop} Let $g =
  b_0,e_1^{\epsilon_1},b_1,e_2^{\epsilon_2},\ldots,e_n^{\epsilon_n},b_n$
  be some element of $\pi_1(\G(X),u)$. Then we call the based
  $\G(X)$-graph $\L(g;u)$ a \emph{$g$-loop} if $\L(g;u)$ consists of a
  cycle starting at $u$ whose edges have labels
  \[(b_0,e_1^{\epsilon_1},1),(b_1,e_2^{\epsilon_2},1),\ldots
  (b_{n-2},e_{n-1}^{\epsilon_{n-1}},1),(b_{n-1},e_n^{\epsilon_n},b_n)\]
\end{defn}

\begin{defn}
  $(\G(A),v_0)$ be be a graph of groups decomposition of $\FRS$. Let
  the \emph{$\x,\z$-wedge}, $\W(F,\x,\z;u)$, be the based
  $\G(A)$-graph formed from a vertex $v$ with label $(F,v_0)$ and
  attaching the loops $\L(\x;v_0)$ and $\L(\z;v_0)$.
\end{defn}

It is clear that $\pi_1(\W(F,\x,\z;u),u)= \brakett{F,\x,\z} =
\FRS$.

\subsubsection{Folding moves on $\G(A)$-graphs}\label{sec:graph-folds}
Let $\B$ be a $\G(A)$ graph, with underlying graph $B$. We now briefly
define the moves on $\B$ given in \cite{KMW-2005} that we will use, we
will sometimes replace an edge $e$ by $e^\mo$ to shorten the
descriptions:
\begin{itemize}
\item[A0:] {\bf Conjugation at $v$:} For some vertex $v$ of and some
  $g \in A_{[v]}$ do the following: replace $\B_v$ by $g\B_v g^\mo$,
  up to changing $e$ by $e^\mo$ we may assume that each edge $e$
  incident to $v$ are such that $i(e)=v$. Such an edge has label
  $(e_i,[e],e_t)$. Replace $e_i$ by $ge_i$. 
\item[A1:] {\bf Bass-Serre Move at $e$:} For some edge $e$, replace its label
  $(a,[e],b)$ by $(a i_e(c),[e],t_e(c^\mo)b)$ for some $c$ in $A_{[e]}$.
\item[A2:] {\bf Simple adjustment at $u$ on $e$:} For some vertex $u$
  and some edge $e$ such that w.l.o.g $i(e)=u$, we replace the label
  $(a,[e],b) \tr{~by~} (ga,[e],b)$ where $g \in \B_u$.
\item[F1:] {\bf Simple fold of $e_1$ and $e_2$ at the vertex $u$:} For
  a vertex $u$ and edges $e_1,e_2$ such that w.l.o.g.
  $i(e_1)=i(e_2)=u$ but $t(e_1)=v_1 \neq t(e_2)=v_2$ and $[v_1] =
  [v_2]$, if $e_1$ and $e_2$ have the same label, then make a new
  graph by identifying the edges $e_1$ and $e_2$.  The resulting edge
  has the same label as $e_1$ and the $\B$-vertex group associated to
  the result of the identification of $v_1$ and $v_2$ is
  $\brakett{\B_{v_1},\B_{v_2}}$.
\item[F4:] {\bf Double edge fold (or collapse) of $e_1$ and $e_2$ at
    the vertex $u$:} For edges $e_1,e_2$ such that w.l.o.g.
  $i(e_1)=i(e_2)=u$, $t(e_2)=t(e_2)=v$, and $[e_1]=[e_2]=f$ if they
  have labels $(a,f,b_1)$ and $(a,f,b_2)$ respectively. Then we can
  identify the edges $e_1$ and $e_2$, the resulting edges has label
  $(a,f,b_1)$ and the the group $\B_v$ is replaced by
  $\brakett{\B_v,b_1^\mo b_2}$.  \emph{We will also call such a fold a
    collapse from $u$ towards $v$}.
\end{itemize}
The moves F2 and F3 in \cite{KMW-2005} are analogous to F1 and F4,
respectively only they involve simple loops. However, because these
moves only show up implicitly in Section \ref{sec:1v-F-class}, we do not
describe them explicitly. We also introduce three new moves:
\begin{itemize}
\item [T1:] {\bf Transmission from $u$ to $v$ through $e$:} For an
  edge $e$ such that $i(e) = u$ and $t(e)=v$ with label
  $(a,[e],b)$. Let $g \in A_{[e]}$ be such that $a i_{[e]}(g)a^\mo = c
  \in \B_u$, then replace $\B_v$ by $\brakett{\B_v, b^\mo t_{[e]}(g)
    b}$. Transmissions are assumed to be \emph{proper}, i.e. they
  result in a change of the $\B$-vertex groups.
\item [L1:] {\bf Long range adjustment:} Perform a sequence of
  transmissions through edges $e_1,\ldots e_n$ followed by a simple
  adjustment at some vertex $v$ that changes the label of some edge
  $f$ but leaves unchanged the labels of the edges $e_1,\ldots
  e_n$. Finally replace all the modified $\B$-vertex groups by what
  they were before the sequence of transmissions (See Figure \ref{fig:L1}.) 
\item [S1:]{\bf Shaving move:} Suppose that $u$ is a vertex of valence
  1 such that $u=t(e)$ and $v=i(e)$, $e$ has label $(a,[e],b)$ and
  $\B_u = b^\mo (t_{[e]}(C)) b$, where $C \leq A_{[e]}$. Then collapse
  the edge $e$ to its endpoint $v$ and replace $\B_v$ by
  $\brakett{\B_v, a (i_{[e]}(C)) a^\mo}$.
\end{itemize}

\begin{convention}
Although formally applying a move to a $\G(A)$-graph $\B$ gives a new
graph $\B'$. Unless noted otherwise we will denote this new
$\G(A)$-graph as $\B$ as well.
\end{convention}

\begin{figure}
\centering
\begin{tikzpicture}
  [inner sep=0.15, worbit/.style={circle, draw=black,fill=black,minimum size
    = 0.15cm}, scale=1.5]
  \node (a) at (-0.75,0.75) [worbit]{}; 
  \node (b)  at (-0.75,-0.75) [worbit] {}; 
  \node (c) at (0,0) [worbit]{}; 
  \node (d)  at (1.2,0) [worbit] {}; 
  \draw (a)++(0,0.2) node{\small{$(G_1,v_1)$}};
  \draw (b)++(0,-0.2) node{\small{$(G_2,v_2)$}};
  \draw (c) ++(-1,0) node {\small{$(\langle G_3,g_1,g_2 \rangle,v_3)$}};
   \draw (0.1,0.5) node{\scriptsize{$(a_1,e_1,b_1)$}};
  \draw (0.1,-0.5 )node{\scriptsize{$(a_2,e_2,b_2)$}};
  \draw (0.6,0.2) node{\scriptsize{$(a_3,e_3,b_3)$}};
  \draw[-angle 45,very thick] (a) -- (c); 
  \draw[-angle 45, very thick] (b) -- (c);
  \draw[-angle 45] (c)-- (d);

  \node (e) at (2.25,0.75) [worbit]{}; 
  \node (f)  at (2.25,-0.75) [worbit] {}; 
  \node (g) at (3,0) [worbit]{}; 
  \node (h)  at (4.2,0) [worbit] {}; 
  \draw (e)++(0,0.2) node{\small{$(G_1,v_1)$}};
  \draw (f)++(0,-0.2) node{\small{$(G_2,v_2)$}};
  \draw (g) ++(-0.6,0) node {\small{$(G_3,v_3)$}};
  \draw (3.1,0.5) node{\scriptsize{$(a_1,e_1,b_1)$}};
  \draw (3.1,-0.5 )node{\scriptsize{$(a_2,e_2,b_2)$}};
  \draw (3.6,0.2) node{\scriptsize{$(g_3a_3,e_3,b_3)$}};
  \draw[-angle 45] (e) -- (g); 
  \draw[-angle 45] (f) -- (g);
  \draw[-angle 45] (g)--(h);
\end{tikzpicture}
\caption{An example of a L1 long range adjustment. First make T1
  transmissions through the thickened edges labeled $(a_1,e_1,b_1)$
  and $(a_2,e_2,b_2)$. These change the $\B$-vertex group $G_3$ to
  $\bk{G_3,g_1,g_2}$. We then perform an A2 simple adjustment which
  changes $(a_3,e_3,b_3)$ to $(g_3a_3,e_3,b_3)$ for some $g_3 \in
  \bk{G_3,g_1,g_2}$ (but maybe not in $G_3$.) Finally we change
  $\bk{G_3,g_1,g_2}$ back to $G_3$. The end result is the
  $\G(A)$-graph on the right.}\label{fig:L1}
\end{figure}
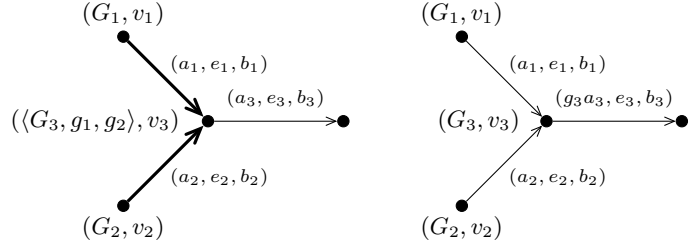

We regard the T1 transmission as the group $\B_u$ sending the
element $c$ to $\B_v$ through the edge $e$. In this paper, since all
the edge groups are finitely generated abelian, we can use finitely
many T1 transmission moves instead of the edge equalizing moves F5-F6
in \cite{KMW-2005}. The moves T1,L1, and S1 do not change the group
$\pi_1(\B,u)$. Note moreover that vertices $v$ of valence 1 with
$\B_v=\{1\}$ can be shaved off.

\subsubsection{The folding process}

\begin{defn} A $\G(A)$-graph such that it is impossible to apply any
  of the above moves other than A0-A2 is called \emph{folded}.
\end{defn}

This next important result is essentially a combination of Proposition
4.3, Lemma 4.16 and Proposition 4.15 of \cite{KMW-2005}.
\begin{thm}\label{thm:folded} \cite{KMW-2005} Applying the moves
  A0-A2, F1-F4, and T1 to $\B$ do not change $H=\pi_1(\B,u)$, moreover
  if $\B$ is folded, then the associated data (see Definition
  \ref{defn:G(A)-graph}) gives the graph of groups decomposition of $H$ induced
  by $H \leq \pi_1(\G(A))$
\end{thm}
 
This theorem implies the existence of a folding process. Consider the
three following classes of moves:
\begin{itemize}
\item {\bf Adjustment:} Apply a sequence of moves A0-A2, L1, and S1. 
\item{\bf Folding:} Apply moves F1 or F4.
\item{\bf Transmission:} Apply move T1.
\end{itemize}
First note that each folding decreases the number of edges in the
graph, and that adjustments (except for shavings) are essentially
reversible. In the folding process there is therefore a finite number
of foldings and between foldings there are adjustments and
transmissions. 

\section{When $\FRS$ is freely decomposable modulo $F$}\label{sec:decomp-classification}
This next proof is essentially the proof of Theorem 6.2 of
\cite{KMW-2005}.
\begin{prop}\label{prop:free-decomp-FRS}
Suppose that $\FRS=\F*H$, then $\F$ is generated by $F$ and
$\big(2 - \rk{(H)}\big)$ elements.
\end{prop}
\begin{proof} First note that the underlying graph of the splitting
  $\FRS=\F*H$ consists of an edge and two distinct vertices. Let
  $\G(A)$ denote this graph of groups and let $\B$ be any $\G(A)$-graph.
  Only the moves A0-A3,F1,and F4 can be applied.

  Take $\W$ to be the wedge $\W(F,\x,\z)$. Since $\pi_1(\W)=\FRS$ we
  have by Theorem \ref{thm:folded} that $\W$ can be brought to a graph
  with a single edge and two distinct vertices. The underlying graph
  of $\W$ has two cycles and $A$ has no cycles, which means that two
  F4 collapses must occur.  Moreover each collapse, maybe after
  applying F1 moves, either contributes a generator to $H$ or to
  $\F$. The result now follows.
\end{proof}

\begin{cor}\label{cor:free-decomp-FRS}
  If $\FRS$ is freely decomposable modulo $F$ then either it is one
  generated modulo $F$ or \[ \FRS = \left\{ \begin{array}{l}
      F*\brakett{x,y}\\
      F*_\bk{u}Ab(u,t)*\brakett{x}\\
      F*Ab(x,y)\\
\end{array}\right.\]
\end{cor}
\begin{proof}
Apply Proposition \ref{prop:free-decomp-FRS} and Theorems
\ref{thm:FGMRS} and \ref{thm:onevariable}.
\end{proof}

\section{When all the vertex groups of the JSJ of $\FRS$ except $\F$ are
  abelian}\label{sec:abelian-classification}
We consider the case where the JSJ of $\FRS$ has abelian vertex groups
but only one non-abelian vertex group $\F\geq F$. Before we continue
we need some extra machinery.

\subsection{Preliminaries}

\subsubsection{The (generalized) JSJ decomposition}\label{sec:JSJ}
As noted earlier $\FRS$ always has a generalized JSJ as given in Definition
\ref{defn:gen-JSJ}. We give some more details that will be necessary
to our work.

\begin{defn}
  Let $G$ act on a simplicial tree $T$ without inversions. An element
  or subgroup of $G$ is called \define{elliptic} if it fixes a point
  of $T$. Otherwise it is called \define{hyperbolic}. If $D_1$ and
  $D_2$ are two splittings of $G$ then $D_1$ is hyperbolic
  w.r.t. $D_2$ if an edge group of $D_1$ is hyperbolic w.r.t. the
  action of $G$ on the Bass-Serre tree associated to $D_2$.
\end{defn}

Consider now the following moves that can be made on a
graph of groups.

\begin{defn}[Moves on $\mathcal{G}(A)$]\label{defn:moves} 
  We have the following moves on a graph of groups $\mathcal{G}(A)$
  that do not change the fundamental group.
\begin{itemize}
\item \emph{Conjugate boundary monomorphisms:} Replace $i_e$ by
  $\gamma_g\circ i_e$ where $\gamma_g$ denotes conjugation by $g$
  and $g \in A_{i(e)}$.
\item \emph{Slide:} If there are edges $e,f$ such that $i_e(A_e) \geq
  i_f(A_f)$ then we change $A$ by redefining $i:EA \rightarrow VA$ so
  that $f \mapsto t(e)$ and replacing the homomorphisms $i_f$ by
  $t_e\circ i_e^{-1} \circ i_f$.
\item \emph{Folding:} If $i_e(A_e)\leq A \leq A_{i(e)}$, then
  replace $A_{t(e)}$ by $A_{t(e)}*_{t_{e}(A_e)}A$, replace $A_e$ by a
  copy of $A$ and change the boundary monomorphism accordingly. (We
  remark that the name ``folding'' comes from the effect of the move
  on the Bass-Serre tree, we could also call it ``edge group
  enlargement'' or ``vertex group expansion''.)
\item \emph{Collapse an edge $e$:} For some edge $e \in EA$, let $A
  \rightarrow A'$ be the quotient obtained by collapsing $e$ to a point
  $[e] \in A'$. We get a new graph of groups with underlying graph
  $A'$ as follows: if $i(e)\neq t(e)$ we set $A_{[e]}$ to be the free
  product with amalgamation $A_{i(e)}*_{A_e}A_{t(e)}$, if $i(e) =
  t(e)$ we set $A_{[e]}$ to be the HNN extension
  $A_{i(e)}*_{A_e}$. The boundary monomorphisms are the natural ones.
\end{itemize}
\end{defn}

\begin{defn}\label{defn:almost-reduced} \begin{itemize}
  \item A splitting $D$ is \emph{almost reduced} if vertex groups of
    vertices of valence one and two properly contain the images of
    edge groups, except possibly the vertex groups of vertices between
    two MQH subgroups that may coincide with one of the edge groups.
  \item A splitting $D$ of $\FRS$ is \emph{unfolded} if $D$ cannot
      be obtained from another splitting $D'$ via a folding move (See
      Definition \ref{defn:moves}).\end{itemize}
\end{defn}

\begin{defn}\label{defn:elementary-splitting}
  An \define{elementary splitting} is a splitting whose underlying
  graph has one edge.
\end{defn}

\begin{thm}[Proposition 2.15 of \cite{KM-JSJ}]\label{thm:jsj}  Suppose that $\FRS$
  is freely indecomposable modulo $F$. Then there exists an almost reduced unfolded cyclic splitting
  $D$ called the \emph{cyclic JSJ splitting of $\FRS$ modulo $F$} with
  the following properties:\begin{enumerate}
    \item[(1)] Every MQH subgroup of $\FRS$ can be conjugated into a vertex
    group in $D$; every QH subgroup of $\FRS$ can be conjugated into one
    of the MQH subgroups of $\FRS$; non-MQH [vertex] subgroups in $D$ are
    of two types: maximal abelian and non-abelian [rigid], every
    non-MQH vertex group in $D$ is elliptic in every cyclic splitting
    of $H$ modulo $F$.
    \item[(2)] If an elementary cyclic splitting $\FRS=A*_CB$ or $\FRS=A*_C$ is
    hyperbolic in another elementary cyclic splitting, then $C$ can be
    conjugated into some MQH subgroup.
    \item[(3)] Every elementary cyclic splitting $\FRS=A*_CB$ or $\FRS=A*_C$
    modulo $F$ which is elliptic with respect to any other elementary
    cyclic splitting modulo $F$ of $\FRS$ can be obtained from $D$ by a
    sequence of moves given in Definition \ref{defn:moves}.
    \item[(4)] If $D_1$ is another cyclic splitting of $\FRS$ modulo $F$ that
    has properties (1)-(2) then $D_1$ can be obtained from $D$ by a
    sequence of slidings, conjugations, and modifying boundary
    monomorphisms by conjugation (see Definition \ref{defn:moves}.)
  \end{enumerate}
\end{thm}

\begin{defn}\label{defn:acylindrical} An action of $G$ on a simplicial tree $T$
  is said to be $k$-\emph{acylindrical} if the the diameter of a
  subset of $T$ fixed by a non-trivial element of $G$ is at most
  $k$. A \emph{splitting} of $G$ is said to be $k$-acylindrical if the
  action of $G$ on the induced Bass-Serre tree is $k$-acylindrical.
\end{defn}

\begin{conv}\label{conv:abelian-elliptic}
  It is, for example, possible to write an amalgam $G*_uAb(u,t)$ as an
  HNN extension $\bk{G,t \mid t^\mo u t = u}$. \emph{We will always
    chose our JSJ so that non-cyclic abelian subgroups of $\FRS$ are
    elliptic.} This is necessary to ensure 2-acylindricity of the
  splitting, and in our situation this is always possible.
\end{conv}

\begin{defn}\label{defn:hair} Let $D$ be the JSJ of
  $\FRS$ modulo $F$. If $e$ is a an edge ending in a vertex $v$ of
  valence 1 in $A$, the graph underlying $D$, and $A_v$ is cyclic,
  then $e$ is called a \emph{hair}. Let $D'$ be splitting of $\FRS$
  obtained by collapsing hairs into the adjacent vertex groups. Then
  $D'$ is called the \emph{hairless JSJ of $\FRS$}.
\end{defn}

We note that since we require $D$ to be almost reduced, it is a simple
exercise involving the use of commutation transitivity to see that
after removing all the hairs of $D$,  the hairless splitting $D'$ will indeed not
have any hairs. We also note that passing to a hairless splitting
doesn't change the group of canonical (or modular) automorphisms (see
Section 2.15 of \cite{KM-JSJ} or Definition 5.2 of \cite{Sela-DiophI}
for details.)

\begin{conv}
Unless stated otherwise, we will always replace the JSJ by the
hairless JSJ.
\end{conv}

\subsubsection{Strict resolutions}\label{sec:resolutions}
Strict resolutions are an extremely useful tool for studying the
vertex groups of a JSJ. 

\begin{defn} An epimorphism $\rho:\FRS \rightarrow F_{R(S')}$ of fully
  residually $F$ groups is called (weakly) \emph{strict} if it satisfies the
  following conditions on the (generalized) JSJ modulo $F$.
  \begin{enumerate}
    \item For each abelian vertex group $A$, $\rho$ is injective on
    the subgroup $A_1 \leq A$ generated by the boundary subgroups in
    $A$.
    \item $\rho$ is injective on edge groups.
    \item The images of QH subgroups are non-abelian.
    \item For every rigid subgroup (as defined in (1) of Theorem
      \ref{thm:jsj}) $R$, $\rho$ is injective on $R$.
    \item Distinct factors of the Grushko decomposition of $\FRS$
      modulo $F$ are mapped onto distinct free factors a free
      decomposition of $F_{R(S')}$ modulo $F$.
\end{enumerate}
\end{defn}

\begin{conv}
Weakly strict differs from strict as defined in \cite{Sela-DiophI}
only in item 3. We have simplified the definition for the convenience
of the reader. Throughout this paper we shall say \emph{strict}
  instead of \emph{weakly strict.}
\end{conv}

\begin{defn}
  A strict resolution of $\FRS$ \[\mathcal{R}:\xymatrix{F_{R(S_0)}=\FRS \ar[r]_{\pi_1}
    &\ldots \ar[r]_{\pi_p}& F_{R(S_p)}\ar[r]_{\pi_{p+1}}& F*F(Y)}\] is a
  sequence of proper epimorphisms of fully residually $F$
  groups such that all the epimorphism are 
strict.
\end{defn}

\begin{thm}\label{thm:strict-discriminate} If $\FRS$ is fully
  residually $F$ then it admits a strict resolution $\R$. 
\end{thm}

As formulated Theorem \ref{thm:strict-discriminate}, is an easy
Corollary of the definitions of the canonical (or modular)
automorphisms and the fact that $\Hom(\FRS,F)$ can be encoded in a
finite Hom (also called a Makanin-Razborov) diagram (see Theorem 5.12
of \cite{Sela-DiophI} or Theorem 11.2 of \cite{KM-JSJ} for details.)

\subsubsection{Iterated centralizer extensions}\label{sec:ice}

\begin{defn}\label{defn:ICE}
  A (rank $n$) centralizer extension of $F$ is an amalgam $F*_\bk{u}A_u$
  where $u \in F$ is malnormal and $A_u$ is free abelian (of rank
  $n+1$.) $G$ is an iterated centralizer extension of $F$ if
  either: \begin{itemize}
    \item $G$ is a centralizer extension of $F$; or
    \item $G = H*_\bk{w}A_w$ where $H$ is an iterated centralizer
      extension, the centralizer of $w$ in $H$ is $\bk{w}$ and $A_w$ is
      free abelian.
    \end{itemize}
    We say it is \emph{finite} if it was obtained from $F$ by finitely
    many centralizer extensions.
\end{defn}

\begin{thm}[Theorem 4 of \cite{KM-IrredII}]\label{thm:ICE}
  $G$ is finitely generated and fully residually $F$ if and only if
  it embeds in a finite iterated centralizer extension of $F$.
\end{thm}

\begin{cor}\label{cor:get-split}
  Any subgroup $F < \hat{F} \leq \FRS$ has a $(\leq \Z)$-splitting
  $D$ modulo $F$. Moreover any element $\beta'  \in \hat{F}$ that is
  conjugate in $\FRS$ to some $\beta \in F$ is elliptic in this
  splitting of $\hat{F}$.
\end{cor}
\begin{proof}
  Let $\FRS \leq G$ be the embedding of $\FRS$ into an iterated
  centralizer extension of $F$. The cyclic JSJ of $G$ modulo $F$ is
  very simple: it is a star of groups with a vertex group containing
  $F$ at its center and all the other vertex groups are free
  abelian. The central vertex group is itself either $F$ or an
  iterated centralizer extension of $F$.

  If $\hat{F}$ is elliptic in the JSJ of $G$ then we can replace $G$ by its
  central vertex group. Since $\hat{F} \neq F$ eventually there some
  iterated centralizer extension $\hat{F} \leq H \leq G$ in which $\hat{F}$ has
  a non-trivial induced splitting $D$. 

  \emph{Claim: $\beta'$ is conjugate to $\beta$ in $H$}. Suppose
  towards a contradiction that this was not the case. Obviously
  $\beta$ is conjugate to $\beta'$ in $G$ so if $H=G$ we have a
  contradiction. We have $G = G'*_\bk{u}A_u$ with $A_u$ free abelian and the
  centralizer of $\bk{u}$ in $G'$ cyclic. Now by hypothesis there is
  some element $W(G',A_u) \in G$ such that $W(G',A_u) \beta
  W(G',A_u)^\mo = \beta'$. Looking at normal forms, we immediately see
  that if such a product is to lie in $G'$ we must have that $\beta$
  and $\beta'$ are conjugate to $u$ in $G'$, so $\beta,\beta'$ are
  conjugate in $G'$.

  We repeat replacing $G$ by $G'$. Continuing in this fashion, we
  eventually get that $\beta,\beta'$ are conjugate in $H$ --
  contradiction. \emph{The claim is therefore proved}.

  It therefore follows that $\beta'$ is conjugate to $\beta$ in $H$ so
  $\beta'$ must be elliptic in the induced splitting of $\hat{F}$ modulo
  $F$, which has either cyclic or trivial edge groups.
\end{proof}

\subsubsection{The first Betti number}\label{sec:complexity}

The following useful fact is obvious from a relative
presentation:

\begin{lem}\label{lem:b1-eoc} Let $H<G$ be a rank $n$ centralizer
  extension of $H$ (see Definition \ref{defn:ICE},) then $b_1(G)=b_1(H)+n$.
\end{lem}

\begin{lem}\label{lem:CC} The subgroup $F\leq \FRS$ has property \emph{CC}
  (conjugacy closed), that is to say for $f,f' \in F$ \[\exists g \in \FRS \tr{~such
  that~} f^g = f' \Rightarrow \exists k \in F \tr{~such that~}
  f^k = f'\]
\end{lem}
\begin{proof} Let $f,f'$, and $g$ be as in the statement of the
  Lemma. Let $r:\FRS \rightarrow F$, be a retraction. Then $k=r(g) \in
  F$ has the desired property.
\end{proof}

\begin{proof}[proof of Proposition \ref{prop:b1-complexity}] Suppose towards
  a contradiction that $G\neq F$ but $b_1(G)=N$. Then, by being fully
  residually $F$, there is a retraction $G\rightarrow F$. It follows
  that $b_1(G)\geq N$. If $G$ is freely decomposable modulo $F$, then
  one of its free factors retracts onto $F$ and any other free factor
  maps onto an infinite cyclic group so $b_1(G) > N$ -- contradiction.

  It follows that $G$ is freely indecomposable and since $G \neq F$,
  $G$ has $D$, a non-trivial JSJ decomposition.  Let $F\leq \F \leq G$
  be the vertex group containing $F$, obviously $\F$ is also fully
  residually $F$. By formula (\ref{eqn:b1-bound}), $b_1(G)\geq
  b_1(\F)$ and if $D$ has more than one vertex group then the
  inequality is proper which forces $b_1(G)>N$ -- contradiction. $D$
  is therefore a bouquet of circles with a single
  vertex group $\F$. By Lemma \ref{lem:CC} if $\F=F$, then the stable
  letters of the splitting $D$ in fact extend centralizers of elements
  of $F$, so by Lemma \ref{lem:b1-eoc}, $b_1(G)>b_1(F)$ --
  contradiction.

  It follows that we cannot have $\F = F$. We therefore look JSJ of
  $\F$. Again it must have a unique vertex group
  $\F^1$. Since $\F^1 \neq F$, it must have an essential cyclic
  splitting. Since $\uhdf(G)$ is finite we have a terminating
  sequence \[\F > \F^1 >\ldots \F^r > \F^{r+1}=F\] where $\F^{i+1}$ is
  the unique vertex group of the JSJ of $\F^i$. Now, by assumption, $N=b_1(\F)\geq b_1(\F^1) \ldots \geq b_1(\F^{r+1})=N$ but
  $\F^r$ has a splitting $D^r$ that is a bouquet of
  circles with vertex group $F$, so it is a centralizer extension of
  $F$. It follows that $b_1(\F^r) > b_1(F)$ -- contradiction.
\end{proof}

\subsubsection{Splittings with two cycles}\label{sec:two-stable-letters}

Suppose $\FRS=\bra F, \x, \z\kett$ can be collapsed to a graph of
groups modulo $F$ with one vertex and two edges, i.e.
\begin{equation}\label{eqn:two-stable-letter-splitting} \FRS =
s  \relpres{\bk{\hat{F},t,s \mid A^t = A', B^s = B' }}{A,A',B,B' \leq \hat{F}}\end{equation} 
where $F \leq \hat{F}$. 

\begin{defn} For a generating set $X$ and a word $W=W(X)$ in $X$, for
  a letter $x \in X$ we denote the exponent sum of $x$ in $W$ by
  $\sigma_x(W)$. \end{defn}

\begin{defn} Let $\bk{X,t \mid R}$ be some relative presentation for
  $G$ where $t$ is a stable letter. Any $g \in G$ can be expressed as
  a word $g=W(X,t)$. We define the \emph{exponent sum} $\sigma_t(g)$ of
  $t$ in $g$ as \[\sigma_t(g)= \sigma_t(W(X,t)).\] By Britton's lemma
  this is well defined.\end{defn}

This next lemma follows immediately from abelianizing.

\begin{lem}\label{lem:double-hnn-exposum}
Suppose $\FRS$ split as (\ref{eqn:two-stable-letter-splitting}), then
if some word $W(F,\x,\z)$ lies in $\hat{F}$ then it must have exponent sum
0 in both $\x$ and $\z$.
\end{lem}

\begin{defn}\label{defn:exposed}
  Let $\FRSP$ be a fully residually $F$ group and let $A$ be a
  subgroup of an abelian abelian vertex group $\hat{A}$ of
  $\FRSP$'s JSJ. A direct summand $A'\leq A$, where $A = A' \oplus A''$,
  that does not intersect the images of the edge groups incident to
  $\hat{A}$ is said to be \emph{exposed}.
\end{defn}

\begin{lem}\label{lem:abelian-exposed}
  Let $A \leq \FRS$ be a non-cyclic abelian subgroup, and
  let \[\mathcal{R}:\xymatrix{F_{R(S_0)}=\FRS \ar[r]_{\pi_1} &\ldots
    \ar[r]_{\pi_p}& F_{R(S_p)}\ar[r]_{\pi_{p+1}}& F*F(Y)}\] be a
  strict resolution of $\FRS$. $A$ must eventually be mapped
  monomorphically inside an abelian vertex group $\hat{A}$ of the
  (generalized) JSJ of some quotient $F_{R(S_i)}$ occuring in $\R$ and
  the isomorphic image of $A$ should have an exposed direct summand.
\end{lem}

\begin{proof}
  By Convention \ref{conv:abelian-elliptic} non-cyclic abelian groups
  are always elliptic in a JSJ. Suppose first that throughout $\R$ the
  subgroup $A$ is always mapped inside a rigid vertex group, or inside
  the subgroup generated by the boundary subgroups of some abelian
  vertex group. Since the terminal group of a strict resolution is always
  free non-abelian, $A$ is eventually mapped to a cyclic
  group. This means that some strict quotient was not injective on a
  rigid vertex group or a subgroup of an abelian vertex group
  generated by the boundary subgroups -- contradiction.

  It therefore follows that at some point $A$ is mapped inside
  some abelian vertex group $\hat{A}$ and is not contained in the
  subgroup generated by the incident edge groups. The result now follows.
\end{proof}

\begin{lem}\label{lem:exp-sum-zero-abelian}
  Suppose $A\leq\FRS$ is a non-cyclic abelian subgroup of $\FRS$. Then
  $A$ cannot be generated by words $W_i(F,\x,\z)$ such that for each
  $W_i$ the exponent sums in $\x$ and $\z$ are zero.
\end{lem}
\begin{proof} Let $\R$ be a strict resolution with fully residually
  $F$ groups $\{\FRSi\}$. By Lemma \ref{lem:abelian-exposed}
  in some $\FRSi$ $A$ is mapped monomorphically into an abelian vertex
  group $\hat{A}$ and has an exposed cyclic summand $\bk{r}$. Let
  $\hat{A'}$ be the subgroup of $\hat{A}$ generated by the
  incident edge groups and let $\bk{\eta}\leq \hat{A}$ be the
  maximal cyclic subgroup containing $\bk{r}$, so that $r =
  \eta^n$. Then we can extend the projection $\hat{A} \rightarrow
  \hat{A'} \oplus \bk{\eta}$ to $\FRSj$. The resulting quotient
  can be seen as an HNN extension:
  \[ \ol{\FRSi} = \bk{G,\eta \mid \forall a \in \hat{A'}, [\eta,a]=1}\] On one
  hand $\x$ and $\z$ are sent to elements with normal forms
  $\x'(G,\eta),\z'(G,\eta)$ on the other hand, $\eta^n$ is in the
  image of $A$ and by hypothesis we can write
  $r=R(F,\x'(G,\eta),\z'(G,\eta))$ where $R$ has exponent sum
  zero in $\x'(G,\eta)$ and $\z'(G,\eta)$, but 
  $R(F,\x'(G,\eta),\z'(G,\eta))=r = \eta^n, n \neq 0$ must have exponent sum
  zero in $\eta$ which is a contradiction.
\end{proof}

\begin{cor}\label{cor:no-ab} If $\FRS$ has a cyclic splitting
  modulo $F$ with two cycles, then none of the vertex groups of this
  splitting can contain non-cyclic abelian subgroups.
\end{cor}
\begin{proof}
  By Lemma \ref{lem:double-hnn-exposum} any element conjugable into a
  vertex group of the splitting must be expressible by words in
  $F,\x,\y$ that have exponent sum zero in $\x$ and $\y$, Lemma
  \ref{lem:exp-sum-zero-abelian} now gives a contradiction.
\end{proof}

\subsection{The possible JSJs when all the vertex groups of the JSJ of
$\FRS$ except $\F$ are abelian}\label{sec:only-abelian}
In this section we prove that the JSJs given in Proposition
\ref{prop:abelian} are the only possibilities.

\begin{lem}\label{lem:abelian-good-enough} If the cyclic JSJ
  decomposition of $\FRS$ modulo $F$ contains only one non-abelian
  vertex group then the JSJ is given by the graph of groups $\G(X)$
  where $X$ is one of the following:
  \[\left.\begin{array}{cc}
      
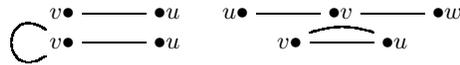
\xymatrix{ v\bullet \ar@{-}[r] & \bullet u}&
      \xymatrix{ u\bullet \ar@{-}[r] & \bullet v \ar@{-}[r] & \bullet
        w}\\
      \xymatrix{ v\bullet \ar@{-}[r] \ar@(ul,dl)@{-} & \bullet u}&
      \xymatrix{ v\bullet \ar@{-}[r] \ar@{-}@/^/[r] & \bullet u}\\
      \end{array}\right.\] with $\F \leq X_v$ and $X_w,X_u$ abelian.
\end{lem}

\begin{proof}
  By CSA and commutation transitivity $\G(X)$ cannot have subgraphs of
  groups \[\xymatrix{ u\bullet \ar@{-}[r] & \bullet w} \tr{~or~}
  \xymatrix{ u\bullet \ar@(ur,dr)@{-}} \] with $X_u,X_w$ non-cyclic
  abelian. Each abelian vertex group $A$ contributes $\rk(A)-1$ to
  $b_1(\FRS)$, so by Proposition \ref{prop:b1-complexity} there are at
  most two of them of rank 2, or one of them of rank at most 3. By
  Corollary \ref{cor:no-ab} the resulting underlying graph of groups
  cannot have two cycles. So far the only possibilities are the graphs
  of groups given in the statement of the lemma
  and \[\xymatrix{u\bullet \ar@{-}[r] & v \bullet \ar@{-}@/^/[r]
    \ar@{-}[r] & \bullet w}\] with $\F \leq X_v$ and $X_w,X_u$
  abelian. But note that $X_u$ must have rank 2 and since we can
  rewrite $G*_\bk{\alpha}A$ as $\bk{G,t | \alpha^t = \alpha}$ if $A$
  is free abelian of rank 2. We can again apply Corollary
  \ref{cor:no-ab} to get a contradiction to the fact that $X_w$ is
  non-cyclic abelian.
\end{proof}

\begin{prop}\label{prop:simply-connected-abelian}
  If $\FRS$ is as in the statement of Lemma
  \ref{lem:abelian-good-enough} then it cannot have the JSJ with
  underlying graph\[X = \xymatrix{ v\bullet \ar@{-}[r] \ar@{-}@/^/[r]
    & \bullet u}\]
  
\end{prop}

\begin{proof} Suppose towards a contradiction that $\FRS$ had the
  JSJ\[\relpres{\bk{\F,A,t \mid \alpha^t = \beta}}{ \alpha \in \F,
    \beta \in A, \bk{\delta} = \F \cap A} \] with $F \leq \F$ and
  $A$ abelian. Let $A' = \bk{\beta,\delta}\leq A$. \emph{Claim:} $F
  \neq \F$. Suppose this were not the case, then $\alpha$ and $\delta$
  lie in $F$ but to discriminate $A'$, either $\beta$ or $\delta$ must
  be sent to arbitrarily high powers via retractions $\FRS \rightarrow
  F$, which is impossible since $\beta$ is conjugate to $\alpha$ and
  $\delta \in F$.
    
  Since $\F\neq F$, $\F$ must have a $(\leq \Z)$-splitting modulo $F$,
  but because our splitting of $\FRS$ is a JSJ $\alpha$ or $\delta$
  must be hyperbolic in any $(\leq \Z)$-splitting of $\F$. We apply
  Lemma \ref{lem:abelian-exposed} to $A' \leq \FRS$. Let $\FRSi$ be
  the quotient in the strict resolution where $A'$ has an exposed
  summand. Since in the initial segment $\xymatrix{\FRS \ar[r]_{\pi_1}
    &\ldots \ar[r]_{\pi_p}& F_{R(S_i)}}$ of the strict resolution $A'$
  was always elliptic, so were the elements $\alpha,\delta$ which
  means $\F$ never split and hence was always mapped monomorphically.
    
  Let $\hat{A}$ be the abelian vertex group of $\FRSi$ containing $A'$
  and let $E\leq \hat{A}$ be the subgroup generated by incident edge
  groups. Since $A'$ has an exposed summand, w.l.o.g. $\delta \not\in
  E$ which means that in the Bass-Serre tree $T$ for the JSJ of
  $\FRSi$, the minimal invariant subtree $T(\bk{\delta})$ consists of
  a vertex whose stabilizer is abelian. $T(\bk{\beta})$ on the other
  hand consists of a vertex whose stabilizer is non-abelian, but
  $\delta \in \F$ -- contradiction.
\end{proof}

\begin{prop}\label{prop:abelian-classification}
If we have the JSJs \begin{itemize}
\item $\FRS=\F*_\bk{\alpha} A_1$ and $\rk(A_1)\geq 3$, or 
\item $\FRS=A_2*_\bk{\beta}\F*_\bk{\gamma} A_3$,
\end{itemize}
  with $A_1,A_2,A_3$ free abelian then $\F=F$.
\end{prop}

\begin{proof}
  This follows immediately by estimating $b_1(\FRS)$ and applying
  Proposition \ref{prop:b1-complexity}.
\end{proof}

The proof that $\F$ is two-generated modulo $F$ is deferred to Section
\ref{sec:abelian-2-gen-mod-F}.

\section{When the JSJ of $\FRS$ has at least two non-abelian vertex
  groups}\label{sec:2-nonab-classification}
\subsection{Preliminaries}
The approach here is to see what kind of graphs of groups we
can obtain as images of $F*\bk{x,y}$. To make the problem tractable we
first consider a coarser splitting. It will turn out that this
is an effective way to get started.
\subsubsection{Maximal abelian collapse}\label{sec:collapses}
Suppose that $D = (\G(X),\FRS)$ the JSJ of $\FRS$ contains at least two non-abelian
vertex groups. Take $D$ and do the following (see Definition
\ref{defn:moves}):\begin{itemize}
\item[(i)] If any boundary subgroup $\brakett{\alpha}=i_e(X_e)$ is a
  proper subgroup of a maximal abelian subgroup $A$, then do a folding
  move (as in Definition \ref{defn:moves}) where we replace $X_e$ by a
  copy of $A$.
\item[(ii)] Ensuring that the resulting graph of groups always has at least
  two non-abelian vertex groups, perform sliding and collapsing moves
  until it is no longer possible to decrease the number of vertices or
  edges
\end{itemize}

In the end the resulting graph of groups $\G(X)$ will have one of
three possible forms:
\begin{equation}\label{eqn:reg-norm-split}\xymatrix{\Fh
    \ar@{-}[r] & H},~ \xymatrix{\Fh \ar@{-}[r] \ar@{-}@/^/[r] &
    H},~\tr{or}~ \xymatrix{\Fh \ar@{-}[r] \ar@{-}@/^/[r] \ar@{-}@/_/[r] &
    H}\end{equation}
where the vertex group $\Fh$ contains $F$ and boundary subgroups are
maximal abelian in their vertex groups. 

\begin{defn}We will call such a splitting a \emph{maximal abelian
  collapse of $\FRS$}.
\end{defn}

\begin{lem}
  The maximal abelian collapse of $\FRS$ is a 1-acylindrical splitting.
\end{lem}
\begin{proof}
  On one hand by item (ii) of our construction, distinct edges have
  distinct boundary subgroups. On the other hand we also have that
  they are maximal abelian in their vertex groups, and maximal abelian
  subgroups are malnormal in fully residually free groups. The result
  now follows.
\end{proof}

Although we may have sacrificed some information in passing to a
maximal abelian collapse, we now have a 1-acylindrical splitting. This
will enable us to use the very useful Lemmas \ref{lem:make-elliptic}
and \ref{lem:long-range-crit}.

\subsubsection{Balancing folds and adjoining roots}

It may happen that the image of an edge group is not maximal cyclic in
one of the adjacent vertex groups.

\begin{defn}\label{defn:balanced}
  Let $\G(X)$ be a $(\leq \Z)$-graph of groups. An edge group $X_e$ is
  said to be \emph{balanced} if its images are are maximal cyclic in
  the vertex groups. A graph of group is called \emph{balanced} if all
  its edge groups are balanced.
\end{defn}

The main technical advantage of having a balanced $(\leq\Z)$-graph of
groups is that if all the edge groups are maximal abelian (i.e. they
don't lie in non-cyclic abelian subgroups) \emph{then our graph of
  groups is 1-acylindrical.}

By commutation transitivity one of the image of an edge group must
be maximal in the vertex groups. Our graphs of groups aren't always
balanced, however we may do the following.

\begin{defn}\label{defn:balancing-fold}
  Let $X_e$ be a non-balanced edge group of a $(\leq \Z)$-graph of
  group $\G(X)$, then the folding move (as in Definition
  \ref{defn:moves}) which enlarges $X_e$ so that its images in the
  adjacent vertex groups is maximal cyclic is called a \emph{balancing
    fold.}
\end{defn}

The result of a balancing fold on the ``enlarged'' vertex group is the
adjunction of a proper root. We will want to make balancing folds, but
we also want to keep the rank of the vertex groups under control. In
\cite{Weidmann-adjoin-root} Weidmann proves the following:

\begin{thm}[Theorem 1 of \cite{Weidmann-adjoin-root}]\label{thm:weidmann-add-a-root}
  Let $G$ be a group, $g \in G$ be an element of order $n \in \mathbb{N} \cup
  \{\infty\}$ and $k\geq 2$ an integer. Then\[
  \rk{~G*_\bk{g=z^k}\bk{z \mid z^{nk}}} \geq \rk{~G}.
  \]
\end{thm}

We will need the following variant of this result that isn't an
immediate corollary.

\begin{lem}\label{lem:adjoin-a-root}
  Let $\FRS$ be a finitely generated fully residually $F$ group and
  let \[\widehat{\FRS} = \FRS*_\bk{z} \bk{\sqrt[n]{z}},\] where
  $(\sqrt[n]{z})^n = z$, be a fully residually $F$ quotient of of
  $F*\bk{x_1,\ldots,x_m}$. Then $\FRS$ is also a fully residually $F$ quotient of of
  $F*\bk{x_1,\ldots,x_m}$. 
\end{lem}

\begin{proof}[sketch]
  The argument used to prove Theorem 1 of \cite{Weidmann-adjoin-root}
  is in fact perfectly suitable for our purposes. Let $\G(X)$ be the
  graph of groups for $\FRS*_\bk{z} \bk{\sqrt[n]{z}}$ and let $\G(X')$
  be the graph of groups for $\FRS*_\bk{z} \bk{z} = \FRS$. We start a
  folding sequence for $\widehat{\FRS}$ with $\B_0$ the $\G(X)$-graph
  that consists of a vertex $u$ with $\B_{0u} = F$ and $\x_i$-loops
  $\mathcal{L}(\x_i;u)$, i.e. the graph underlying $\B_0$ is a bouquet
  of $m$-circles. After this initial setup the arguments of the proof
  Theorem \ref{thm:weidmann-add-a-root}, that consider different types
  of folds in the folding sequence, go through and we find that $\FRS$
  is also generated by $F$ and $m$ other elements.
\end{proof}

\subsubsection{Weidmann-Nielsen normalization for groups acting on trees}
We present some of the techniques developed by Weidmann in
\cite{Weidmann-2002}. Let $G$ act on a simplicial tree $T$.

\begin{defn}\label{defn:weidmann-nielsen}
  Let $M\subset G$ be partitioned as \[ M = S_1 \sqcup \ldots \sqcup S_p
  \sqcup \Haches.\] We say that $M$ has a \emph{marking}
  $(S_1,\ldots,S_p;\Haches)$. The elementary Weidmann-Nielsen transformations
  on \emph{marked sets} are:
\begin{itemize}
\item[WN1:] Replace some $S_i$ by $g^\mo S_ig$ where $g\in M - S_i$.
\item[WN2:] Replace some element $h_i \in \Haches$ by $g_1h_ig_2$ where
  $g_1,g_2 \in M -\{h_i\}$
\end{itemize}
\end{defn}

\begin{defn}\label{defn:T_F}
For a subgroup $K \leq G$ we denote the minimal $K$-invariant
subtree $T(K)$ and we denote:
\[T_{\bk{S_i}} = \{x \in T \mid \exists g \in \bra S_i \kett
\setminus \{1\} \tr{~such that~}
gx=x \} \cup T(\bk{S_i})\] 
\end{defn}

 We now formulate the main results
in \cite{Weidmann-2002}.

\begin{thm}\label{thm:weidmann-nielsen} Let $M$ be a set with markings
  $(S_1,\ldots,S_p;\Haches)$. Then either \[ \bra M \kett = \bra S_1
  \kett * \ldots *\bra S_p \kett * F(\Haches) \] or by successively
  applying transformations WN1 and WN2 we can bring
  $(S_1,\ldots,S_p;\Haches)$ to a \emph{normalized} marked set
  \[\tilde{M}=(\tilde{S_1},\ldots,\tilde{S_p};\{\tilde{h_1},\ldots,\tilde{h_s}\})\]
  such that one of the following must hold:
  \begin{enumerate}
  \item $T_\bk{\tilde{S_i}} \cap T_\bk{\tilde{S_j}} \neq \emptyset$, for some $i
    \neq j$.
  \item $\exists \tilde{h_i} \in \{\tilde{h_1},\ldots,\tilde{h_s}\},
    \tilde{h_i}T_\bk{\tilde{S_i}} \cap T_\bk{\tilde{S_i}} \neq
    \emptyset$
  \item There is some $\tilde{h_i} \in \{\tilde{h_1},\ldots,\tilde{h_s}\}$ that fixes a point of $T$.
  \end{enumerate}
\end{thm}

Notice that in passing from a marked set $M$ to a normalized marked
set $\tilde{M}$ as in Theorem \ref{thm:weidmann-nielsen} the subgroups
$\bk{S_i}$ and $\bk{\tilde{S_i}}$ differ only in that there is some $g
\in G$ such that $\bk{S_i} = g^\mo \bk{S_i}g$. From this it is not
hard to see that we can chose our sequence of transformations WN1 and
WN2 so that one of the subsets $S_j$ in $M$ remains invariant.

\begin{convention}
  Suppose $F \leq \bk{S_i}$ . When we apply Weidmann-Nielsen
  transformations to a marked generating set of
  $(S_1,\ldots,S_p;\Haches)$ of $\FRS$ we want to make sure that the
  elements of $F$ remain fixed pointwise. We therefore do not use any
  moves of type WN1 that will change $S_1$. This restriction doesn't
  alter the applicability of Theorem \ref{thm:weidmann-nielsen}.
\end{convention}

\begin{defn}\label{defn:W-N-normalization}
  \define{Weidmann-Nielsen normalization} is the process of using
  moves WN1 and WN2 to bring a marked generating set to a
  \emph{normalized} generating set as in Theorem
  \ref{thm:weidmann-nielsen} .
\end{defn}

\subsubsection{Elements with small translation lengths}

We always act on a tree $T$ from the left i.e. for all $g,h \in \FRS$ and
for all $v \in T$ we have \[ ghv=g(hv)\] it follows that for any point
$v \in T$, and for any $g \in \FRS$ we will have $\stab(gv) = g(\stab(v))
g^{-1} = {}^g\stab(x)$. Recall the definition of $T_F$ given
in Definition \ref{defn:T_F}.

\begin{lem}\label{lem:tf-transitive}
  Let $\FRS$ act on a simplicial tree $T$ with edge stabilizers
  maximal abelian in their vertex groups and with $F$ elliptic. If
  there is some edge $e \subset T_F$ and some $g \in \FRS$ such that
  $ge \subset T_F$, then there is some $f \in F$ such that $f^\mo ge = e$.
\end{lem}
\begin{proof}
  Let $\bk{\gamma} = stab(e) \cap F$ and let $\bk{\beta} =
  \stab(ge)\cap F = {}^g stab(e) \cap F$. Since edge stabilizers are
  maximal abelian to prove the claim it is enough to find some $f \in
  F$ such that $[{}^f\gamma,\beta]=1$ as this would imply ${}^fstab(e)
  = stab(ge)$. By hypothesis $[{}^g \gamma,\beta]=1$. Let $\psi:\FRS
  \rightarrow F$ be a retraction, then $[{}^{\psi(g)}\gamma,\beta]=1$,
  so $f=\psi(g)$ is the desired element.
\end{proof}

If the action of $\FRS$ on its Bass-Serre tree $T$ is as in Lemma
\ref{lem:tf-transitive} then the action is 1-acylindrical. Suppose $F$
is elliptic and let $\fix(F)=v_0$, suppose also that $\rho$ is
hyperbolic. Of particular interest to us is the situation where $T_F
\cap \rho T_F \neq \emptyset$, as in item 2. of Theorem
\ref{thm:weidmann-nielsen}.

We will focus on the situation where $d(v_0,\rho v_0) = 2$ and where
the path $[v_0,\rho v_0]$ intersects two $\FRS$-vertex orbits. Let $w_0$
be a vertex such that $d(w_0,v_0)=1$ and suppose that $H = stab(w_0)
\cap F \neq \{1\}$. Let $w_1$ be the vertex in $[v_0,\rho v_0]$ that is an
$\FRS$-translate of $w_0$. Consider the two relative sub-presentations
(i.e. they are ``subgraphs of groups of'' $\FRS$.)

\begin{equation}\label{eqn:make-elliptic-splittings}
  \relpres{\bk{\Fh,H,t \mid A^t = B}}{A\leq \Fh, B\leq H, C = \Fh \cap H}
  \tr{~or~} \Fh*_CH
\end{equation}

with $F\leq \Fh = \stab(v_0)$. 

\begin{lem}\label{lem:make-elliptic} Let the action of $\FRS$ on $T$
  be as in the statement of Lemma \ref{lem:tf-transitive}. Let $v_0 =
  \fix(F)$ and suppose that there is some $\rho$ such that $T_F \cap
  \rho T_F \neq \emptyset$, $d(v_0,\rho v_0) =2$, and the path
  $[v_0,\rho v_0]$ intersects two $\FRS$-vertex orbits. Let $H =
  stab(w_0)$ be as described above. Using the notation introduced in
  (\ref{eqn:make-elliptic-splittings}) then there exist $f_1, f_2 \in
  F$ such that either\begin{enumerate}
  \item the path $[v_0,\rho v_0]$ intersects one $\FRS$-edge orbit and $f_2
    \rho f_1 = h \in H$; or
  \item the path $[v_0,\rho v_0]$ contains two $\FRS$-edge orbits and $f_2 \rho
    f_1 = th$ for some $h \in H$; which could be assumed to be the
    stable letter $t$ if we change the presentation by conjugating a
    boundary monomorphism.
\end{enumerate}\end{lem}
\begin{proof}
  We consider case 1. By hypothesis $T_F$ cannot be a point, $w_0 \in
  T_F$, and $H \cap F \neq \{1\}$. Consider Figure
  \ref{diagram1}. Then we must have \[\rho = a_2b_1a_1; a_i \in \Fh,
  b_1 \in H\] with $w_1 = a_2 w_0$. Now $a_2 \te \in T_F$ so by Lemma
  \ref{lem:tf-transitive} there exists an $f_2 \in F$ such that
  $f_2a_2 \te = \te$ so that $f_2a_2 \in stab(\te) = H \cap \Fh$. So up
  to replacing $b_1$ by $f_2a_2b_1 \in H$ we may now assume that $\rho
  = b_1a_1$ and that $w_0 \in [v_0,\rho v_0]$.
\begin{figure}
\centering

\begin{tikzpicture}
  [inner sep=0.2,
  vorbit/.style={circle, draw=black,thick,minimum size = 0.2cm},
  worbit/.style={circle, draw=black, fill=black,thick,
    minimum size = 0.2cm},scale=2]
  \clip (-0.75,0.6) rectangle (3.30,-2.25);
  \dottedcircle{0,0}{1};
  \dottedcircle{2,0}{1};
  \dottedcircle{0,-2}{1};
  \node (v0) at (0,0) [vorbit,label=above:$v_0$] {};
  \node (a2b1v0) at (0,-2)
  [vorbit,label=below:${a_2b_1v_0=a_2b_1a_1v_0}$]{};
  \node (b1v0) at (2,0) [vorbit,label=above:${b_1v_0=b_1a_1v_0}$] {};
  \writeabove{v0}{$T_F$}{0.45};
  \writeabove{b1v0}{$b_1a_1T_F$}{-0.5};
  \writeright{a2b1v0}{$\rho T_F$}{0.75};
  \node (w0) at (1,0) [worbit,label=above:$w_0$] {};
  \node (b1a1w0) at (3,0) [worbit, label=above:$b_1a_1w_0$]{};
  \node (a2w0) at (0,-1) [worbit, label=left:$a_2w_0$]{};
  \node (a1w0) at (0.707106781,-0.707106781) [worbit,
  label=right:$a_1w_0$]{};
  \draw[very thick] (v0) -- node[below] {$e$} (w0);
  \draw[very thick] (w0) -- (b1v0);
  \draw[very thick] (b1v0) -- node[below] {$b_1e$}  (b1a1w0);
  \draw[very thick] (v0) -- node[left] {$a_1e$} (a1w0);
  \draw[very thick] (v0) -- node[left] {$a_2e$}(a2w0);
  \draw[very thick] (a2w0) -- node[left] {$a_2b_1e$} (a2b1v0);  
\end{tikzpicture}

\caption{The path from $v_0$ to $\rho v_0$ only intersects one
$\FRS$-edge orbit}
\label{diagram1}
\end{figure}
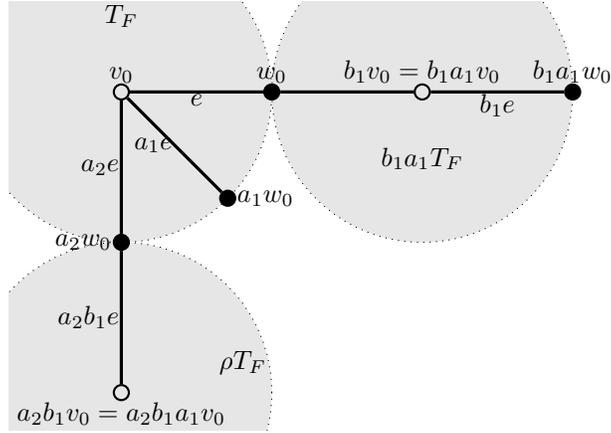

Looking again at Figure \ref{diagram1} we see that $T_F \cap b_1a_1T_F
\neq \emptyset$ but more specifically that $b_1 \te \subset b_1a_1
T_F$. This gives :
\begin{eqnarray*}
&& {}^{b_1}\stab(e)\cap {}^{b_1a_1}F \neq \{1\}\\
&& \Rightarrow \stab(e) \cap {}^{a_1}F \neq \{1\} \\
&& \Rightarrow {}^{a_1^{-1}}\stab(e) \cap F \neq \{1\}
\end{eqnarray*}
Which implies $a_1^{-1}e \subset T_F$ by Lemma \ref{lem:tf-transitive}
there is some $f_1^\mo \in F$ such that $f_1^\mo a_1^\mo \te = \te
\Rightarrow f_1^\mo a_1^\mo = b' \in H \cap F$. It follows that $\rho
f_1 = b_1a_1f_1 = b_1b'^\mo \in H$. Case 2 is proved similarly.
\end{proof}

\subsubsection{Avoiding transmissions}\label{sec:flow}
Folding sequences are difficult to analyse. The goal of this section
is to prove some lemmas that give us some control over folding
sequences, making them more tractable. 

The only moves applied to a $\G(A)$-graph $\B$ that may increase the
number of non-trivial $\B$-vertex groups are T1 transmissions and F4
folds, however an F4 fold decreases the number of cycles in the
underlying graph $B$ by 1. Whenever there is a proper transmission to
some $\B$-vertex group $\B_v$, it gets enlarged to $\bk{\B_v,\alpha}$
for some $\alpha \in A_{[v]}$, this in a sense increases the
complexity of $\B$. We will give  conditions that enable
us to perform a maximal number of F1 and F4 foldings without having to
resort to transmissions.  This enables us to keep most $\B$-vertex
groups of our $\G(A)$-graphs trivial or cyclic.

\begin{conv}\label{conv:nice-edge-groups}
  Throughout this section we will assume that in the graph of groups
  $\G(A)$, the images of edge groups are maximal abelian and malnormal
  in the vertex groups. We will also require our splittings to be
  1-acylindrical. These conditions imply that if there are distinct
  edges $e,f$ in $A$ such that $i(e)=u=i(f)$ then the images of
  $i_e:A_e \rightarrow A_u$ and $i_f:A_f \rightarrow A_u$ cannot be
  conjugated into one another inside $A_u$.
\end{conv}

\begin{defn}\label{defn:cancellable} 
  Let $u$ be a vertex in a $\G(A)$-graph $\B$ such that $\B_u $ is
  either abelian or trivial. Suppose we have a subgraph \[l=
  \xymatrix{v\bullet \ar[r]^{(a,e,b)} & u\bullet
    \ar[r]^{(a',e^\mo,b')} & \bullet w}\] such that we are able to
  perform a sequence of transmissions $t_1,t_2$, where $t_1$ and $t_2$
  are through the different edges of $l$, and suppose
  afterwards $\B_u$ is still abelian, then $l$ is called a
  \emph{cancellable path centered at $u$}.
\end{defn}

\begin{lem}\label{lem:abelian-free-prod}
  Let $A,B \leq G$ be two abelian subgroups of a fully residually free
  group $G$ such that for some $a \in A, b \in B$ we have $[a,b]\neq
  1$. Then we have \[\bra A, B \kett = A*B\]
\end{lem}
\begin{proof}
  Let $w=a_1b_2a_3\ldots b_n$ be product of non-trivial factors $a_i
  \in A$ and $b_j \in B$ with perhaps the exception that $a_1$ or
  $b_n$ are trivial. Since $G$ is fully residually free there exists a
  map of $G$ into $F$ such that all the non-trivial $a_i, b_j$ as well
  as some commutator $[a,b], a\in A, b\in B$ do not vanish. We have
  that the $a_i$ are sent to powers of some element $u \in F$
  and the $b_j$ are sent to powers of some $v \in F$. It follows that the
  homomorphic image of $w$ is sent to a freely reduced word in $u$ and
  $v$, and since $u,v \in F$ do not commute they freely generate a
  free subgroup of $F$. It follows that $w$ is not sent to a trivial
  element.
\end{proof}

\begin{lem}\label{lem:yellow-fold} Let $\B$ be a $\G(A)$-graph and
  suppose it has a cancellable path centered at $u$. Then it is
  possible to perform a folding (F1 or F4) at $u$ in $\B$ using only a
  Bass-Serre A1 move, and maybe a conjugation A0 move.
\end{lem}

\begin{proof} 
  W.l.o.g. the edges in the cancellable path are edges $e,e'$ with
  $i(e) = u = i(e')$ and with labels $(a,[e],b), (a',[e'],b')$
  respectively. We have $\B_u \leq A_{[u]}$ and we have injections
  $i_{[e]}:A_{[e]} \hookrightarrow A_{[u]}$ and $i_{[e']}:A_{[e']}
  \hookrightarrow A_{[u]}$. Denote the images $i_{[e]}(A_{[e]}) = A$
  and $i_{[e']}(A_{[e']}) = A'$.

  Let $\B_2$ be the $\G(A)$-graph obtained from $\B$ after applying
  the two transmissions that witness the fact that $e,u,e'$ is a
  cancellable path, then \[(\B_2)_u \cap a A a^\mo \neq \{ 1 \} \neq
  (\B_2)_u \cap a' A' a'^\mo\]

  By Lemma \ref{lem:abelian-free-prod} $\bk{a A a^\mo, a' A a'^\mo}$
  is either $(aAa^\mo)*(a' A a'^\mo)$, or free abelian. Since
  $(\B_2)_u$ is abelian, $\bk{a A a^\mo, a' A a'^\mo}$ is free
  abelian. The subgroups $a A a^\mo$ and $a' A' a'^\mo$ therefore lie
  inside a maximal abelian group $C$. It therefore follows that $A$ and
  $A'$ are not conjugacy separated in $A_{[u]}$. Convention
  \ref{conv:nice-edge-groups} forces $A = A'$ and hence $[e]=[e']$. This also means
  that by malnormality \[a'^\mo a A a^\mo a' = A \Rightarrow a'^\mo a
  = i_{[e]}(\alpha) \in A\] for some $\alpha \in
  A_{[e]}$.

  Therefore if we consider $\B$ before any transmissions were
  performed, using a Bass-Serre A1 move we can change the
  label of $e'$ as follows:\[(a',[e'],b') \rightarrow (a'
  i_{[e]}(\alpha),[e],t_{[e]}(\alpha^\mo) b') = (a,[e],b'')\] and then
  either perform an F4 fold if $t(e) = t(e')$ or an A0 conjugation at
  $t(e)$ and then an F1 fold at $u$.
\end{proof}

\begin{lem}\label{lem:long-range-crit} Let $\B$ be a
  $\G(A)$-graph. Suppose that $\B_{v_0},\ldots,\B_{v_m}$ are the
  non-trivial $\B$-vertex groups. Suppose for some $u
  \not\in\{v_0,\ldots,v_m\}$, after a sequence $t_1,\ldots,t_n$ of
  transmissions yielding a $\G(A)$-graph $\B_n$, the $\B$-vertex group
  $(\B_n)_u$ is abelian or trivial and after maybe making some A0-A2 adjustments
  it is possible to perform either move F1 or F4 at $u$. Then it is
  also possible to perform a folding move F1 or F4 at $u$ after only
  applying a sequence of A0-A2 and L1 adjustments to $\B$.
\end{lem}
\begin{proof}
  By Lemma \ref{lem:yellow-fold} we may assume that there are no
  cancellable paths in $\B$. We observe that the A0 conjugation and
  the A1 Bass-Serre moves which can be made at (an edge incident to)
  $u$ do not depend on $\B_u$.

  Suppose first that $(\B_n)_u$ is trivial, then there
  are no new possible A2 simple adjustments at $u$ in $\B_n$, so the
  result holds.

  Suppose now that $(\B_n)_u$ is non-trivial abelian and that in
  $(\B_n)$ we can perform an A2 move changing the label of $e$, where
  $i(e)=u$ and then after performing moves A1,A0 at $t(e)$ we can do
  either move F1 or F4 identifying $e$ and $e'$. In particular $e$ and
  $e'$ had labels $(a,[e],b)$ and $(a',[e],b')$ respectively and after
  doing move A2 at $u$ the labels became either $(a'c,[e],b)$ and
  $(a',[e],b')$ or $(a,[e],b)$ and $(ac,[e],b')$, for some $c$ in the
  edge group, respectively.

  If there were no transmission through either $e$ or $e'$, then
  w.l.o.g. there were no transmissions through $e'$ and we can use a L1
  long range adjustment in $\B$ to change $(a',[e'],b')$ to
  $(ac,[e'],b')$. In both cases after applying an A1 Bass-Serre move,
  we can then apply an F4 or (after the appropriate A0 conjugation)
  an F1 folding move.

  Otherwise there were transmissions through $e$ and $e'$, so$e,u,e'$
  is a cancellable path in $\B_{n}$. Lemma \ref{lem:yellow-fold}
  implies that we only needed A0 and A1 moves to change the labels of
  $e,e'$ to enable an F1 or F4 fold. In particular, we could have made
  these moves before in $\B$ before any transmissions were used. The
  result now follows.
\end{proof}

\subsubsection{The strategy}
For each possibility given in (\ref{eqn:reg-norm-split}) in Section
\ref{sec:collapses} we will do the following:\begin{enumerate}
\item Get the maximal abelian collapse of the JSJ of $\FRS$, and use
  this splitting as the underlying graph of groups $\G(X)$. This must
  be one of the graphs given in (\ref{eqn:reg-norm-split}).
\item We prove using Theorem \ref{thm:weidmann-nielsen} and Lemma
  \ref{lem:make-elliptic}, that we can always arrange so that $\x$ is
  somehow simple i.e. $\x$ lies in $\Fh \cup H$ or is a stable letter,
  up to conjugating boundary monomorphisms.
\item We then take the wedge $\B=\W(F,\x,\z)$, with the loop $\L(\x)$
  of length at most 2. Now $\B$ must fold down to $\G(X)$.
\item We will then apply folding moves to simplify the graph as much
  as possible \emph{while avoiding transmissions}. It will turn out
  using the results of Section \ref{sec:flow} that the resulting
  $\G(X)$-graph $\B$ will have the underlying graph as $X$.
\item All that will then remain to get a folded graph is to make some
  transmission moves, keeping track of these will tell us how the
  vertex groups are generated.
\item Finally, by arguing algebraically we will recover the original
  cyclic JSJ decomposition of $\FRS$ modulo $F$.
\end{enumerate}
\subsection{The one edge case}\label{sec:MAC-1e}

We consider the case where the maximal abelian collapse of $\FRS$ is a free product with
amalgamation: \begin{equation}
  \label{eqn:2-verts-1-edge} \FRS = \Fh*_A H
\end{equation} with $F \leq \Fh$, $A$ maximal abelian in both
factors and $H$ non-abelian.  Throughout this section $\Fh,A,H$
will denote these groups.

\subsubsection{Arranging so that $\x$ lies in $H$}

\begin{lem}\label{lem:2v-1e-simplification} Let $\FRS$ be freely
  indecomposable modulo $F$ with a maximal abelian collapse
  (\ref{eqn:2-verts-1-edge}). After Weidmann-Nielsen normalization on
  $(F,\x,\z)$ we can arrange; conjugating boundary monomorphisms if
  necessary; so that $\x$ lies in either $\Fh$ or $H$.
\end{lem}
\begin{proof}
  Since we are assuming free indecomposability of $\FRS$ modulo $F$,
  we can apply Theorem \ref{thm:weidmann-nielsen}.  Let $T$ be the
  Bass-Serre tree induced from the splitting
  (\ref{eqn:2-verts-1-edge}). Let $v_0 = \fix(F)$. We start by looking
  at the marked generating set $(F;\{\x,\z\})$. We consider different
  cases.

  {\bf Case I $T_F$ is a point:} Since $\FRS$ isn't freely
  decomposable by Theorem \ref{thm:weidmann-nielsen} w.l.o.g. after
  Weidmann-Nielsen normalization $\x$ must be elliptic. $T_\bk{\x}$ is
  either a vertex or an edge

  {\bf Case I.I $T_F \cap T_\bk{\x} = \emptyset$:} Consider the marked
  generating set $(F,\bra \x \kett;\{\z\})$ and apply Theorem
  \ref{thm:weidmann-nielsen} again. We now find that after
  Weidmann-Nielsen normalization either, w.l.o.g. $\z T_\bk{\x} \cap T_\bk{\x}
  \neq \emptyset$ or $\z$ is also elliptic.
  
  We always have the latter possibility. Indeed, by 1-acylindricity
  $T_\bk{\x}$ is either an edge or a point so for $\z T_\bk{\x} \cap T_\bk{\x} \neq
  \emptyset$ we must have that $\y$ fixes one of the endpoints of
  $T_\bk{\x}$ which implies that $\z$ is elliptic.
  
  If $\z$ is also elliptic then the trees $T_F,T_\bk{\x},T_\bk{\y}$ cannot all
  be disjoint, otherwise $\FRS$ would be freely decomposable modulo
  $F$.  If $T_\bk{\y} \cap\ T_F \neq \emptyset$ then we can switch $\x$ and
  $\z$ and pass to Case I.II Otherwise $T_\bk{\y} \cap T_\bk{\x} \neq \emptyset$
  then the tree $T_{\bra \z,\x \kett}$ is either a point, an edge, or has radius
  1. Passing to the marking $(F,\bra \x,\z \kett;\emptyset)$ and
  applying Theorem \ref{thm:weidmann-nielsen} implies that $T_{\bra
    \z,\x \kett}$ can be taken so that $T_F \cap T_{\bra \z,\x
    \kett}\neq \emptyset$. Which means that, conjugating boundary
  monomorphisms in $\Fh$ if necessary, both $\x$ and $\z$ can be
  brought into $H$.

  {\bf Case I.II $T_F \cap T_\bk{\x} \neq \emptyset$:} We are assuming that
  $T_F = v_0$. Since $\x$ fixes $v_0$ we have $\x \in \Fh$.

  {\bf Case II $T_F$ is not a point:} Conjugating boundary
  monomorphisms, we can arrange for some generator $\alpha$ of $A$ to
  lie in $F$. We apply Theorem \ref{thm:weidmann-nielsen} and find
  that either $\x T_F \cap T_F \neq \emptyset$ or $\x$ is elliptic. In
  the former case we have $d(v_0,\x v_0)=2$ and we can apply Lemma
  \ref{lem:make-elliptic} to make $\x \in H$ and we are done.
  Otherwise $\x$ is elliptic and we consider the next case.

  {\bf Case II.I $T_F \cap T_\bk{\x} = \emptyset$:} We consider the marked
  set $(F,\x;\{\z\})$ and we see that applying Theorem
  \ref{thm:weidmann-nielsen} we can either arrange for $\z$ to be
  elliptic or get that either $T_F \cap \z T_F \neq \emptyset$ or
  $T_\bk{\x} \cap \z T_\bk{\x}\neq \emptyset$. In Case I.I the latter
  possibility was seen to be impossible unless $\z$ was elliptic.  If
  $\z T_F \cap T_F \neq \emptyset$, then we can apply Lemma
  \ref{lem:make-elliptic} as for the previous case and obtain that $\z
  \in H$ and we are done.

  It therefore remains to verify the case where $\z$ is elliptic and
  $T_\bk{\y} \cap T_F = \emptyset$. For our group not to be freely
  decomposable modulo $F$ we must have that $T_\bk{\y} \cap T_\bk{\x} \neq
  \emptyset$.  Moreover since both $\x,\z$ are elliptic the tree
  $T_{\bra \x,\z \kett}$ must have radius 1. This is dealt with
  exactly as in the end of Case I.I.
  
  {\bf Case II.II $T_F \cap T_\bk{\x} \neq \emptyset$:} This means that
  $\x$ either fixes $v_0$ or some vertex in $T_F$. In all cases,
  applying Lemma \ref{lem:tf-transitive} and conjugating boundary
  monomorphisms if necessary, this implies that $\x \in H$ or $\x \in
  \Fh$.
\end{proof}

\begin{lem}\label{lem:no-conjugator}
  Let $A \leq \FRS$ be abelian and let $c$ be such that $c^\mo A c
  \cap A = \{1\}$. Then $c \not\in K = \bk{A,c^\mo A c}$ and the
  centralizer of $A$ in $K$, $Z_K(A)=A$.
\end{lem}
\begin{proof}
  By Lemma \ref{lem:abelian-free-prod} $K = A*B$ where $B =  c^\mo A
  c$.  If there is some word $U(A,B)$ such that we have the relation
  \[U(A,B)^\mo a^\mo U(A,B) a = 1\] for some $a\in A$, then the free
  product structure implies that $U(A,B) = U(A)$, so the second claim
  holds. It also follows that $Z_K(B)=B$. Suppose now that $c \in K$
  then for each $a \in A$ we have that $c^\mo a c \in B$, which is
  impossible from the free product structure.
\end{proof}

\begin{lem}\label{lem:2-v-1-e-Fx} Suppose  $\FRS$ has a maximal
  abelian collapse (\ref{eqn:2-verts-1-edge}). If $\FRS$ is generated
  as $\bk{F,\x,\y}$ with $\x \in \Fh$, then $\FRS$ is freely
  decomposable modulo $F$.
\end{lem}
\begin{proof}
  Let $\G(X)$ be the graph of groups representing the splitting
  (\ref{eqn:2-verts-1-edge}). Then $\FRS$
  can be represented by a $\G(X)$-graph $\B$ which consists
  of a vertex $v$ with $\B_v = \bfxk$ and a $\z$-loop $\L(\z,v)$.

  Since $\pi_1(\B,v)=\FRS$, by Theorem \ref{thm:folded} we should be
  able to bring $\B$ to $\G(X)$ using the moves of Section
  \ref{sec:graph-folds}. Now only the $\B$-vertex group $\B_v$ 
  is non-trivial. We do our folding process using only  moves 
  A0-A3,F1,F4,L1 and S1. If an F4 collapse occurs then $\B$ doesn't
  have any cycles, but it may have an extra non-trivial cyclic
  $\B$-vertex group. By doing S1 shaving moves we can assume that
  either $\B$ is a line with endpoints the vertices $v,u$ with $\B_v =
  \bfxk$, $\B_u$ cyclic and all other $\B$-vertex groups trivial, or
  $\B$ is a loop with only $\B_v=\bfxk$ non-trivial.

  By Lemma \ref{lem:long-range-crit} and Lemma \ref{lem:yellow-fold}
  we see that we can always avoid using transmissions unless one of
  the three following possibilities occurs: 

  {\bf Case I:} $\B$ has two vertices $v,u$ and one edge $e$. At this
  point we have $\B_u = H'$ which is cyclic. After a transmission
  $\B_u = \bk{A,H'}$, where $A$ is conjugable into the image of an
  edge group. The graph is then folded, but we see by Lemma
  \ref{lem:abelian-free-prod} that $H = A * H'$ which implies free
  decomposability of $\FRS$ modulo $F$.

  {\bf Case II:} $\B$ has no cycles, three vertices and two edges. We
  assume that all S1 shaving moves were performed. Then the only
  possibility for $\B$ is that it has endpoints $v$ and $u$, with
  $\B_u$ cyclic, and the other $\B$-vertex group $\B_w$ is trivial. If
  it is possible to transmit from $\B_u$ then $u$ can be shaved
  off. If it is possible to transmit from $v$ to $w$ and then from $w$
  to $u$ then there is a cancellable path centered at $w$ and Lemma
  \ref{lem:yellow-fold} applies. So in both cases we can continue folding
  $\B$ without using transmissions. 

  {\bf Case III:} $\B$ consists of a cycle of length 2, with vertices
  $v$ and $u$. This means that the $\B$-vertex group $\B_u$ is
  trivial. For $\B$ to be folded, after a transmission we must be able
  to perform an F4 collapse move.

  If the collapse is towards $u$ we distinguish two
  possibilities. Either no transmissions are needed and $\B_u$ will be
  generated by some element and the edge group, which implies free
  decomposability modulo $F$; or there is a transmission from $v$ to
  $u$ through one edge followed by a transmission from $u$ to $v$
  through the other edge, but this gives a cancellable path so Lemma
  \ref{lem:yellow-fold} applies, and we get a collapse from $u$
  towards $v$ which we immediately deal with below.

  The remaining possibility is that the collapse is towards $v$. If no
  transmissions were needed then $\B_u$ is generated by the edge group
  and we have $\FRS = \Fh$ -- contradiction.  Otherwise by Lemma
  \ref{lem:long-range-crit}, before the collapse, we must do
  \emph{two} transmissions from $v$ to $u$ so that, after an A0
  conjugation move, we have w.l.o.g. $\B_u =\brakett{b_1 A'
    b_1^\mo,A''}$ where $A',A'' \subset A$ and $b_1 \not\in A$.  We
  now either do a simple adjustment and collapse towards $v$ or some
  non-trivial transmission from $u$ to $v$. For a collapse we need
  $b_1 \in \brakett{b_1 A' b_1^\mo,A''}$, for a non-trivial
  transmission we need either $A''$ or $b_1 A' b_1^\mo$ to have a
  proper centralizer in $\B_u$. Both of these are forbidden by Lemma
  \ref{lem:no-conjugator}.
      
\end{proof}

\subsubsection{Arranging so that $\y$ lies in $H$ as well}

\begin{lem}\label{lem:F-and-abelian}
  Let $A\leq \FRS$ be a maximal non-cyclic abelian subgroup, then
  either:\begin{itemize}
  \item $\bk{F,A} = F * A$.
  \item There is some $p\in F$ such that $\bk{F,A} =  F*_{\bk{p}}A$.
  \item $\bk{F,A} = \bk{F,r,A \mid [r,p] = 1}$ where $p \in F$ and
    $\bk{a} = \bk{F,r} \cap A$. Moreover $\bk{F,r} = \bk{F,a}$. In
    particular no conjugate of an element of $F$ is centralized by $A$.
\end{itemize}
\end{lem}
    
\begin{proof}
  Since $\bk{F,A} \leq \FRS$ it is fully residually $F$ and since $F
  \neq \bk{F,A}$ it has an essential cyclic or free splitting modulo
  $F$. Since $A$ is non-cyclic abelian it is forced to be
  elliptic. Suppose that $\bk{F,A} \neq F * A$, then it has an
  essential cyclic splitting and since it is generated by elliptic
  elements the underlying graph of groups has no cycles. The only
  possibilities are that there is some $a \in A$ such that $\bk{F,A} =
  \bk{F,a}*_{\bk{a}} A$ or that there is some $p \in F$ such that
  $\bk{F,A} = F*_{\bk{p}} H$ with $H = \bk{p,A}$. 

  We first consider the former case, since $\bk{F,A}$ is assumed to be
  freely indecomposable, Theorem \ref{thm:onevariable} implies that
  $\bk{F,a} = \bk{F,r | [r,p]}$ for some $p \in F$, but if we had 
  $[a,f]=1$ for some $f \in F$ by normal forms we see $[f,A]\neq 1$
  contradicting commutation transitivity. $A$ is therefore the
  centralizer of an element that is hyperbolic in the cyclic JSJ of
  $\bk{F,r | [r,p]}$ modulo $F$.

  The remaining case is that $\bk{F,A} = F*_{\bk{p}} H$ with $H =
  \bk{p,A}$.  By Lemma \ref{lem:abelian-free-prod} $H\neq \bk{p}*A$
  only if $p \in A$. The result therefore holds.
\end{proof}
    
\begin{cor}\label{cor:F-and-abelian}
  Let $A\leq \FRS$ be a non-cyclic maximal abelian subgroup that
  centralizes some $\alpha \in F$, and let $a \in \FRS$ be such that
  there is no $f \in F$ such that $fa \in A$. Then $\bk{F,aAa^\mo} = F
  * aAa^\mo$ and $a \not\in \bk{F,aAa^\mo}$.
\end{cor}
\begin{proof}
  By Lemma \ref{lem:F-and-abelian} $\bk{F,aAa^\mo} = F * aAa^\mo$ or
  $F*_{\bk{p}}aAa^\mo$ for some $p \in F$. Suppose towards a
  contradiction that the latter possibility held. Then $p = a \alpha''
  a^\mo; \alpha'' \in A$. $A$ is discriminated by $F$-retractions
  $\FRS \rightarrow F$. This means that every element of $aAa^\mo
  \setminus F$ can be sent to arbitrarily high powers of $\alpha$ via
  such retractions. It follows that $\alpha'' = \alpha^n$. Since $F$
  has property CC there is some $f \in F$ such that $f (a \alpha^n
  a^\mo) f^\mo = \alpha^n$ which implies that $[fa,\alpha] = 1
  \Rightarrow fa \in A$, since $A$ is maximal abelian. This
  contradicts the hypothesis.

  We therefore have $\bk{F,aAa^\mo} = F * aAa^\mo$, and $a \not\in
  \bk{F,aAa^\mo}$ follows from the fact that the free product
  structure implies that the centralizer of $\alpha$ lies in $F$.
\end{proof}
    
\begin{lem}\label{lem:2-v-1-e-Hx}
  Suppose $\FRS$ is freely indecomposable modulo $F$ and has a maximal
  abelian collapse (\ref{eqn:2-verts-1-edge}) and is generated as
  $\bk{F,\x,\y}$ with $\x \in H$. Then $\FRS$ is generated by
  $\brakett{F,\x,\z'}$ where $\z'$ also lies in $H$.
\end{lem}
\begin{proof}
  We start with the $\G(X)$-graph $\B$ with one edge $e$ and two
  vertices $v,u$ with $\B_v = F$ and $\B_u = \brakett{\x}$, then at
  $v$ attach the $\z$-loop $\L(\z,v)$. Start our adjustment-folding
  process applying only moves A0-A3,F1,F4,L1,S1, as much as possible,
  but avoiding transmissions.

  {\bf Case I:} We brought $\B$ to a graph with two vertices $u,v$ and
  one edge without having to use transmissions. We then either have
  $\B_u=F$ and $\B_v = \bk{\x,\z'}$, in which case the result follows;
  or $\B_u=\bk{F,\y'}$ and $\B_v = \bk{\x}$, which by Lemma
  \ref{lem:2-v-1-e-Fx} implies free decomposability of $\FRS$ modulo
  $F$ -- contradiction.

  {\bf Case II:} An F4 collapse occured. In this case $\B$ is a line
  with one endpoint either $u$ or $v$ and the other endpoint is some
  vertex $w$ with $\B_w = \bk{\y'}$. We see that if any transmissions
  from $w$ were possible then could use an S1 shaving move and remove
  $w$. On one hand this graph should fold down to $\G(X)$, on the
  other hand by Lemma \ref{lem:long-range-crit}, if $w$ is at distance
  at least 2 from either $u$ or $v$ we can apply an F1 move without using
  transmissions.

  We can therefore assume that $\B$ has two edges, and three vertices
  with $w$ as an endpoint. The last fold will be an F1 fold.  We note
  that it is impossible for there to be a transmission \emph{to} $w$
  followed by a transmission \emph{from} $w$. Indeed, suppose this was
  the case, then after the first transmission we have by Lemma
  \ref{lem:abelian-free-prod} that $\B_w = \bk{\y',A'} =
  \bk{\y'} * \bk{A'}$ where $A'$ is conjugable into $A$. Now for there
  to be a transmission back from $w$ we need $A'$ to have a proper
  centralizer in $\bk{\y'} * \bk{A'}$, which is impossible.

  So any transmissions preceding a simple adjustment that enables the
  F1 fold will be through the edge connecting $u$ and $v$. It follows
  that instead we can make a long range adjustment L1 on the edge
  adjacent to $w$, and then apply the F1 fold. Since no transmissions
  were used we have reduced this to the Case I.

  {\bf Case III:} No collapses occured. In this case $\B$ contains a
  cycle and the only non-trivial $\B$-vertex groups are $\B_u$ and
  $\B_v$. We note that the cycle must have an even length so by Lemma
  \ref{lem:long-range-crit} we can assume the cycle has
  length 2 and contains $u$ or $v$. We distinguish two subcases

  {\bf Case III.I:} $\B$ has three vertices $u,v,w$ with $B_w = \{1\}$
  and the cycle in $\B$ consists of two edges $e,f$ going from $w$ to
  $u$ or $w$ to $v$. First note that it is impossible for there to be
  a transmission to $w$ through $e$ followed by a transmission from
  $w$ through $f$ as then we would have a cancellable path and could
  make a collapse at $w$ contrary to our assumptions.

  Suppose now that it was possible to transmit to $w$ through the
  edges $e$ and $f$, then by Lemma \ref{lem:no-conjugator} it is
  impossible to perform a simple adjustment at $w$ preceding an F4
  collapse at $w$ and it is impossible to make a new transmission
  from $w$ back through $e$ or $f$. It therefore follows that the next
  fold was preceded only by transmission through the edge between $u$
  and $v$, so again we can make an L1 long range adjustment to change
  the label of either $e$ or $f$ (but not both) and then perform
  either an F4 collapse, which brings us to case I, or an F1 fold
  which brings us to case III.II.

  {\bf Case III.II:} $\B$ has two vertices, two edges and one
  cycle. Then we can represent $\B$ as the $\G(X)$-graph
  \begin{equation}\label{eqn:1-e-2-v-2-C} 
    \xymatrix{v\bullet
      \ar@/^/[r]^{(a,e,b)}
      \ar[r]_{(a',e,b')} & u \bullet}
  \end{equation} with $a,a' \in \Fh$ and
  $b,b' \in H$ and $\B_u = F, \B_v = \bk{\x}$. Now if a
  transmission from $u$ were possible then w.l.o.g. we could
  express $\FRS$ as the same $\G(X)$-graph but with $\B_v = \bk{F,
    ab\x b^\mo a^\mo}$ and $\B_u = \{1\}$, which by Lemma
  \ref{lem:2-v-1-e-Fx} implies free decomposability modulo $\FRS$.

  It follows that if no transmissions from $v$ are possible then no
  transmissions at all are possible and we can therefore make an F1
  collapse without transmission and reduce to the Case I.

  We may therefore assume that a transmission from $\B_v$ is possible
  which implies that $F \cap A = \bk{\alpha} \neq \{1\}$ and that
  there is some $f \in F$ such that $fa' \in A$. This means that
  using an A2 simple adjustment, an A1 Bass-Serre move and an A0
  conjugation we may assume that $a' = b' = 1$ in
  (\ref{eqn:1-e-2-v-2-C}). We can therefore put $\B_v = F$ and $\B_u =
  \bk{\x,\alpha}$.  Now note that if it were also possible to transmit
  from $v$ through the other edge then before any transmissions we
  could use an A2 simple adjustment an A1 Bass-Serre move to change the
  label $(a,e,b)$ to $(1,e,b'')$ in (\ref{eqn:1-e-2-v-2-C}), which
  means that we can reduce to the Case I. Hence,
  \begin{itemize}
  \item[$(\dagger)$]We may assume that there is no $f \in F$ such that
    $fa \in A$.
  \end{itemize}
  If no further transmissions are possible then we must be able to
  perform a collapse from $u$ to $v$. This means that $b \in
  \bk{\x,\alpha}$ so that after collapsing we get a $\G(X)$-graph $\B$
  with one edge labeled $(1,e,1)$ and $\B$-vertex groups $\B_v =
  \bk{F,a}$ and $\B_u = \bk{\x,\alpha}$. After transmissions we have
  $\B_u = \bk{A,\x}$ which by Lemma \ref{lem:abelian-free-prod}
  implies that $\FRS$ is freely decomposable modulo $F$.

  We can therefore assume that $b \not \in \bk{\x,\alpha}$ but that
  there is a transmission from $u$ to $v$ through the edge labeled
  $(a,e,b)$. This means that there is some $\alpha' \in A$ such that
  $b^\mo \alpha' b \in \bk{\x,\alpha}$. So after the transmission we
  have that $\B_v = \bk{F,a \alpha' a^\mo}$. By Corollary
  \ref{cor:F-and-abelian} and ($\dagger$) $\B_v \leq F* a A a^\mo$ and
  since $\bk{\alpha'}\leq A$, it follows that $\B_v = F* a
  \bk{\alpha'} a^\mo$. So no further transmissions from $v$ to $u$ are
  possible since by the free product structure $Z_{\B_v}(\alpha)$, the
  centralizer of $\alpha$ in $\B_v$ is $\bk{\alpha}$.

  Since $\B$ must fold down to a graph with one edge we must be able
  to perform a simple adjustment to change the label $(1,e,1)$ to
  $(a,e,1)$ and fold $\B$ down to: \[ \xymatrix{ v \bullet
    \ar[r]^{(a,e,1)} & \bullet u}\] But to do this we would need $a\in
  \B_v$, but we saw in the previous paragraph that after the only
  possible transmission $\B_v \leq F* a A a^\mo$, so by Corollary
  \ref{cor:F-and-abelian} this is impossible. Having exhausted all the
  possibilities the result follows.
\end{proof}

From the previous lemmas we get.

\begin{prop}\label{prop:2-v-1-e-classification} If $\FRS$ is
  freely indecomposable modulo $F$ and has a maximal abelian collapse
  (\ref{eqn:2-verts-1-edge}), then, conjugating boundary monomorphisms
  if necessary, $\FRS$ can be generated by $F$ and two elements $\x,\z
  \in H$.
\end{prop}

\subsubsection{Recovering the original cyclic splitting from the
  maximal abelian collapse}

This next proposition enables us to revert to a cyclic splitting.
\begin{prop}\label{prop:1e-h-3gen}
  Suppose that $\FRS$ is freely indecomposable modulo $F$ and has a
  maximal abelian collapse (\ref{eqn:2-verts-1-edge}) then $\FRS$
  admits a cyclic splitting \[\F'*_\brakett{\alpha}H'\] where either:
  \begin{enumerate}
  \item $\F'=F$ and $H'$ is generated by $\brakett{\alpha,\x,\z},
    \alpha \in F$. Hence $H'$ is a three generated fully residually
    free group (see Theorem \ref{thm:FGMRS}).
  \item $\F'=\brakett{F,\alpha}$ and $H'=\brakett{\x,\z}$ with $\alpha
    \in H'$ i.e. $H'$ is free of rank 2.
  \end{enumerate}
\end{prop}
\begin{proof}
  We first consider when the amalgamating maximal abelian subgroup $A$
  in (\ref{eqn:2-verts-1-edge}) is cyclic. We write
  $A=\brakett{\alpha}$ and $\F'=\F, H'=H$, then this splitting is in
  fact already a maximal abelian collapse so we can apply Proposition
  \ref{prop:2-v-1-e-classification} and we get that $\x,\z \in H$.
  Looking at normal forms we have that $\Fh=\brakett{F,\alpha}$ and
  $H=\brakett{\alpha,\x,\z}$ if $\alpha\in F$ then $\Fh=F$ and $H$ is
  three generated fully residually free. If $\alpha \not\in F$ then we
  must have $\alpha \in \brakett{\x,\z}$ and it follows that
  $H=\brakett{\x,\z}$.

  We now consider the case where $A$ in (\ref{eqn:2-verts-1-edge}) is
  not cyclic. By Proposition \ref{prop:2-v-1-e-classification} we can
  assume that $\x,\z \in H$ and it follows that
  $\Fh=\brakett{F,A}$. First suppose that some conjugate of $A$
  centralizes some element of $F$. Then by Lemma
  \ref{lem:F-and-abelian} and by free indecomposability of $\FRS$
  modulo $F$ we have: \[\Fh=F*_\bk{\alpha}A\] which means that $\FRS$
  admits the splitting: \[\FRS=F*_\bk{\alpha}(A*_AH)\] and as before
  we get that $H'=\brakett{\x,\z,\alpha}$ is a three generated fully
  residually free group.

  The remaining possibility is that no conjugate of $A$ centralizes an
  element of $F$. So by Lemma \ref{lem:F-and-abelian} and by free
  indecomposability of $\FRS$ we have $\bk{F,A} = \Fh'*_\bk{a} A$
  where $\Fh'$ is a rank 1 centralizer extension of $F$. This means
  that we can unfold the splitting (\ref{eqn:2-verts-1-edge}) to get
  the cyclic splitting\[ \Fh'*_\bk{a} (A*_AH)\] and by proposition
  \ref{prop:2-v-1-e-classification}, $F \leq \Fh'$ and $\x,\y \in
  H$. Since no conjugate of $F$ intersects $A$ we need
  w.l.o.g. $\bk{x,y} \cap A = \bk{a^n}$ for some $n \in \Z$. Now by
  free indecomposability and by Theorem \ref{thm:onevariable},
  $\bk{F,a^n}$ must be a centralizer extension of $F$. Note however
  that the centralizer of $a^n$ in $\bk{F,a^n}$ must be $\bk{a^n}$ so
  there are no transmissions from $\Fh'$ back to $(A*_AH)$. Therefore
  \[\FRS = \Fh'*_{\bk{a^n}}(\bk{x,y})\] which contradicts
  the assumption that $\bk{a}$ was contained in a non-cyclic abelian
  subgroup.
\end{proof}
    
\define{An element in a free group is called primitive if it belongs
  to a basis}. For the next result, we need:
\begin{thm}[Main Theorem of \cite{Baumslag-1965}]
  \label{thm:Baumslag1965}
  Let $w=w(x_1,x_2,\ldots,x_n)$ be an element of a free group $F$
  freely generated by $x_1,x_2,\ldots,x_n$ which is neither a proper
  power nor a primitive. If $g_1,g_2,\ldots,g_n, g$ are elements of a
  free group connected by the relation\[ w(g_1,g_2,\ldots,g_n)=g^n
  ~~(m>1)\] then the rank of the group generated by
  $g_1,g_2,\ldots,g_n, g$ is at most $n-1$.
\end{thm}

\begin{lem}\label{lem:abelianbetween} Suppose $\FRS$ splits as
  \[ F*_\bk{\alpha}A*_\bk{\gamma} \bra \x,\z \kett\] where $A$ is
  non-cyclic abelian and $\bk{\gamma} \cap \bk{\alpha} = \{1\}$, then
  $\FRS$ is freely decomposable modulo $F$. In particular, $\bra \gamma \kett$
  must be a free factor of $\bk{\x,\z}$.
\end{lem}

\begin{proof}
  First note that $\alpha\in F$ and $\gamma \in \bk{\x,\y}$ aren't
  proper powers (otherwise, looking at normal forms, we would find a
  contradiction to commutation transitivity.) Suppose towards a
  contradiction that $\gamma \in \bk{\x,\y}$ was not primitive.

  Since $\bk{\x,\z}$ is a two generated non-abelian subgroup of a
  fully residually free group, there are retractions
  $f:\FRS\rightarrow F$ such that the restriction of $f$ to
  $\bk{\x,\z}$ is a monomorphism. Since we also have that $\alpha \in
  F$, there are retractions that are monomorphic on $\bk{\x,\y}$ that
  send $\gamma$ to arbitrarily high powers of $\alpha$. We fix $f$ so
  that it is injective on $\bk{\x,\y}$ and such that $f(\gamma)$ is a proper
  power. Denote by $K\leq F$ the image of $\bk{\x,\z}$. $K$ is free of
  rank 2 but we see that in the ambient group $f(\gamma)$ is a proper
  power. Theorem \ref{thm:Baumslag1965} applies and the assumption
  that $\gamma \in \bk{\x,\z}$ is not a proper power and not primitive
  force the image of $\bk{\x,\z} \rightarrow F$ to by cyclic
  contradicting injectivity of $f$ on $\bk{\x,\z}$.
\end{proof}

Lemma \ref{lem:f2-hnn} is a restatement of Lemma 2.10 of
\cite{Touikan-2007}. Unfortunately upon rereading I noticed that the
Lemma as stated is false, we give here a corrected version of the Lemma. (In
the original version the line (\ref{eqn:corrected}) is
$H=\brakett{G,t|t^\mo p t = q}; p,q \in G-\{1\}$.) Thankfully this
mistake doesn't impact the results of \cite{Touikan-2007} in a
significant way.

\begin{lem}\label{lem:f2-hnn}Let $H$ be
  a free group of rank 2 and let $w\in H$ be non primitive, and not a
  proper power. Then the only possible almost reduced (see Definition
  \ref{defn:almost-reduced}), hairless non-trivial cyclic splitting of
  $H$ modulo $w$ is \begin{equation}\label{eqn:corrected}
    H=\brakett{G,t|t^\mo p^n t = q}; p,q \in H-\{1\} \end{equation}
  where $w \in G, G$ is a free group of rank 2 and $n \in
  \Z$. Moreover $G=\brakett{p}*\brakett{q}$ so that $H=\brakett{p,t}$.
\end{lem}

Consider a splitting of a free group $H$ of rank 2 such as the one
given in (\ref{eqn:corrected}) we can apply a balancing fold (see
Definition \ref{defn:balancing-fold}) to replace the vertex group $G$
by \begin{equation}\label{eqn:balanced}\hat{G} = \bk{p}*\big(
  \bk{q}*_\bk{q} (t^\mo \bk{p} t)\big).\end{equation} Letting $\hat{q}
= t^\mo p t$ we have $\hat{q}^n = q, \hat{G} = \bk{p}*\bk{\hat{q}}$
and $H = \bk{\hat{G} ,t \mid t^\mo p t = \hat{q}} = \bk{p,t}$.

\begin{prop}\label{prop:1e-2v-3genH} Suppose that $\FRS$ is freely
  indecomposable modulo $F$ and that it splits as
  \begin{equation}\label{eqn:1e-2v-3genh} 
    F*_\brakett{\alpha} H \end{equation} and $H=G*_\bk{\beta}Ab(\beta,r)$ is a
  rank 1 free extension of a centralizer of a free group $G$ of rank 2.
  Then (\ref{eqn:1e-2v-3genh}) refines to 
  \begin{equation}\label{eqn:1e-2v-3genh2}
    F*_\brakett{\alpha} G*_\bk{\beta}Ab(\beta,r) 
  \end{equation}
\end{prop}

\begin{proof} The hypothesis arises as item 1 of Proposition
  \ref{prop:1e-h-3gen}. We need to show that $\alpha$ is conjugable into
  $G$. Suppose towards a contradiction that this was impossible. 

  We first consider the possible cyclic JSJ splittings of $H =
  G*_\bk{\beta}Ab(\beta,r)$. These correspond to cyclic splittings of
  $G$ modulo $\beta$. By Lemma \ref{lem:f2-hnn}, the two non-trivial
  possibilities, after possibly applying a balancing fold, are \[G =
  \bk{\beta}*\bk{\beta'} \tr{~or~} G =\brakett{G',s \mid s^\mo \gamma
    s = \gamma'} \tr{~with~} \beta \in G'.\] Since $\FRS$ is not freely decomposable, and
  replacing $G$ by $G'$ if necessary, we have by hypothesis that
  either $\alpha$ is conjugable into $Ab(\beta,r)$ or that
  $G*_\bk{\beta}Ab(\beta,r)$ has no cyclic splittings modulo
  $\alpha$. We note that in all cases, by commutation transitivity
  $\beta$ cannot be a proper power in $G$.

  Consider the case where $G*_\bk{\beta}Ab(\beta,r)$ has no
  cyclic splittings modulo $\alpha$. Let $\R$ be a strict
  resolution of $\FRS$. Since $\alpha \in F$ is always forced to be elliptic, the
  image of the subgroup $G*_\bk{\beta}Ab(\beta,r)$ in all the quotients of
  $\FRS$ in $\R$ is always forced to be elliptic, which implies that it
  is isomorphic to a subgroup of a free group, which
  impossible.

  Suppose now that $\alpha \in Ab(\beta,r)$. Since $\alpha\in
  Ab(\beta,r)$ is not a proper power $\bk{\alpha}$ is a direct summand
  of $Ab(\beta,r)$ if $\bk{\beta} \cap \bk{\alpha} \neq \{1\}$ then
  $\alpha=\beta$ and we are done. Otherwise Lemma
  \ref{lem:abelianbetween} applies and $\FRS$ is freely
  decomposable -- contradiction.
\end{proof}

\begin{prop}\label{prop:2-2-A-noconj}Suppose that $\FRS$ is freely
  indecomposable modulo $F$ and that it splits
  as \[\F'*_\brakett{\alpha} H' \] where $F\leq \F'$ and $H'$ is free
  of rank 2 and suppose moreover that $H'$ splits further as an HNN
  extension \[H'=\brakett{G,t \mid t^\mo \gamma^n t = \gamma'}\]
  modulo $\alpha$. Then $\alpha$ cannot be almost conjugate (as in
  Definition \ref{defn:almost-conjugate}) to either $\gamma$ or
  $\gamma'$ in $G$
\end{prop}
\begin{proof} Obviously by Lemma \ref{lem:f2-hnn} $\gamma$ and
  $\gamma'$ are not almost conjugate in $G$. If $\alpha$ is almost
  conjugate to either $\gamma$ or $\gamma'$ in $G$ then this is still
  true after performing a balancing fold to get
  (\ref{eqn:balanced}) and writing $\gamma'$ again instead of
  $\hat{\gamma}'$ and writing $G$ again instead of $\hat{G}$.

  So we may now assume that $H' = \bk{G,t \mid \gamma^t = \gamma'}$
  with $G = \bk{\gamma}*\bk{\gamma'}$ and that w.l.o.g.  $\alpha$ and
  $\gamma$ are almost conjugate in $G$. This means that there are some
  $\eta, g \in G$ such that $\bk{\alpha^g}, \bk{\gamma} \leq
  \bk{\eta}$.  On the other hand $\gamma$ is not a proper power in $G$
  so $\bk{\gamma} = \bk{\eta}$ which implies that $\bk{\alpha^g} \leq
  \bk{\gamma}$ so, conjugating boundary monomorphisms, we get a
  relative presentation\[ \FRS = \relpres{\bk{\F', G, t \mid t \gamma t
      = \gamma'}}{\bk{\alpha} = \F' \cap G, \gamma, \gamma' \in G}\]
  with $G = \bk{\gamma'}*\bk{\gamma}$ and $\bk{\alpha} \leq
  \bk{\gamma}$ which means that we can rewrite this relative
  presentation using a Tietze transformation as \[\FRS
  = \bk{\F',\gamma,t\mid}\] with $\gamma^m = \alpha \in F$ for some $m \in
  \Z$. This mean that $\FRS$ is freely decomposable modulo $F$-- contradiction.
\end{proof}

The same argument yields:

\begin{prop}\label{prop:2-3-D-noconj} Suppose that $\FRSP$ splits
  as \[F*_\brakett{\alpha} H*_\bk{\beta}Ab(\beta,t) \] with $H$ free of
  rank 2 and suppose moreover that $H$ splits further as an HNN
  extension \[H=\brakett{G,t \mid t^\mo \gamma^n t = \gamma'}\] modulo
  $\alpha$ and $\beta$, then we cannot have that $\alpha$, $\beta$ and
  $\gamma$ are almost conjugate in $H$.
\end{prop}

\begin{lem}\label{lem:QH-possibility}
  Suppose that $\FRS$ has a $QH$ subgroup $Q$ then its JSJ must be of
  the form: \begin{equation}\label{eqn:QH-possibility}
    \FRS=\F*_\bk{q}H_1*_\bk{a_2}\ldots *_\bk{a_k}H_k*_\bk{p}Q
\end{equation}
\end{lem}
\begin{proof}
  Consider the quotient $\FRS \rightarrow \FRS/R = K$ obtained by
  killing the vertex group $\F$ as well as all the edge groups. On one
  hand $K$ is generated by the image of $\bk{x,y}$ on the other hand
  it contains a closed surface group $\ol{Q}$ as a free factor. The
  only possibility is $\ol{Q} = \Z\oplus \Z$, and we immediately see
  that the underlying graph of the JSJ must be simply connected. It
  moreover follows that $\ol{Q} = K$, hence there are no other QH
  subgroups.

  Now a non-cyclic fully residually free group $L$ is either free
  abelian or maps onto a free group of rank 2. This means that for all
  $\gamma \in L$ we have that the quotient $L/\tr{ncl}(\gamma)$ is
  non-trivial. So if $K = \ol{Q}$ then JSJ of $\FRS$ must be as in
  (\ref{eqn:QH-possibility}).
\end{proof}

\begin{proof}[Proof of Corollary \ref{cor:QH}] Let $\FRS$ have a MQH
  subgroup $Q$.  Propositions \ref{prop:1e-h-3gen},
  \ref{prop:1e-2v-3genH} and Lemma \ref{lem:QH-possibility} together
  imply that the JSJ of $\FRS$ cannot have an abelian vertex
  group. And must either be $\F*_\bk{\alpha}Q$, with $\F$ a rank 1
  centralizer extension of $F$, or $F*_\bk{\alpha}Q$. The latter case
  implies the result. Suppose the towards a contradiction former case
  i.e. 2. of Proposition \ref{prop:1e-2v-3genH} held. Then $\x,\y$
  freely generate $Q$ and so $\alpha$ must be in the commutator
  subgroup of $Q$. Now $\alpha$ can be written as a word in
  $\x,\y$ with exponent sum 0 in $\x,\y$. It follows that $\F$ is
  generated by elements of exponent sum 0 in $\x,\y$ and therefore
  cannot contain a non-cyclic abelian subgroup by Lemma
  \ref{lem:exp-sum-zero-abelian} -- contradiction.
\end{proof}

\begin{cor}\label{cor:2v-1e}If $\FRS$ is freely indecomposable and the maximal abelian
  collapse of its cyclic JSJ decomposition modulo $F$ is a free
  product with amalgamation. Then all the possibilities for the JSJ of
  $\FRS$ are to be found in the descriptions given Sections
  \ref{sec:allfree}, \ref{sec:abvertgr}, and \ref{sec:ICEICE}.
\end{cor}

\subsection{The two edge case}\label{sec:MAC-2e}

We now consider the case where the maximal abelian collapse of $\FRS$ has
underlying graph: \begin{equation}\label{eqn:2e-graph}X=\xymatrix{v \bullet
  \ar@{-}@/_/[r]_{f} \ar@{-}@/^/[r]^e & \bullet u} \end{equation} to which we
give the relative presentation:
\begin{equation}\label{eqn:2vertices-2edges}
  \relpres{\bk{\Fh,H,t \mid B^t = C }}{B \leq \Fh, C\leq H, A = \Fh \cap H}
\end{equation} 
Where $F\leq \Fh=X_u, H = X_v$ and $A,B,C$ are maximal abelian and
conjugacy separated in their vertex groups. Throughout this section
the groups $\Fh, H, A, B, C$ are as above.

\subsubsection{Arranging so that either $\x \in H$ or $\x = t$}

\begin{lem}\label{lem:x-in-F-or-H} Let $\FRS$ be freely indecomposable
  modulo $F$ and have a maximal abelian collapse
  (\ref{eqn:2vertices-2edges}). After Weidmann-Nielsen normalization
  on $(F,\x,\z)$ modulo $F$ we can arrange; conjugating boundary
  monomorphisms if necessary; so that $\x$ either lies in $\Fh \cup H$
  or $\x=t$.
\end{lem}
\begin{proof} We first observe that $\FRS$ cannot be generated by
  elliptic elements w.r.t.the splitting (\ref{eqn:2vertices-2edges}).
  We apply Theorem \ref{thm:weidmann-nielsen} to the marked generating
  set $(F;\{\x,\z\})$. Let $T$ be the Bass-Serre tree of this
  splitting.  Let $v_0\in T$ be the vertex fixed by $F\leq\FRS$.

  Suppose that $T_F$ is a point. Then after Weidmann-Nielsen
  normalization, $\x$ must be brought to an elliptic element. We can
  then arrange $T_{\bk{\x}} \cap T_F \neq \emptyset$ or $\z T_F \cap
  T_F \neq \emptyset$, since $\y$ cannot also be elliptic we must have
  $T_{\bk{\x}} \cap T_F \neq \emptyset \Rightarrow \x \in \Fh$. Suppose
  now that $T_F$ is not a point, then after Weidmann-Nielsen
  normalization we can get either:
  
  {\bf Case I:} $T_F\cap \x T_F\neq \emptyset$. Then we can apply Lemma
  \ref{lem:make-elliptic} and get that $\x$ is either in $H$ or
  $\x=th, h \in H$.

  {\bf Case II:} $\x$ is elliptic. Then we use Theorem
  \ref{thm:weidmann-nielsen} on the marked set $(F,\x;\{\z\})$. $\z$
  cannot also be elliptic.

  If $T_F\cap T_{\bk{\x}} \neq \emptyset$ then if $v_0 \in
  T_{\bk{\x}}$ then we can assume that $\x \in \Fh$. If $\x$ fixes a
  vertex $w'$ adjacent to $v_0$, then we can assume that $\x \in
  \stab(w')$. By Lemma \ref{lem:make-elliptic} we either
  have that $\x$ can be brought into $H$ or $tHt^\mo$, after possibly
  changing the relative presentation, the result follows.

  If $T_F \cap \z T_F\neq \emptyset$ then as before we can arrange to
  that $\z = th$ and interchanging $\x$ and $\z$ the result will
  follow. The remaining case is $T_{\bk{\x}} \cap \z
  T_{\bk{\x}}\neq\emptyset$. We note, however, that $T_{\bk{\x}}$
  intersects at most one $\FRS$-edge orbit, but since $\z$ have
  exponent sum 1 in the stable letter, we find that the path $\rho$
  connecting $T_{\bk{\x}}$ and $\z T_{\bk{\x}}$ must intersect both $\FRS$-edge
  orbits. It follows that $T_{\bk{\x}} \cap \z T_{\bk{\x}} =
  \emptyset$.  So we must have $T_F \cap \z T_F \neq \emptyset$, and
  the result follows from Lemma \ref{lem:make-elliptic}.
\end{proof}

\begin{lem}\label{lem:free-dec-H}
  If the vertex group $H$ in the splitting
  (\ref{eqn:2vertices-2edges}) is generated by conjugates of its
  boundary subgroups, i.e. $H = \bra A^{h_1},C^{h_2} \kett, ~h_i \in H$,
  then $\FRS$ is freely decomposable modulo $F$.
\end{lem}
\begin{proof}
  W.l.o.g. by conjugating boundary subgroups if necessary we may
  assume, $h_1,h_2 =1$. By Lemma \ref{lem:abelian-free-prod}, $H =
  A*C$ and we have the relative presentation $\bk{\Fh, H, t | t^\mo B t
    = C}$ with $A, B\leq \Fh$. Using Tietze transformations we can
  rewrite this as $ \Fh * \bra t \kett$.
\end{proof}

\begin{lem}\label{lem:2e-xinF}
  Suppose that $\FRS$ has a maximal abelian collapse
  (\ref{eqn:2vertices-2edges}). If $\FRS$ is generated as
  $\bk{F,\x,\y}$ with $\x \in \Fh$, then $\FRS$ is freely decomposable
  modulo $F$.
\end{lem}
\begin{proof} Consider the $\G(X)$-graph $\B$ obtained by attaching a
  $\y$-loop to the vertex $v$ where $\B_v = \bk{F,\x}$. All other
  vertices in $\B$ have trivial $\B$-vertex group. In the folding process
  F4 collapses are impossible since the final graph should
  have a cycle. By Lemma \ref{lem:long-range-crit} we can perform our
  folding sequence without using transmissions as long as the
  underlying graph is not the graph with two vertices $u,v$, two edges
  $e,f$, and one cycle. When we do reach this point all that remains to be
  done to get a folded graph is to do transmissions. By hypothesis the
  boundary subgroups $C, A \leq H$ are conjugacy separated and so the
  first transmission from $v$ to $u$ through $e$ cannot be followed by
  a transmission form $u$ to $v$ through $f$. It therefore follows
  that after the first two transmissions $\B_v =\bk{C',A'}$, with
  $C',A' \leq C,A$ respectively. By Lemma \ref{lem:abelian-free-prod}, $\B_v
  = C'*A'$ and so no further transmissions are possible from $u$ to
  $v$.  $\B$ is therefore folded and free decomposability now follows
  from Lemma \ref{lem:free-dec-H}.\end{proof}

\subsubsection{Arranging so that w.l.o.g. $\x =t$ and $\y \in H$.}

\begin{lem} Suppose that $\FRS$ has a maximal abelian collapse
  (\ref{eqn:2vertices-2edges}) and is freely indecomposable modulo
  $F$. If $\FRS$ is generated as $\bk{F,\x,\y}$ with $\x \in H$, then,
  conjugating boundary monomorphisms if necessary, we can arrange so
  that $\FRS$ is generated as $\bk{F,\x,\y'}$ with $\z'=ath, h \in H,
  a \in \Fh$.
\end{lem}
\begin{proof}We first note that if $\x$ is conjugate into an edge
  group, then we can assume $\x \in \Fh$, which by Lemma
  \ref{lem:2e-xinF} leads to a contradiction.  The hypotheses imply
  that $\FRS$ is the fundamental group of a $\G(X)$-graph $\B$
  obtained by taking an edge labelled, say $(1,e,1)$, with endpoints
  $v$ and $u$, attaching the $\z$-loop $\L(\z,v)$, and setting $\B_v =
  F,\B_u = \brakett{\x}$.
  
  We start our folding process using only moves A0-A3, F1, F4, L1,
  S1. Note that if a F4 collapse occurs, then the underlying graph
  will be simply connected, which is impossible. Any cycle must have
  even length and by Lemma \ref{lem:long-range-crit}, as long as there
  are more than four vertices we can continue our folding process
  while avoiding transmissions.

  Suppose that $\B$ has only four vertices, noting that this must fold
  to a graph like (\ref{eqn:2e-graph}) using F1 moves we see
  (exchanging the labels $e$ and $f$, if necessary) that the only
  possibilities after doing S1 moves are:
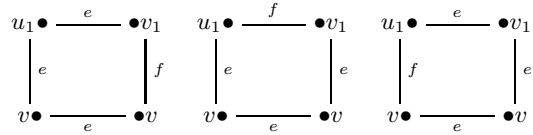
\[ \xymatrix{u_1 \bullet
    \ar@{-}[d]^e\ar@{-}[r]^e & \bullet v_1 \ar@{-}[d]^f \\ v \bullet
    \ar@{-}[r]_e & \bullet v}~~ \xymatrix{u_1 \bullet
    \ar@{-}[d]^e\ar@{-}[r]^f & \bullet v_1 \ar@{-}[d]^e \\ v \bullet
    \ar@{-}[r]_e & \bullet v} ~~\xymatrix{u_1 \bullet
    \ar@{-}[d]^f\ar@{-}[r]^e & \bullet v_1 \ar@{-}[d]^e \\ v \bullet
    \ar@{-}[r]_e & \bullet v}\] where the edges are labelled by their
  image in $X$ via the map $[~]$ (recall Definition
  \ref{defn:G(A)-graph}.) If we consider all possible sequences of
  transmissions we see that either there is a cancellable path at
  $u_1$ or $v_1$ or that $\B_{u_1}$ and $\B_{v_1}$ are contained in
  conjugates of their edge groups. In either case $\B_u,\B_v$ remain
  unchanged and we can make an F1 fold without transmissions.

  Suppose now that $\B$ has only three vertices, then the only
  possibilities after doing S1 moves are 
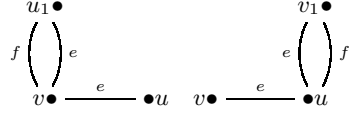
\[\xymatrix{u_1\bullet
    \ar@{-}@/_/[d]_f \ar@{-}@/^/[d]^e & \\ v \bullet \ar@{-}[r]^e &
    \bullet u}~\xymatrix{ &v_1\bullet \ar@{-}@/_/[d]_e
    \ar@{-}@/^/[d]^f \\ v \bullet \ar@{-}[r]^e & \bullet u} \] with
  $\B_{u_1}$ or $\B_{v_1}$ trivial. The only folding that can occur is
  an F1 fold at $v$ or $u$. By Lemma \ref{lem:no-conjugator}
  $\B_u$ or $\B_v$ can only be changed by transmissions through the
  edge between $u$ and $v$ it follows that we only need to make
  A0-A4,L1 moves before our F1 fold.

  $\B$ can therefore be brought to a graph of the form:
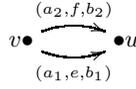
\[\xymatrix{v
    \bullet \ar@/_/[r]_{(a_1,e,b_1)} \ar@/^/[r]^{(a_2,f,b_2)}& \bullet
    u}\] with $\B_v=F$ and $\B_u=\brakett{\x}$. Moreover we see that
  it is possible to leave the label $(1,e,1)$ of the edge between $u$
  and $v$ unchanged throughout the folding process, so we may assume
  that $a_1=b_1=1$.  $\FRS$ is therefore generated by $F, \x \in H$
  and $\z' = a_2 f, b_2, e^\mo$ with $a_2 \in \Fh$ and $b_2 \in H$,
  i.e. in the relative presentation we can write $\y' = a_2tb_2$.
\end{proof}

\begin{lem} Suppose that $\FRS$ has a maximal abelian collapse
  (\ref{eqn:2vertices-2edges}) and is freely indecomposable modulo
  $F$. If $\FRS$ is generated as $\bk{F,\x,\y}$ with $\x = t$, then we
  can arrange so that $\FRS$ is generated as $\bk{F,\x,\y'}$ with $\z'
  \in H$.
\end{lem}
\begin{proof} The hypotheses imply that $\FRS$ is the fundamental
  group of a $\G(X)$-graph $\B$ obtained by taking two edges with
  labels $(1,e,1)$ and $(1,f,1)$ with common endpoints $v$ and $u$,
  setting $\B_v=F, \B_u = \{1\}$, and attaching the $\z$-loop
  $\L(\z,v)$.  We start our folding process using only moves
  A0-A3,F1,F4,L1,S1, but avoiding transmissions.

  {\bf Case I:} Suppose we were able to bring $\B$ to a graph with two
  vertices and two edges with $\B_v = \brakett{F,\z'},\B_u = \{1\}$ or $\B_v
  = F, \B_u = \{\z'\}$. To get a folded graph, all that remains are
  transmissions, but note that in the former case, Lemma
  \ref{lem:free-dec-H} will imply free decomposability modulo $F$, in
  the latter case the result follows.

  {\bf Case II:} Suppose first that a F4 collapse occurred then $\B$ has
  one cycle and two non-trivial $\B$-vertex groups. By Lemma
  \ref{lem:long-range-crit} unless we have one of the two following
  possibilities: 
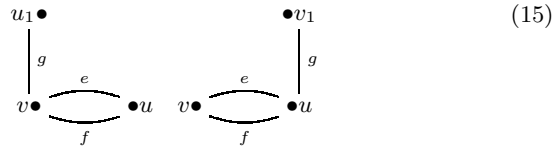
\begin{equation}\label{eqn:2e-collapsed}
    \xymatrix{u_1\bullet \ar@{-}[d]^g & \\ v \bullet \ar@{-}@/^/[r]^e
      \ar@{-}@/_/[r]_f & \bullet u}~\xymatrix{&\bullet v_1
      \ar@{-}[d]^g \\ v \bullet \ar@{-}@/^/[r]^e \ar@{-}@/_/[r]_f &
      \bullet u} \end{equation} we can continue folding without using
  transmissions. Here $\B_{u_1}$ or $\B_{v_1}$ is cyclic. The next
  fold is an F1 fold at $u$ or $v$, or shaving off $u_1$ or $v_1$. By
  Lemma \ref{lem:abelian-free-prod} unless $u_1$ or $v_1$ can be
  shaved off, there can be no transmissions to $u_1$ or $v_1$ and then
  back again through $g$ it therefore follows that prior to the F1
  fold no transmissions through the edge $g$ are needed so we can
  change its label with an A2 simple adjustment or an L1 long range
  transmission. This brings us to Case I. 

  {\bf Case III:} Suppose that no collapses occurred. By Lemma
  \ref{lem:long-range-crit} $\B$ has at most 4 vertices.  The only
  possibility with 4 vertices is something of the form:
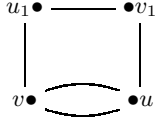
\[
  \xymatrix{u_1\bullet \ar@{-}[d] \ar@{-}[r]& \bullet v_1\ar@{-}[d]\\
    v \bullet \ar@{-}@/^/[r] \ar@{-}@/_/[r] & \bullet u}\] with $\B_u,
  \B_{u_1}, \B_{v_1}$ trivial. If there are no cancellable paths, then
  the next fold is of type F1 at $u$ or $v$, and no transmissions are
  needed. If the fold is at $v_1$ or $u_1$ Lemma
  \ref{lem:long-range-crit} ensures that no transmissions are needed
  to make the fold.

  {\bf Case III.I} Suppose now that after shaving $\B$ has three vertices then the
  possibilities are:
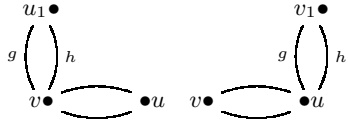
\[\xymatrix{u_1\bullet \ar@{-}@/_/[d]_g
    \ar@{-}@/^/[d]^h & \\ v \bullet \ar@{-}@/_/[r] \ar@{-}@/^/[r] &
    \bullet u}~\xymatrix{ &v_1\bullet \ar@{-}@/_/[d]_g \ar@{-}@/^/[d]^h \\
    v \bullet \ar@{-}@/_/[r] \ar@{-}@/^/[r] & \bullet u}\] with $\B_u,
  \B_{u_1}, \B_{v_1}$ trivial.

  {\bf Case III.I.I:} If $[g] \neq [h]$ then since the edge groups are
  conjugacy separated and by Lemma \ref{lem:abelian-free-prod} after
  any sequence of transmissions there can be no transmissions from
  $u_1$ or $v_1$ back to $u$ or $v$ respectively. The next fold is an
  F1 fold and can therefore be done using an A2 simple adjustment or
  an L1 long range adjustment to change the label of $g$ or $h$. This
  brings us to Case III.II.

  {\bf Case III.I.II:} If $[g] = [h]$ then either we can F4 collapse at $u_1$ or $v_1$
  towards $v$ or $u$ respectively which after shaving off $u_1$ or $v_1$
  reduces to the Case I. Otherwise by Lemma
  \ref{lem:no-conjugator} after any sequence of transmissions there
  can be no transmissions from $u_1$ or $v_1$ back to $u$ or $v$ respectively
  and no F4 collapse at $u_1$ or $v_1$. The subsequent fold can
  therefore be made without any transmissions as in the previous
  paragraph. A collapse at this point reduces to Case II, otherwise we
  are in Case III.II.

  {\bf Case III.II:} Suppose now that $\B$ has two vertices and three
  edges, the possibilities are
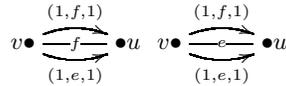
\[\xymatrix{v\bullet
    \ar@/^/[r]^{(1,f,1)} \ar@{-}[r]|{f} \ar@/_/[r]_{(1,e,1)} & \bullet
    u} \xymatrix{v\bullet \ar@/^/[r]^{(1,f,1)} \ar@{-}[r]|e
    \ar@/_/[r]_{(1,e,1)} & \bullet u}\] where $\B_u = \{1\}$ and the
  remaining edge is marked only by its image in $X$ via map
  $[~]$. Note that these cases are symmetric. Suppose the middle edge
  has label $(a,e,b)$ with $a \in \Fh$ and $b \in H$. First note that
  if it is possible to transmit through both $e$-type edges from $v$
  to $u$, because $F$ has property $CC$ this means that there is some
  $f \in F$ such that $fa \in i_e(X_e)$ which means that after an A2
  simple adjustment and an A1 Bass-Serre move we can make a F4
  collapse towards $u$ and the result follows. \begin{itemize}
  \item[$(\dagger)$] We may therefore assume that there is no $f \in
    F$ such that $fa \in i_e(X_e)$.
\end{itemize}
  If only one transmission from $v$ to $u$ is possible before the F4
  collapse, then we could use an L1 long range adjustment instead, and
  reduce to the Case I. The final remaining possibility is that
  there are transmissions from $v$ to $u$ through the edges labeled
  $(1,e,1)$ and $(1,f,1)$. In particular we can assume\begin{itemize}
  \item[$(\ddagger)$] Some conjugate of $X_e$ centralizes $\alpha \in
    F$.
  \end{itemize}
  Let $\B_u = \bk{\alpha,\beta}$ be the $\B$-vertex group after these
  transmissions. If $b \in \B_u$ and we can make a simple adjustment
  on the label of the edge labeled $(a,e,b)$ and then do an F4
  collapse, then we could have used an L1 long range adjustment
  instead, and reduce to Case I.

  We may therefore assume that there is a transmission from $u$
  through the edge labeled $(a,e,b)$ so that $\B_v$ is now $\bk{F,a A
    a^\mo}$ where $A \leq i_e(X_e)$. Since we $A$ centralizes some
  conjugate of an element of $F$ and by $(\dagger),(\ddagger)$, we can
  apply Corollary \ref{cor:F-and-abelian} to obtain $\bk{F,a A a^\mo}
  = F * a A a^\mo$, $a \not\in \B_v$ so there can be no collapse at
  $v$ and by the free product structure there can be no more
  transmissions. The graph is therefore already folded, contradicting the fact
  that $\FRS = \pi_1(\G(X))$ as in (\ref{eqn:2e-graph}).
\end{proof}

All these lemmas combine to give:

\begin{prop}\label{prop:two-edges}
  If $\FRS$ is freely indecomposable modulo $F$ and has a maximal
  abelian collapse (\ref{eqn:2vertices-2edges}), then $\FRS$ is
  generated by $F$, the stable letter $t$, and some element $\z \in
  H$.
\end{prop}

\subsubsection{Recovering the original cyclic splitting from the
  maximal abelian collapse}

The next proposition enables us to revert to a cyclic splitting.

\begin{prop}\label{prop:2e-h-3gen} Suppose that $\FRS$ is freely
  indecomposable modulo $F$ and has a maximal abelian collapse
  (\ref{eqn:2vertices-2edges}), then $\FRS$ admits one of the two
  possible cyclic splittings:
  \begin{enumerate}
  \item \[ \FRS = \relpres{\bk{F,H',\x \mid \beta^\x = \beta'}}{\beta
      \in F, \beta' \in H', \bk{\alpha} = F \cap H'}\]where
    $H'=\brakett{\alpha,\beta',\z}$ is a three generated fully
    residually free group.
  \item \[\FRS = \relpres{\bk{\F',H',\x \mid \beta^\x = \beta'}}{\beta
      \in \F', \beta' \in H', \bk{\alpha} = \F \cap H'}\] where $\F'$ is
    a rank 1 free extension of a centralizer of $F$ and $H'$ is
    generated by $\alpha,\z$. Moreover $H'$ may not split further as
    an HNN extension.
  \end{enumerate}
\end{prop}

\begin{proof} Let $\Fh,A,B,C,H,t$ be as in
  (\ref{eqn:2vertices-2edges}). By Proposition \ref{prop:two-edges} we
  can assume that $\x=t$ and $\z \in H$. We can always assume that
  $F\cap A \neq \{1\}$ (otherwise we could derive free decomposability
  modulo $F$.)

  \emph{Suppose first that $F\cap A=\brakett{\alpha}$ and $F\cap B =
    \{1\}$.} To ensure free indecomposability modulo $F$ we need there
  to be some $\gamma \in \brakett{\alpha,\z}$ such that $\x \gamma
  \x^\mo = \beta \in B \leq\Fh$. Now by Theorem \ref{thm:onevariable}
  if $\brakett{F,\beta}\neq F*\brakett{\beta}$ then we must have
  $\brakett{F,\beta} = \brakett{F,t | [p,t]=1}$ for some $p \in F$. If
  $p$ is not conjugate to $\alpha$ in $F$ then $\FRS$ has a cyclic
  splitting as in item 2. If $p$ and $\alpha^\pmo$ are conjugate in $F$,
  then we can assume that $\alpha = p$, so then the group $A$ in
  (\ref{eqn:2vertices-2edges}) is non-cyclic abelian of rank 2. We
  study the maximal abelian subgroup $C\leq H$ we already had that
  $\gamma(\z,\alpha) \in C$ if $C\leq H$ is not cyclic then there must
  be some $\gamma_1 \in \brakett{\z,A}$ such that $\gamma$ and
  $\gamma_1$ do not lie in a common cyclic subgroup and which
  satisfies the relation
  \begin{equation}\label{eqn:backwash} [\gamma,\gamma_1]=1
  \end{equation}
  however by Lemma \ref{lem:abelian-free-prod} we have
  \begin{equation}\label{eqn:backwashmore}\brakett{A,\z} =
    A*\brakett{\z}\end{equation}
  which means that (\ref{eqn:backwash}) is impossible. It follows
  that $\Fh=\brakett{F,\beta}$. This gives the cyclic
  splitting:
  \[\FRS = \relpres{\bk{\Fh,H,\x \mid \beta^\x = \gamma}}{\beta \in
    \Fh, \gamma \in H, \bk{\alpha} = \Fh \cap H}
  \] with $H=\bk{\alpha,\y}$.
  
  Suppose now towards a contradiction that $H$ split further as an
  HNN extension: \[ H =\bk{K,t \mid \delta^t=\delta'}; \delta,\delta'
  \in K\] modulo $\alpha, \gamma$, then we have \[\FRS =
  \relpres{\bk{\Fh,K, \x,t \mid \beta^\x = \gamma, \delta^t =
      \delta'}}{\beta \in \Fh, \gamma,\delta,\delta' \in K,
    \bk{\alpha} = \Fh \cap K}\] Then we can collapse this splitting to
  a double HNN and by Corollary \ref{cor:no-ab} $\Fh$ cannot contain any
  non-cyclic abelian subgroups -- contradiction.
  
  \emph{Suppose now that $F\cap A=\brakett{\alpha}$ and $F\cap B =
    \brakett{\beta}$.} If both $A$ and $B$ are cyclic we are done: by
  Proposition \ref{prop:two-edges} $H$ is generated by
  three elements. We therefore assume w.l.o.g. that $A$ isn't cyclic.
  First note that by Lemma \ref{lem:F-and-abelian}, $\bk{F,A}$ has the
  JSJ $F*_\bk{\alpha}A$. Suppose first that the $B$ isn't cyclic. We
  have a surjection \[A*_\bk{\alpha}F*_\bk{\beta}B \rightarrow
  \bk{F,A,B}\] which is injective on $F*_\bk{\beta}B$ as well. Suppose
  this map wasn't injective, then some element $w$ lies in the kernel,
  moreover $A*_\bk{\alpha}F*_\bk{\beta}B$ should not have any
  essential cyclic splittings modulo $w,F$. On the other hand
  $\bk{F,A,B}$ should have a non-trivial JSJ modulo $F$ but the
  triviality of the image $w$ implies that $\bk{F,A,B}$ has only one
  vertex group -- contradiction.

  Hence, whether or not $B$ is cyclic, we always have that $\bk{F,A,B}
  = A*_\bk{\alpha} F *_\bk{\beta} B$. This means we have the cyclic
  splitting\[\FRS = \relpres{\bk{F,H',t \mid \beta^t = \gamma}}{ \beta
    \in F, \gamma \in C\leq H', \bk{\alpha} = F \cap H'}\] where $H' =
  \bk{H,A,C}$. Now by Proposition \ref{prop:two-edges}, even with the
  new splitting, we still have that $\x=t$ and that $\y \in H\leq H'$
  so considering a folding sequence starting at this point we have
  that $F$ is a full vertex group and all that remains are transmissions
  from $F$ to $H'$ to get $\FRS$. It follows that $H' =
  \bk{\y',\alpha,\gamma}$, so it's 3-generated.
\end{proof}

\begin{prop}\label{prop:2-3-A-noconj} Suppose that $\FRS$ splits
  as \[\relpres{\bk{F,H,t \mid \beta^t = \gamma}}{\beta \in F, \gamma
    \in H, \bk{\alpha} = F \cap H}\] Then $\alpha$ and $\gamma$ cannot
  be almost conjugate $H$. Moreover $\alpha$ and $\beta$ cannot be
  almost conjugate in $F$.
\end{prop}
\begin{proof} 
  Suppose towards a contradiction that $\alpha, \gamma$ were almost
  conjugate in $H$, then after conjugating boundary monomorphisms,
  making folding moves (as in Definition \ref{defn:moves}) which keep
  the splitting cyclic and making a sliding move we would get a
  splitting:\[\relpres{\bk{F,H,t \mid \beta^t = \alpha}}{\alpha,
    \beta\in F, \bk{\alpha} = F \cap H}\] whose maximal abelian
  collapse is as considered in Section \ref{sec:MAC-1e} --
  contradiction. Similarly, if $\alpha$ and $\beta$ were almost
  conjugate in $F$, we can similarly derive a contradiction to the
  choice of maximal abelian collapse.
\end{proof}
  
\begin{prop}
  Suppose we have the splitting \[ \FRS = \relpres{\bk{F,H', \y \mid
      \beta^\y = \gamma}}{\beta \in F, \gamma \in H',
    \bk{\alpha}=F\cap H'}\] where $H' = H*_\bk{\delta}Ab(\delta,r)$ is
  a rank 1 free extension of a centralizer of the free group $H$ of rank
  2. Then this splitting can be refined to \[\FRS =
  \relpres{\bk{F,H,Ab(\delta,r), \y \mid \beta^\y = \gamma}}{\beta \in
    F, \gamma \in H, \bk{\alpha}=F\cap H, \bk{\delta} = H \cap
    Ab(\delta,r)}\] such that $\delta, \alpha,\gamma$ are not almost
  conjugate \emph{in} $H$. This cyclic splitting is the JSJ of $\FRS$.
\end{prop}
\begin{proof}
  If $\alpha$ and $\gamma$ are almost conjugate, then we can derive a
  contradiction arguing as in the proof of  Proposition \ref{prop:2-3-A-noconj}. 

  If $H$ could split further as an HNN extension modulo
  $\alpha,\gamma,\delta$ then we could apply Corollary \ref{cor:no-ab}
  contradicting the existence of a non-cyclic abelian subgroup of
  $\FRS$. It follows that the cyclic JSJ splitting of $H'$ modulo
  $\alpha,\gamma$ either consists of one vertex group $H'$ or is an
  amalgam $H*_\bk{u}Ab(\delta,r)$. In the former case since
  $\alpha,\gamma$ are conjugable into $F$ their images will be
  elliptic in in every term of a strict resolution of $\FRS$ which
  implies that  $H*_\bk{u}Ab(\delta,r)$ is free -- contradiction.

  Suppose now that either $\alpha$ or $\gamma$, say w.l.o.g. $\gamma$,
  is conjugable in $H'$ into $Ab(\delta,r)$ but that $\gamma$ was not
  almost conjugate to $\delta$. We can conjugate boundary
  monomorphisms so that $\gamma \in Ab(\delta,r)$ and replace
  $Ab(\delta,r)$ with $\bk{\gamma,\delta}$. This gives an $F$-subgroup
  $G \leq \FRS$ that has a splitting with vertex groups $F,
  \bk{\gamma,\delta}$, and $H$. We that in every term of a strict resolution
  of $G$ the images of the elements $\alpha,\beta,\gamma,\delta$ are
  forced to be elliptic, this means that $H$ is always mapped
  monomorphically, so $\gamma$ and $\delta$ will always be
  conjugable into some non-abelian vertex group, this means that
  $\bk{\gamma,\delta}$ will never have an exposed direct summand (see
  Definition \ref{defn:exposed}) contradicting Lemma \ref{lem:abelian-exposed}.
\end{proof}

 \begin{cor}\label{cor:2v-2e} If $\FRS$ is freely indecomposable has a
   maximal abelian collapse (\ref{eqn:2vertices-2edges}). Then all the possibilities for the JSJ of
  $\FRS$ are to be found in the descriptions given Sections
  \ref{sec:allfree}, \ref{sec:abvertgr}, and \ref{sec:ICEICE}. 
\end{cor}

\subsection{The three edge case}\label{sec:MAC-3e}
We now consider the case where the maximal abelian collapse of $\FRS$
has underlying graph: \begin{equation}\label{eqn:3e-G(X)}X=\xymatrix{v \bullet
  \ar@{-}@/^/[r]^g \ar@{-}[r]|e \ar@{-}@/_/[r]_f & \bullet u}\end{equation} to which we
give the relative presentation:
\begin{equation}\label{eqn:three-edge-groups}
 \relpres{\bk{\Fh,H,s,t \mid A^s=B, D^t = E }}{A,D \leq \Fh, B,E \leq H,
   C = \Fh\cap H}
\end{equation}
Where $F\leq \Fh = X_v, H = X_u$ and $A,B,C,D,E$ are maximal abelian
and conjugacy separated in their vertex groups. Note that
by Corollary \ref{cor:no-ab}, $\FRS$ cannot contain any non-cyclic
abelian subgroups, in particular the subgroups $A,B,C,D,E$ must all be
cyclic.

\begin{lem}\label{lem:3e-x}Let $\FRS$ be freely indecomposable modulo
  $F$ with a maximal abelian collapse
  (\ref{eqn:three-edge-groups}). After Weidmann-Nielsen normalization
  on $(F,\x,\z)$ modulo $F$ we can arrange; conjugating boundary
  monomorphisms if necessary so that $\x = t$.
\end{lem}
\begin{proof}
  Since we are assuming free indecomposability of $\FRS$ modulo $F$,
  we can apply Theorem \ref{thm:weidmann-nielsen} to the marked
  generating set $(F;\{\x,\y\})$. Let $T$ be the Bass-Serre tree
  corresponding to the splitting (\ref{eqn:3e-G(X)}).  We note that
  neither $\x$ nor $\z$ can be brought to elliptic elements w.r.t. the
  splitting (\ref{eqn:3e-G(X)}). Let $v_0 \in T$ be the vertex fixed
  by $F$.  W.l.o.g. after  Weidmann-Nielsen normalization we have $T_F
  \cap \x T_F \neq \emptyset$. 1-acylindricity implies that $d(v_0,\x
  v_0) = 2$. It follows w.l.o.g. that $\x$ as a $\G(X)$-path is of the form
  $a_1,f,b_1,g^\mo,a_2$ where $b_1 \in H, a_1,a_2 \in \Fh$, by Lemma
  \ref {lem:make-elliptic} we can arrange so that $\x = f,b,g^\mo$,
  and conjugating boundary monomorphisms enables us to assume that $\x
  = f,g^\mo=t^\mo$ in terms of the relative presentation
  (\ref{eqn:three-edge-groups}).
\end{proof}
\begin{lem}\label{lem:3e-y} Let $\FRS$ be freely indecomposable modulo
  $F$ with a maximal abelian collapse (\ref{eqn:three-edge-groups})
  and with $\x = t$, then $\FRS$ is generated by $F$, $t$, and $s$.
\end{lem}
\begin{proof} The hypotheses imply that $\FRS$ is the fundamental group
  of a $\G(X)$-graph $\B$ obtained by taking two edges with labels
  $(1,e,1)$ and $(1,f,1)$ with common endpoints $v$ and $u$, setting
  $\B_v=F, \B_u = \{1\}$, and attaching the $\z$-loop $\L(\z,v)$.

  Again we start our adjustment-folding process, using only moves
  A0-A2,L1,F1, S1. F4 collapses are forbidden since they reduce the
  number of cycles in the underlying graph. As long as after S1
  shavings there are strictly more than 4 vertices, we see by Lemma
  \ref{lem:long-range-crit} that we can always perform a folding move
  without using transmissions. It follows that we can bring $\B$ to a
  graph with 4 vertices such that $\B_v=F$ is the only non-trivial
  $\B$-vertex group.

  Interchanging $e$ and $f$ if necessary, and noting that the exponent sum
  $\sigma_s(\z)=1$, w.l.o.g. the only possibilities are:
  
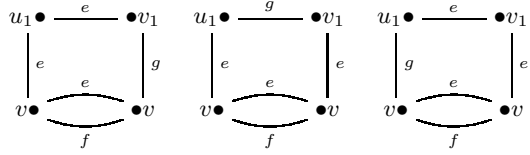
\[ \xymatrix{u_1 \bullet \ar@{-}[d]^e\ar@{-}[r]^e & \bullet v_1
    \ar@{-}[d]^g \\ v \bullet \ar@{-}@/^/[r]^e \ar@{-}@/_/[r]_f &
    \bullet v}~~ \xymatrix{u_1 \bullet \ar@{-}[d]^e\ar@{-}[r]^g &
    \bullet v_1 \ar@{-}[d]^e \\ v \bullet \ar@{-}@/^/[r]^e
    \ar@{-}@/_/[r]_f & \bullet v} ~~\xymatrix{u_1 \bullet
    \ar@{-}[d]^g\ar@{-}[r]^e & \bullet v_1 \ar@{-}[d]^e \\ v \bullet
    \ar@{-}@/^/[r]^e \ar@{-}@/_/[r]_f & \bullet v}\] where the edges
  are marked by their images in $X$ via the map $[~]$. In all three cases we see that
  after applying transmissions the groups $\B_{u_1},\B_{v_1}$ are
  cyclic. Again Lemma \ref{lem:long-range-crit} applies and we can
  continue our adjustment-folding process.

  If there are only three vertices then the only possibilities are:
  
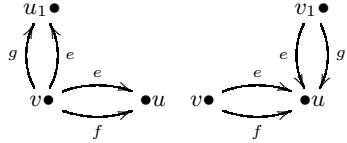
\[\xymatrix{u_1\bullet & \\ v
    \bullet \ar@/_/[r]_f \ar@/^/[r]^e \ar@/_/[u]_e \ar@/^/[u]^g &
    \bullet u} ~\xymatrix{ & v_1\bullet \ar@/^/[d]^g \ar@/_/[d]_e\\ v
    \bullet \ar@/_/[r]_f \ar@/^/[r]^e & \bullet u}\] and the last fold
  is of type F1 at either $u$ or $v$ and in particular no
  transmissions are needed. We get that $\B$ is given by 
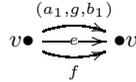
\[\xymatrix{v
    \bullet \ar@/^/[r]^{(a_1,g,b_1)} \ar[r]|e \ar@/_/[r]_f & \bullet
    v}\] Where the edges labelled $e$ and $f$ have labels $(1,e,1)$
  and $(1,f,1)$ respectively. In the end we have that $\FRS$ is
  generated by $F,t$ and some element $\z'$ represented by the $\G(X)$
  path $a_1,g,b_1,e^\mo$ where $b_1 \in H$ and $a_1 \in \Fh$. After
  conjugating boundary monomorphisms, we may assume that $\z' = s$.
\end{proof}

\begin{cor}\label{cor:2v-3e}If $\FRS$ is freely indecomposable and the maximal abelian
  collapse of its cyclic JSJ decomposition modulo $F$ has three
  edges. Then all the JSJ of $\FRS$ is of type C.2. in Section \ref{sec:allfree}.
\end{cor}

\begin{proof} All we need to show is that the vertex groups are $F$
  and a free group of rank 2.

  By the two previous lemmas we have that $\FRS$ is the fundamental
  group of the $\G(X)$-graph \[ \B = 
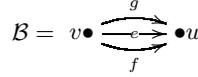
\xymatrix{v \bullet \ar@/^/[r]^g
    \ar[r]|e \ar@/_/[r]_f & \bullet u}\] with $\B_v = F$ and $\B_u=
  \{1\}$. To get a folded graph, all that are needed are
  transmissions. We also saw that the edge groups are cyclic. Suppose
  first that the only possible transmission is from $v$ to $u$ through
  $e$, then by conjugacy separability of the edge groups it is
  impossible for there to be any further transmissions from $u$ back
  to $v$ through the other edges.

  Suppose that now there were transmissions possible only from $v$ to
  $u$ through edges $e$ and $f$. So as not to have free
  decomposability modulo $F$, we must have a transmission from $u$ to
  $v$ through $g$. We note that the boundary subgroups associated to
  the edges $e,f$ must be maximal cyclic because they lie in $F$, it
  then follows that there are no further possible transmissions and
  the graph is folded. In particular we find that $\B_u=\F =
  \brakett{F,\alpha}$ where $\alpha$ is the element transmitted from
  $H$ to $\F$. $\FRS$ is freely indecomposable modulo $F$ only if $\F \neq
  F*\brakett{\alpha}$, but by Theorem \ref{thm:onevariable} the only
  other possibility for $\brakett{F,\alpha}$ is $F*_uAb(u,t)$, which
  is impossible since $\FRS$ has no non-cyclic abelian subgroups.

  It therefore follows that $\F = F$ and $H$ is a free group of rank 2
  generated by its boundary subgroups.
\end{proof}
 
\subsection{The proof of the Proposition \ref{prop:2v-class}}\label{sec:2-nonab-proofs}
If the JSJ of $\FRS$ has more than one non-abelian vertex group then
it falls into the premises of Corollaries \ref{cor:2v-1e},
\ref{cor:2v-2e}, or \ref{cor:2v-3e} so our list of possible JSJs given
in Section \ref{sec:2-nonab} is complete.

\section{When the JSJ of $\FRS$ has one vertex group}\label{sec:1v-classification}
We now consider the situation where $\FRS$ has a cyclic JSJ
decompositions modulo $F$, with only one vertex group. We have the
relative presentations:
\begin{equation}\label{eqn:1v-1e} 
  \bk{\F ,t \mid \beta^t = \beta'}; \beta, \beta'
\in \F 
\end{equation} 

\begin{equation}\label{eqn:1v-2e}
\bk{\F,s ,t \mid \beta^t = \beta', \alpha^s=\alpha'}; \alpha,
\alpha', \beta, \beta' \in \F 
\end{equation}

\subsection{$\F$ is 2-generated modulo $F$.}\label{sec:1v-F-class}
We first need some further auxiliary results.

\begin{lem}\label{lem:1v-eoc}
Let the JSJ of $\FRS$ be either (\ref{eqn:1v-1e}) or
(\ref{eqn:1v-2e}). Then $\beta$ and $\beta'$ cannot be conjugate in $\F$.
\end{lem}
\begin{proof}
  Suppose the contrary. Then $\beta'$ and $\beta$ are conjugate in
  $\F$ so conjugating boundary monomorphisms gives us
  \[\bk{\F,t}= \bk{\F,r \mid [r,\beta] = 1}\] which implies
  that the JSJ of $\FRS$ has an abelian vertex group -- contradiction.
\end{proof}

This Corollary now follows immediately from the fact that $F \leq \FRS$ has
property CC.
\begin{cor}\label{lem:Ftilde-not-F}
  Let the $\FRS$ split as in (\ref{eqn:1v-1e}) or (\ref{eqn:1v-2e})
  then $\F \neq F$.
\end{cor}

\begin{lem}\label{lem:2v-2e-abelian}
  Suppose $\FRS$ splits as a double HNN
  extension:\[\relpres{\bk{\hat{F},t,s \mid \alpha^s = \alpha', \beta^t =
      \beta'}}{\alpha,\alpha',\beta,\beta'\in \F}\] where $\hat{F}$ has no
  cyclic or free splittings modulo
  $F,\alpha,\alpha',\beta,\beta'$. Then w.l.o.g. either $\bk{\alpha}$
  or $\bk{\beta}$ is conjugable into $F$, but not both.
\end{lem}
\begin{proof}
  We may consider $\FRS$ as a double HNN extension. Suppose towards a
  contradiction that neither $\alpha,\alpha'$ nor $\beta,\beta'$ were conjugable
  into $F$ in $\hat{F}$. Then $T_F$ is a point, which means for any
  hyperbolic $g \in \FRS, T_F \cap gT_F = \emptyset$. Since $\FRS$ is
  a double HNN, looking at exponent sums of stable letters, we see it
  must be generated by $F$ and \emph{at least} two hyperbolic
  elements. It now follows from Theorem \ref{thm:weidmann-nielsen} on the
  marked generating set $(F;\{\x,\y\})$ that $\FRS$ is freely
  decomposable modulo $F$ -- contradiction.

  Suppose now towards a contradiction that both $\beta$ and $\alpha$ were
  conjugable into $F$. Then by Corollary \ref{cor:get-split} there is
  a splitting of $\hat{F}$ modulo $\alpha,\beta,\alpha',\beta',F$, with
  either trivial or cyclic edge groups. -- contradiction.
\end{proof}

We can now say something about the vertex groups of the JSJs
(\ref{eqn:1v-1e}) and (\ref{eqn:1v-2e}).

\begin{lem}\label{lem:1v-2e-F} Suppose $\FRS$ has the
  JSJ (\ref{eqn:1v-2e}) then:\begin{enumerate}
    \item $\F$ has no non-cyclic abelian subgroups
    \item One of the edge groups, say $\brakett{\alpha}$, is
      conjugate into $F$.
    \item The elements $\alpha$ and $\beta$ are not conjugate in $\FRS$
    \item The centralizers of $\alpha, \beta$ are cyclic in $\F$. In
      particular after making balancing folds, the splitting is 1-acylindrical.
    \item After balancing folds, where $\F_f$ denotes the resulting
      vertex group, we have $\F_f = \bk{F,\alpha',\beta'}$ with $\alpha'$ conjugable into
      $F$ and $\beta'$ conjugable into $\bk{F,\alpha'}$. In particular
      $\F$ is also 2 generated modulo $F$.
    \end{enumerate}
\end{lem}
\begin{proof}
  By Corollary \ref{cor:no-ab} $\F$ has no non-cyclic abelian
  subgroups. Items 2. and 3. follow from Lemma
  \ref{lem:2v-2e-abelian}. Item 4. now follows from the previous
  items.

  To prove item 5. we first replace the vertex group $\F$ by $\F_f$,
  which is obtained by performing balancing folds. By Lemma
  \ref{lem:adjoin-a-root} if $\F_f$ is 2-generated modulo $F$ so is
  $\F$. We now have a 1-acylindrical splitting and the Bass-Serre tree
  $T$ has two edge orbits. We now use Theorem
  \ref{thm:weidmann-nielsen}. By items 3. and 4. $T_F$ has radius 1
  and contains edges from only one orbit. We also have that $\x$ and
  $\y$ cannot be elliptic. It therefore follows that after
  Weidmann-Nielsen normalization we have w.l.o.g. that $T_F \cap \x
  T_F \neq \emptyset$ and if $\fix(F)=v_0$ then $d(v_0,\x v_0) \leq
  2$. From this we may assume that there are no symbols $t$ in the
  normal form of $\x$ w.r.t. the relative presentation
  (\ref{eqn:1v-2e}) which means that $\x$ is forced to have exponent
  sum 1 in $s$, we therefore have $d(v_0,\x v_0)=1$ so that $\x =
  f_1sf_2, f_i \in \F$. Now conjugating boundary monomorphisms if
  necessary we have w.l.o.g. $\x^\mo \alpha \x =\alpha' \in
  \F$. W.l.o.g. some conjugate of $\beta$ lies in $\bk{F,\alpha'}$,
  otherwise $\bk{F,\x,\y} = \bk{F,\x} *\bk{\y}$, since $\y$ must have
  exponent sum 1 in $t$.

  After Weidmann-Nielsen normalization $T_{\bk{F,\x}} \cap \y
  T_{\bk{F,\x}} \neq \emptyset$. Let $E$ denote the edges in the same
  $\FRS$-orbit as the edges in $\axis(s) \subset T$. The minimal
  invariant trees $T(\bk{F,\x})$ and $\y T(\bk{F,\x})$ do not contain
  edges from $E$. The shortest path $\eta$ between the minimal
  invariant subtrees $T(\bk{F,\x})$ and $\y T(\bk{F,\x})$ therefore
  must contain some edge in $E$, since $\y v_0 \in \y
  T(\bk{F,\x})$. By 1-acylindricity $\eta$ has length at most 2. Let
  $\eta$ have endpoints $v_1 \in T(\bk{F,\x})$ and $v_2\in \y
  T(\bk{F,\x})$. Let $\rho_1,\rho_2 \bk{F,\x}$ be such that $\rho_1
  v_1 = v_0, (\y \rho_2 \y^\mo) \y v_0=v_2$. Then $d(v_0,\rho_1\y\rho_2
  v_0)\leq 2$. Replacing $\y$ by $\rho_1 \y \rho_2$ if necessary we
  have w.l.o.g. either $\y = f_1t f_2, f_i \in \F$ or $\y =
  f_1sf_2tf_3$.

  {\bf Case I:} If $\y = f_1t f_2, f_i \in \F_f$. Then for $u,w \in
  \bk{F,\x}$ the stable letters $t$ cancel in products\[ (f_1t f_2)u
  (f_1t f_2)^\mo \tr{~or~} (f_1t f_2)^\mo w (f_1t f_2) \] if and only
  if either ${}^{f_2}u \in \bk{\beta'}$ or $w^{f_1} \in \bk{\beta}$
  W.l.o.g. we may assume $w \in \bk{\beta}$ and conjugating boundary
  monomorphisms we may assume $\y^\mo w \y \in \bk{\beta'}$. Now either
  $\bk{F,\alpha',\y^\mo w y} = \F_f$ or $\y^\mo w y = (\beta')^m$ and
  $\bk{F,\alpha',(\beta')^m} \cap \bk{\beta} \geq \bk{F,\alpha'} \cap
  \bk{\beta}$. Then we can make another transmission to get $\F_f \geq
  \bk{F,\alpha',(\beta')^{m_1}}$ with $|m_1| < |m|$. Either we have
  equality or we have $\bk{F,\alpha',(\beta')^{m_1}} \cap \bk{\beta}
  \geq \bk{F,\alpha',(\beta')^m}\cap \bk{\beta}$ etc. This cannot go on indefinitely
  and we find finally that $\F_f$ is 2-generated modulo $F$.

  {\bf Case II:} Suppose $\y = f_1sf_2tf_3$. Then $\FRS$ is the
  fundamental group of the $\G(X)$-graph $\B:$
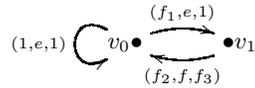
\[ \xymatrix{ v_0\bullet
    \ar@(ul,dl)_{(1,e,1)} \ar@/^/[r]^{(f_1,e,1)} & \bullet v_1
    \ar@/^/[l]^{(f_2,f,f_3)}}\] with $\B_{v_0} = \bk{F,\alpha'}$ and
  $\B_{v_1} = \{1\}$. This folds down to a bouquet of two
  circles. Since it is impossible to increase $\B_{v_0}$ via
  transmissions, there must be some $f' \in \bk{F,\alpha'}$ such that
  $f'f_1 \in \bk{\beta}$ so that we can fold together the edges labeled
  $(1,e,1)$ and $(f_1,e,1)$. We may now assume that $\y = f_2tf_3$ so
  we have reduced to Case I.
\end{proof}

\begin{lem}\label{lem:1v-1e-2genmodF} If $\FRS$ has the JSJ 
  (\ref{eqn:1v-1e}) then w.l.o.g. $\beta, \beta' \not\in F\leq
  \F$. Moreover the centralizers of $\beta$ and $\beta'$ are cyclic in
  $\F$, so after balancing folds the splitting is
  1-acylindrical. After balancing folds, where $\F_f$ denotes the
  resulting vertex group, we have w.l.o.g. $\F_f =
  \bk{F,\x,\beta'}$. In particular $\F$ is also 2-generated modulo $F$.
\end{lem}
\begin{proof} Let $\pi: \FRS \rightarrow \FRSP$ be a composition of
  strict epimorphisms that is injective on $\F$. Suppose there is a
  one edge splitting of $\FRSP$ modulo $F$ with either trivial or
  cyclic edge group such that $D$, the induced splitting of $\F$, is
  non-trivial and $\beta$ is elliptic. $\beta'$ is then also elliptic
  so we can refine the splitting (\ref{eqn:1v-1e}), contradicting the
  fact that it is a JSJ. It follows that $\beta,\beta'$ must be
  hyperbolic elements in the generalized JSJ of $\F$ and therefore
  have cyclic centralizers in $\F$. By Corollary \ref{cor:get-split},
  $\beta$ is not conjugable in $\FRS$ into $F$.  1-acylindricity after
  balancing folds now follows.

  Suppose we performed our balancing folds and the resulting splitting
  of $\FRS$ has the unique vertex group $\F_f$. Let $T$ the Bass-Serre
  tree corresponding to this splitting and consider the marked
  generating set $(F;\{\x,\z\})$ we have that $T_F$ must be a point,
  so by Theorem \ref{thm:weidmann-nielsen} after Weidmann-Nielsen
  normalization $\x$ can be sent into $\F$. It follows that we must
  have $\beta \in \bfxk$. And since we must have $T_\bfxk \cap \z
  T_\bfxk \neq \emptyset$ by 1-acylindricity we easily conclude $\z =
  f_1t^\pmo f_2$ for $f_1,f_2 \in \F$, conjugating boundary
  monomorphisms, we may therefore assume that $\z = s$ and arguing
  exactly as in Case I of the proof of Lemma \ref{lem:1v-2e-F} we have
  $\F_f=\brakett{F,\x,\beta'}$. The 2-generation of $\F$ modulo $F$
  now follows from Lemma \ref{lem:adjoin-a-root}.
\end{proof}

\subsection{When all the vertex groups of the JSJ of $\FRS$ except $\F$ are
  abelian (revisited)}\label{sec:abelian-2-gen-mod-F}

We are now able to prove that when all the vertex groups of the JSJ of
$\FRS$ except $\F$ are abelian, then $\F$ is also 2-generated modulo
$F$.

\begin{prop}\label{prop:abelian-Ftilde-2gen}
  If $\FRS$ is as in Proposition 
  \ref{prop:simply-connected-abelian} then $\F$ is generated by $F$
  and at most two other elements.
\end{prop}

\begin{proof}
  By Proposition \ref{prop:abelian-classification} we may assume
  $\FRS$ has the vertex groups $\F$ and $A$, which is abelian. 

  {\bf Case I:} If the JSJ of $\FRS$ is $\F*_\bk{\alpha}A$ then,
  perhaps after making a balancing fold which replaces $\F$ by $\F_f$,
  we have a retraction $\F_f*_\bk{\alpha}A \rightarrow \F_f$ so by
  Lemma \ref{lem:adjoin-a-root} $\F$ is two-generated modulo $F$.

  {\bf Case II:} Suppose now that the JSJ of $\FRS$
  is \[\relpres{\bk{\F,A,t \mid \alpha^t = \alpha'}}{\alpha,\alpha'
    \in \F, \bk{\beta} = \F \cap A}\] by Proposition
  \ref{lem:2v-2e-abelian} w.l.o.g., conjugating boundary monomorphisms
  if necessary, either $\beta$ or $\alpha$, but not both, lies in $F$.

  {\bf Case II.I:} Suppose $\beta$ lies in $F$, then by Corollary
  \ref{cor:get-split} the JSJ of $\bk{\F,t}$ is $\bk{\F,t \mid \alpha^t =
    \alpha'}$. Again we have a retraction $\FRS \rightarrow \bk{\F,t \mid
    \alpha^t = \alpha'}$ (balancing folds aren't necessary since we
  can't add proper roots to elements of $F$), so that $\bk{\F,t \mid
    \alpha^t = \alpha'}$ is two-generated modulo $F$. The result now
  follows from Lemma \ref{lem:1v-1e-2genmodF}.

  {\bf Case II.II:} Suppose finally that $\alpha$ lies in $F$. Then by
  Corollary \ref{cor:get-split} $\bk{\F,t \mid \alpha^t = \alpha'}$
  admits a non-trivial $(\leq\Z)$-splitting modulo
  $\alpha,\alpha'$. We first assume that all possible balancing folds
  were applied and, abusing notation, we do not change the notation
  for the vertex groups, by Lemma \ref{lem:adjoin-a-root} this will not
  affect the result. Again we have a retraction $\FRS \rightarrow
  \bk{\F,t \mid \alpha^t = \alpha'}$ so that $\bk{\F,t \mid \alpha^t =
    \alpha'}$ is 2-generated modulo $F$. Note moreover that
  $b_1(\bk{\F,t \mid \alpha^t = \alpha'}) < b_1(\FRS)$ which implies
  (since $\F\neq F$) that $b_1(\bk{\F,t \mid \alpha^t = \alpha'})=N+1$.

  {\bf Case II.II.I:} Suppose first that $\bk{\F,t \mid \alpha^t =
    \alpha'}$ is freely decomposable modulo $F$, say as
  $\widehat{F}*H$ with $F\leq \widehat{F}$. Then $\F$ cannot be
  elliptic w.r.t. this splitting since otherwise $\F\leq \widehat{F}$
  which means we must have $t \in \widehat{F}$ as well --
  contradiction. $\F$ can therefore split as $\F'*K$ with $\alpha \in
  \F' \leq \widehat{F}$ and $\alpha' \in K$ and we have $\bk{\F,t \mid \alpha^t = \alpha'}
  = \big(\F'*_\bk{\alpha} {}^tK \big)* \bk{t}$. Since $b_1(\F) =
  N+1$. We must have $\F'*_\bk{\alpha} {}^tK = F$ so that $\F =
  F*t^\mo\bk{\alpha} t$ and the result holds.

  {\bf Case II.II.II:} Suppose now that the JSJ of $\bk{\F,t \mid \alpha^t
    = \alpha'}$ has only one vertex group. Then the only possibility
  is that the JSJ of $\bk{\F,t \mid \alpha^t = \alpha'}$ is
  $\bk{\F',t,s\mid\alpha^t = \alpha', \delta^s = \delta'}$ with
  $\alpha,\alpha',\delta,\delta' \in \F'$. In particular we have $\F =
  \bk{\F',s}$. We may again assume that this splitting is balanced. By
  Lemma \ref{lem:1v-2e-F} we have $\F' = \bk{F,\alpha',\delta'}$ which
  means that $\F = \bk{F,\alpha',s}$ and the result holds.

  {\bf Case II.II.III:} Suppose finally that the JSJ of $\bk{\F,t \mid
    \alpha^t = \alpha'}$ has more than 2 vertex groups. Then the JSJ
  can be obtained by refining the splitting $\bk{\F,t \mid \alpha^t =
    \alpha'}$. In particular the underlying graph of the JSJ is not
  simply connected. Since $b_1(\bk{\F,t \mid \alpha^t = \alpha'}) =
  N+1$, Corollaries \ref{cor:abelian-low-betti-uhd} and
  \ref{cor:nonab-uhd} imply that that $\bk{\F,t \mid \alpha^t =
    \alpha'}$ has no non-cyclic abelian subgroups. It therefore
  follows from Proposition \ref{prop:2e-h-3gen} item 1.  and Corollary
  \ref{cor:2v-3e} that $\F=F*_{\alpha}H$ or $\bk{F,H,s \mid \gamma^s =
    \epsilon}$ with $\bk{\alpha} = F \cap H$ and $\gamma \in F,
  \epsilon \in H$, both of which are generated by two elements modulo
  $F$. The result therefore holds.
\end{proof}

\begin{proof}[proof of Proposition \ref{prop:abelian} ]The result now
  follows from Propositions \ref{prop:simply-connected-abelian},
  \ref{prop:abelian-classification}, and
  \ref{prop:abelian-Ftilde-2gen}.
\end{proof}

\subsection{Uniform hierarchical depth does not increase for finitely
  generated subgroups}

The purpose of this section is to show that uniform hierarchical depth
is well behaved when passing to finitely generated subgroups. This is necessary to
bound the $\uhd$ of $\F$ when JSJ of $\FRS$ has one vertex and one
edge. This next lemma is essentially an application of Theorem
\ref{thm:folded} coupled with the observation that if edge groups are
cyclic or trivial, there are only finitely many adjustments and
transmissions that can be applied to a $\G(X)$-graph $\B$ that
actually change the $\B$-vertex groups.

\begin{lem}\label{lem:fg-vertex-groups}
  Let $G$ be finitely generated. If $G$ is the fundamental group of a graph of groups
  with edge groups either cyclic or trivial, then all the vertex
  groups are f.g.
\end{lem}

\begin{lem}\label{lem:maximal-vs-jsj}
  Let $\G(X)$ be $(\leq \Z)$-splitting (see Definition
  \ref{defn:splitting}) of $G$ (modulo $F$) and let $\G(Y)$ be
  the generalized JSJ of $G$ (modulo $F$). Every rigid vertex group of
  $\G(Y)$ is conjugable into a vertex group of $\G(X)$.
\end{lem}

\begin{proof}
  Let $Y_u$ be a rigid (i.e. non-QH, non-abelian) vertex group of
  $\G(Y)$ and suppose towards a contradiction that $Y_u$ was not
  elliptic in $\G(X)$. In such a case we can collapse $\G(X)$ to some
  elementary $(\leq-\Z)$ splitting $D$ such that $Y_u$ is hyperbolic.

  We may first suppose that $G$ is freely indecomposable (modulo
  $F$). Let $\bk{c}$ be the edge group of $D$. Then by items (2) and
  (3) of Theorem \ref{thm:jsj}, we can obtain $D$ from $\G(Y)$ by
  perhaps first refining $\G(Y)$ by further splitting a MQH subgroup
  along a s.c.c. (if $c$ is conjugable into a MQH subgroup) and then
  performing a sequence of slides, foldings and collapses as described
  in Definition \ref{defn:moves}. All these moves preserve the
  ellipticity of $Y_u$ -- contradiction.

  Suppose now that $G$ is freely decomposable (modulo $F$). We have by
  definition of a generalized JSJ that $Y_u$ is a rigid vertex group
  of the cyclic JSJ decomposition of $G_i$ (modulo $F$), where $G_i$
  is a non-free free factor of the Grushko decomposition of $G$
  (modulo $F$). The elementary splitting $D$ induces a non-trivial
  \emph{cyclic} splitting of $G_i$ (modulo $F$) with $Y_u$
  hyperbolic. We can now derive a contradiction as in the previous
  paragraph arguing with $G_i$ in place of $G$.
\end{proof}

\begin{thm}\label{thm:uhd}
  Let $G$ be finitely generated fully residually free group and let $H\leq G$ be a
  finitely generated subgroup, then $\uhd(H)\leq\uhd(G)$.
\end{thm}
\begin{proof}
  We proceed by induction on uniform hierarchical depth.  If $G$ is a
  finitely generated fully residually free group such that $\uhd(G)=0$ then the same
  is true for any finitely generated subgroup of $G$.

  Suppose now that $\uhd(G)=n+1$ and that the theorem held for $m \leq
  n$. Let $H\leq G$ be a finitely generated subgroup.  Let $E$ denote the
  generalized JSJ of $G$. If $H$ is conjugable into a vertex group
  then by induction hypothesis, $\uhd(H) \leq n$ and the result holds.

  Suppose now that $H$ is hyperbolic w.r.t. $E$. Then $H$ has an
  induced $(\leq \Z)$-splitting $D$ as a finite graph of groups with
  vertex groups conjugable into the vertex groups of $E$. On one hand
  the vertex groups of $D$ are conjugable into the vertex groups of
  $E$. On the other hand by Lemma \ref{lem:maximal-vs-jsj} the rigid
  vertex groups of the generalized JSJ of $H$ are conjugable into the
  vertex groups of $D$ and by Lemma \ref{lem:fg-vertex-groups} the
  rigid vertex groups are finitely generated. It follows that we can apply the
  induction hypothesis so for each rigid vertex group $H_i$ of the
  generalized JSJ of $H$ we have $\uhd(H_i) \leq n$. Noting that QH
  and abelian vertex groups have $\uhd=0$, we can now conclude that
  $\uhd(H)\leq n+1$. So the result holds by induction.
\end{proof}

\begin{lem}\label{lem:uhd-leq-uhdf}
  If $G$ is fully residually $F$ then $\uhd(G) \leq \uhdf(G)$.
\end{lem}
\begin{proof}
  We proceed by induction on $\uhdf(G)$. If $\uhdf(G) = 0$ then the
  result holds. Let $\F_G$ be the vertex group of the generalized JSJ
  of $G$(modulo $F$ containing $F$ and let $E$ be the generalized
  cyclic JSJ splitting of $G$ (\emph{not} modulo $F$).

  Suppose for all $m \leq n, \uhdf(K) \leq m \Rightarrow \uhd(K)\leq
  \uhdf(K)$ where $K$ is fully residually $F$. Let $\uhdf(G)=n+1$ and
  let $F \leq H \leq G$ be a finitely generated subgroup.  By definition
  $\uhdf(\F_G) \leq n$ so by induction hypothesis $\uhd(\F_G) \leq
  \uhdf(\F_G)$. Now the vertex groups $G_i$ of $E$ or are either
  finitely generated subgroups of $\F_G$ or finitely generated subgroups of other vertex groups of the
  generalized JSJ of $G$ modulo $F$. In both cases by Theorem
  \ref{thm:uhd} and by induction hypothesis for each
  vertex group $G_i$, $\uhd(G_i) \leq n$, so $\uhd(G)\leq n+1$. The
  result now follows by induction.
\end{proof}

\begin{cor}\label{cor:uhdf-monotonic}
  Let $F \leq H \leq G$ where $G$ is fully residually $F$ and $H$ is
  finitely generated, then $\uhdf(H)\leq\uhdf(F)$.
\end{cor}
\begin{proof}
  We proceed by induction on $\uhdf(G)$. If $\uhdf(G)=0$ then the
  result holds.
  
  Suppose the Corollary was true for all $G$ such that $\uhdf(G) \leq
  n.$ Let $\uhdf(G) = n+1$.  Let $E$ be the generalized JSJ of $G$
  modulo $F$ and let $D$ be the generalized JSJ of $H$ modulo $F$. Let
  $\F_G$ and $\F_H$ be the vertex groups of $E$ and $D$
  respectively containing $F$.

  As argued in Theorem \ref{thm:uhd}, the rigid vertex groups of $D$
  are conjugable into the vertex groups of $E$. It therefore follows
  that $\F_H \leq \F_G$, and by induction hypothesis $\uhdf(\F_H) \leq
  \uhdf(\F_G) \leq n$. As for the other rigid vertex groups
  $H_1,\ldots, H_m$ in $D$, Lemma \ref{lem:uhd-leq-uhdf} implies that
  for each vertex group $G_i$ of $E$, $\uhd(G_i) \leq n$, so by
  Theorem \ref{thm:uhd} $\uhd(H_i)\leq n$. It follows that
  $\uhd(H)\leq n+1$. The result now follows by induction.
\end{proof}

\begin{cor}\label{cor:countereg}
  There is a one relator fully residually $F$ group which does not
  embed in a centralizer extension of $F$.
\end{cor}
\begin{proof}
  The group $\FRS$ given in Example \ref{eg:D-I-1} is freely
  indecomposable modulo $F$ and has a cyclic splitting modulo $F$ with
  a vertex group $F \leq \F_1 = \bk{F,t \mid [t,u]=1}$ which doesn't
  split modulo $F$ and the incident edge group. $\uhd_F(\F_1) = 1$
  which means that $\uhdf(\FRS) =2$. On the other hand any centralizer
  extension of $F$ has $\uhdf=1$, so by Corollary
  \ref{cor:uhdf-monotonic} $\FRS$ cannot embed into any rank 1
  centralizer extension.
\end{proof}

\subsection{The two edge case}
In this section we describe the structure of $\F$ when the JSJ of
$\FRS$ has one vertex and two edges. In particular we will explicitly
bound its uniform hierarchical depth. Suppose
$\FRS,\F,\alpha,\alpha',\beta,\beta',s,t$ are given in as
(\ref{eqn:1v-2e}). By Lemma \ref{lem:1v-2e-F} we can assume $\alpha
\in F$ and $\beta,\beta'$ hyperbolic in the JSJ of $\F$.

\begin{defn}
  For each edge $e$ in the JSJ of $\FRS$ we define the \define{edge
    class associated to $e$} to be the conjugacy class in $\FRS$
  corresponding to the edge group associated to $e$.
\end{defn}

It is important to note that distinct edges of the JSJ of $\FRS$ may
have the same edge class.

\begin{lem}\label{lem:1v-2e-not-one-vertex}
  Let $\FRS$ has the JSJ (\ref{eqn:1v-2e}). If $\F$ is freely
  indecomposable modulo $F$ then the JSJ of $\F_f$ (as given in Lemma
  \ref{lem:1v-2e-F}) has more than one vertex group.
\end{lem}
\begin{proof}
  Suppose towards a contradiction that JSJ of $\F_f$ has only one vertex
  group. By Lemma \ref{lem:1v-2e-F} $\F_f$ is 2 generated modulo $F$
  which means that the JSJ of $\F_f$ has at most two edges.  The unique
  vertex group of $\F_f$ cannot be $F$ itself by Corollary
  \ref{lem:Ftilde-not-F}. Let $\FRSP$ be the first term in a strict
  resolution of $\FRS$ where $\F$ splits. We first consider the case
  where $\FRSP$ is freely indecomposable modulo $F$.

  We can collapse the JSJ of $\FRSP$ to an elementary splitting $\FRSP
  = \pi_1(\G(Z))$ with edge group $C$ and $\F$ hyperbolic.  If $\F$ is
  hyperbolic in $\FRSP$ then so is $\F_f$, moreover by Lemma
  \ref{lem:1v-2e-F}, after replacing $\alpha', \beta'$ by their proper
  roots if necessary, $\F_f$ is generated as $\bk{F,\alpha',\beta'}$.

  The elementary splitting $\G(Z) $ is 2-acylindrical.  If the
  JSJ of $\F_f$ has two edges, then $\F_f$ has two edge classes, but
  they cannot be conjugate in $\FRSP$ because by Lemma
  \ref{lem:1v-2e-F} exactly one of the edge classes of the JSJ $\F_f$
  will conjugable into $F$. It therefore follows that only one
  edge class of $\F_f$ is conjugable into $C$ in $\FRSP$. Therefore
  only one of the edge classes of the JSJ of $\F_f$ corresponds to
  edge groups of induced splitting of $\F_f$ in $\FRSP$.

  Although the JSJ of $\F_f$ has only one vertex group, it is still
  possible for the induced splitting of $\F_f$ to have more than one
  vertex group, but by the previous paragraph the induced
  graph of groups decomposition of $\F_f$ given by the $\G(Z)$-graph
  $\B$ has at most one non-cyclic $\B$ vertex group $\B_v$ and at most
  one cycle. Moreover since this splitting is non-trivial and must
  collapse to a cyclic HNN extension. Since $\G(Z)$ is
  2-acylindrical this means that the cycle in $\B$ has length at most
  2, since when $\F_f$ acts on the Bass-Serre tree $T$ corresponding to
  $\G(Z)$ the edge group of the JSJ of $\F_f$ that is conjugable into
  $C$ should fix an arc between two translates of some $v_0 \in T$
  with $\fix(\B_v) = v_0$.

  Suppose first that $\FRSP$ split as a cyclic HNN extension $\bk{A,r
    \mid c^r = c'}$ with $c,c'$ not conjugate in $A$ and $F \leq A$,
  and $\F$ not elliptic. Up to conjugating boundary monomorphisms (or
  equivalently replacing $r$ by $g_1rg_2$ for some $g_1,g_2 \in A$) if
  necessary, the possible induced splittings of $\F_f$ are given by
  the folded $\G(Z)$-graphs:\begin{equation}\label{eqn:2e-induced-hnn}
    \B = \xymatrix{u\bullet \ar@/_/[r]_{(1,e,1)}
      \ar@/^/[r]^{(a_1,e,a)}& \bullet v} \tr{~or~} \xymatrix{u\bullet
      \ar@(ul,ur)^{(1,e,1)}}\end{equation} with $B_u = \F_1 \leq \F_f,
  \B_v = \bk{c'}$ and $a \in A \setminus \bk{c'}$ but with $[a,c']=1$
  and $a_1 \in A$. In particular the first possibility can only occur
  if $\bk{c'}$ isn't malnormal and hence cyclic in $A$. Let $\G(Y)$
  give the JSJ of $\FRSP$. Since $\G(Z)$ is just a collapse of
  $\G(Y)$, $C = \bk{c}$ is an edge group of $\G(Y)$. By Corollary
  \ref{cor:edge-group-centralizers} $C$ must be an edge group adjacent
  to an abelian vertex group of $\G(Y)$. Since the elementary
  splitting is as an HNN extension $C$ is the edge group of a
  non-separating edge in $e \subset Y$ and has infinite intersection
  with an edge group that is incident to an abelian vertex group of
  $\G(Y)$. It follows that the JSJ of $\FRSP$ has more than one vertex
  and by Lemma \ref{lem:abelian-vertex-gp-malnormal} the JSJ of
  $\FRSP$ must have at least two non-abelian vertex groups.

  Now $\B_u=\F_1$ lies in the subgroup corresponding to the
  ``subgraph of groups'' $\G(Y\setminus e)$, which has vertex groups
  $F$, some free group $H$, and an abelian vertex group $A$. It
  therefore follows that $\F_1$ has an induced splitting $D'$ with $F$ as a
  vertex group, moreover since $\bk{c'},\bk{c}$ are elliptic in $\G(Y)$
  we can refine the splitting of $\F_f$ so that it has $F$ as a vertex
  group, contradicting the fact that its JSJ had only one vertex
  group, and this vertex group isn't $F$.

  If $\bk{c}$ and $\bk{c'}$ are malnormal in $A$ then the second
  possibility of (\ref{eqn:2e-induced-hnn})must occur. By Lemma \ref{lem:1v-2e-F}, $\F_f =
  \bk{F,\alpha',\beta'}$ with $\beta \in \bk{F,\alpha'}$ and $\alpha'$
  conjugate to $F$. It therefore follows that in $\FRSP = \bk{A,r \mid
    c^r = c'}$, $\beta$ has exponent sum 0 in $r$. Now $\beta$ is
  conjugate to $\beta'$ in $\FRSP$ which means that $\beta'$ also has
  exponent sum 0 in $r$. In other words $\F_f$ in $\FRSP$ is generated
  by elements with exponent sum 0 in $r$ which is impossible because
  the induced splitting on the right of (\ref{eqn:2e-induced-hnn})
  implies that $\F_f$ contains an element with exponent sum 1 in $r$.

  Suppose now that $\FRSP$ can split as a cyclic free product with
  amalgamation $A*_CB$, with $F \leq A$ and $B$ non-abelian and $\F$
  not elliptic. This means that $\FRSP$ has a one edge maximal abelian
  collapse. It therefore follows from Proposition \ref{prop:2v-class}
  that the induced splitting of $\F$ has one edge class and one vertex
  group $\F_1$. More specifically, since $A*_CB$ is 2-acylindrical,
  the induced splitting will be given by some $\G(X)$-graph \[\B =
  \xymatrix{ u \bullet \ar@/_/[r]_{(1,e,1)} & \ar@/_/[l]_{(b,e^\mo,a)}
    \bullet v}
  \] with $a\in A, b\in B$ and where $\B_u=F_1$ and $\B_v=C$. We see
  that this is only possible if there is some element $b \in B$ that
  doesn't lie in $C$ but centralizes the cyclic edge group $C$. By
  Lemma \ref{lem:abelian-vertex-gp-malnormal} and Proposition
  \ref{prop:2v-class}, $\FRSP$ must have a cyclic splitting of the form
  $F*_{\bk{u}}(H*_{\bk{u}}K)$ with $K$ abelian. So we could chose
  $A=F$ and still have $\F_f$ hyperbolic, but this would imply
  $\F_1 = F$ which is impossible.

  If $\FRSP = A*_CB$ with $B$ abelian of rank 3 then $A = F$, which
  again yields a contradiction. Recall that we are assuming that the
  JSJ of $\F$ has only one vertex group. If $B = \bk{c} \oplus \bk{r}$
  is free abelian of rank 2 then by 2-acylindricity of the JSJ of
  $\FRSP$, the only possible $\G(Z)$ graph is: \[ \B = \xymatrix{u
    \bullet \ar@/^/[r] ^{(a_2,e,b_2)} \ar@/_/[r]_{(a_1,e,b_1)} &
    \bullet v}\] with $a_i \in A; b_j \in B$ since this graph is
  folded we must have that $b_1b_2^\mo \not\in \bk{c}$. Now note that
  we can also write $\FRSP$ as an HNN extension $\bk{A,r \mid c^r
    =c}$. The fact that $a_1b_1b_2^\mo a_2^\mo \in \F$ implies that
  there is an element of $\F$ that has non-zero exponent sum in
  $r$. But by Lemma \ref{lem:1v-2e-F} $\F = \bk{F,\alpha',\beta'}$
  with $\alpha'$ conjugate to $\alpha \in F$ and $\beta'$ conjugate to
  $\beta \in \bk{F,\alpha'}$. So $\F$ is generated by elements with
  exponent sum 0 in $r$ -- contradiction.

  Finally, consider the case where $\FRSP$ is freely decomposable
  modulo $F$. Since $\F_f$ is freely indecomposable modulo $F$, it must
  lie in one of the free factors $H$ of a Grushko decomposition of
  $\FRSP$ modulo $F$. On the other hand we have the relations
  $\pi(s)^\mo \alpha \pi(s) = \alpha'$ and $\pi(t)^\mo \beta \pi(t) =
  \beta'$ which are possible if and only if $s,t \in H$ which implies
  $\FRSP = H$ -- contradiction.
\end{proof}

\begin{lem}\label{lem:abelian-vertex-gp-malnormal}
  If $\FRS$ is as in Proposition \ref{prop:simply-connected-abelian}
  then the centralizers of the edge groups that are incident to $\F$
  are cyclic. Moreover if the JSJ of $\FRS$ has two edges then the
  incident edge groups are conjugacy separated in $\F$.
\end{lem}
\begin{proof}
  If the JSJ of $\FRS$ has underlying graph\[\xymatrix{ u\bullet \ar@{-}[r] & \bullet v \ar@{-}[r] & \bullet
        w}\] then $\F = F$ by Proposition
      \ref{prop:abelian-classification} and the result follows. If the
      JSJ of $\FRS$ has underlying graph\[\xymatrix{ u\bullet
        \ar@{-}[r] & \bullet v}\] and $\F \neq F$ then $\F$ must not
      have any splittings modulo $F$ and the edge group. It follows
      that the edge group is hyperbolic in the generalized JSJ of
      $\F$, and therefore cannot be non-cyclic abelian. Finally if the
      JSJ of $\FRS$ has underlying graph 
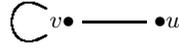
\[ \xymatrix{ v\bullet
        \ar@{-}[r] \ar@(ul,dl)@{-} & \bullet u} \] then the abelian
      vertex group has rank at most 2, so we may consider this group
      as double HNN extension of $\F$. The result now follows by
      applying Lemma \ref{lem:2v-2e-abelian}.
\end{proof}

\begin{cor}\label{cor:edge-group-centralizers}
  Let $\FRS$ be freely indecomposable modulo $F$ and let $\bk{c}$ be
  one of the edge groups of the JSJ of $\FRS$. If $\bk{c}$ has a
  non-cyclic centralizer $Z(c)$, then $Z(c)$ is in fact an abelian
  vertex group of the JSJ of $\FRS$.
\end{cor}
\begin{proof}
  If the JSJ of $\FRS$ has one vertex group then by Lemmas
  \ref{lem:1v-2e-F} and \ref{lem:1v-1e-2genmodF} $Z(c)$ is cyclic. If
  the JSJ $\FRS$ has at least two non-abelian vertex groups, the
  result follows by looking at the non-cyclic abelian subgroups that occur in
  Section \ref{sec:2-nonab}. The remaining case follows from Lemma
  \ref{lem:abelian-vertex-gp-malnormal}.
\end{proof}

\begin{cor}\label{cor:1v-ev-uhd}
  Let $\FRS$ and $\F_f$ be as in Lemma \ref{lem:1v-2e-F}. If $\F$ is
  freely indecomposable modulo $F$, then the JSJ of $\F_f$ has two
  vertex groups $F$ and some free group $H$. It follows that in all
  cases $\uhd(\FRS)\leq 2$.
\end{cor}
\begin{proof}
  By Lemma \ref{lem:1v-2e-F}, $\F_f$ has no non-cyclic abelian
  subgroups. By Lemma \ref{lem:1v-2e-not-one-vertex} the JSJ of $\F_f$
  has at least two vertex groups and since it cannot contain any
  non-cyclic abelian subgroups Corollary \ref{cor:nonab-uhd} implies
  that $\uhd_F(\F_f) = 1$. Now $\F$ is a finitely generated subgroup
  of $\F_f$ so by Corollary \ref{cor:uhdf-monotonic} $\uhdf(\F) \leq
  1$ as well and the result follows.
\end{proof}

\subsection{The one edge case}\label{sec:one-edge-case}
In this section we bound the uniform hierarchical depth of $\FRS$ when
its JSJ has one edge and one vertex.  By Lemma
\ref{lem:1v-1e-2genmodF} $\F$ is generated by two elements modulo $F$
and the element $\beta \in \F$ must be hyperbolic in the JSJ of
$\F$. The JSJ of $\F$ either has only one vertex group or is one of
the groups described in Propositions \ref{prop:decomposable} and
Section \ref{sec:2-nonab}.

\begin{lem}\label{lem:strict-split}
  $\FRS, \F, t, \beta, \beta'$ be as in (\ref{eqn:1v-1e}) and let
  $\pi: \FRS \rightarrow \FRSP$be a proper strict epimorphism. If $\F
  \neq F$ is freely indecomposable modulo $F$ then there is no
  essential cyclic or free splitting of $\FRSP$ modulo $\F = \pi(\F)$.
\end{lem}
\begin{proof}
  Suppose $\FRSP = A*B$ with $\F \leq A$. Then since $\FRSP =
  \bk{\F,\pi(t)}$, $\pi(t)$, written as a reduced word in
  $A,B$, should have a syllable in $B$. We must, however, have the
  equality
  \begin{equation}\label{eqn:frsp}
    \pi(t^\mo) \beta \pi(t) = \beta'
  \end{equation} 
  with $\beta,\beta' \in \F \leq A$ which is impossible. 

  We may therefore assume that $\FRSP$ is freely indecomposable and
  that we can collapse its cyclic JSJ to either a free product with
  amalgamation or as an HNN extension modulo $\F$.

  Suppose $\FRSP$ splits as a cyclic HNN extension $\bk{A,r \mid c^r
    =c'}; c,c' \in A$, but not conjugate in $A$, with $\F \leq
  A$. $\FRSP = \bk{\F,\pi(t)}$, so $\pi(t)$ must have exponent sum 1
  in the stable letter $r$. (\ref{eqn:frsp}) implies that
  w.l.o.g. $\beta$ is conjugate in $A$ to $\bk{c}$ and $\beta'$ is
  conjugate in $A$ to $\bk{c'}$. So conjugating boundary monomorphisms
  (or, equivalently, replacing $r$ by $a_1ra_2; a_i \in A$) we can
  arrange so that $\beta \in \bk{c}$ and $\beta' \in \bk{c'}$ which
  means that $\pi(t)$ has a reduced form $r \cdots r$.  Now by
  Britton's Lemma any word in $\F,\pi(t)$ is reduced if and only if it
  doesn't contain the subword $ \pi(t^\mo) \beta \pi(t)$ or $\pi(t)
  \beta' \pi(t^\mo)$. It therefore follows that (\ref{eqn:frsp}) is
  the only non-trivial relation between $\F$ and $\pi(t)$. This implies
  that $\FRS \approx \FRSP$ contradicting the fact that $\pi$ is a
  proper epimorphism.

  Suppose now that $\FRSP$ splits as $A*_CB$ with $C=\bk{c}$, $B$
  abelian, and $\F \leq A$. If $B$ has rank 3 then $A=F \Rightarrow \F
  = F$ -- contradiction. Suppose now that $B$ is abelian of rank 2,
  i.e. $B = \bk{c}\oplus\bk{b}$. Then w.l.o.g. we may assume that
  $\bk{\beta} \leq \bk{c}$. We can also express $A*_CB$ as an HNN
  extensions $\bk{A,r\mid c^r=c}$. Note that $\FRSP = \bk{\F,\pi(t)}$
  and since $\F \leq A$, $\pi(t)$ has exponent sum 1 in $r$. This
  means that if $\pi(t)^\mo \beta \pi(t) = \beta'$ then there must be
  some $g \in A \setminus \F$ such that $g \beta' g^\mo = \beta$. So
  $\pi(t)g^\mo$ centralizes $\beta$ which means by normal forms that
  $\pi(t)g^\mo =b_1 \in B$, hence $\pi(t) = b_1g$. But now looking at
  words in $\F, b_1g$ we see, again that the only non-trivial
  relations are $(b_1g) \beta' (g^\mo \beta^\mo) = \beta$ and $(g^\mo
  b_1^\mo) \beta (b_1g) = \beta'$, again forcing $\FRS \approx \FRSP$.

  Suppose now that $\FRSP$ splits as $A*_CB$ with $C$ cyclic, $B$ non
  abelian and $\F \leq A$. Then $\FRSP$ has a maximal abelian collapse
  with one edge. By Lemma \ref{lem:1v-1e-2genmodF}
  $\FRSP = \bk{F,\pi(\x),\pi(\y)}$ with $F,\pi(\x)$ lying in the same
  vertex group of the maximal abelian collapse. By Lemma
  \ref{lem:2-v-1-e-Fx} this means $\FRSP$ is freely decomposable
  modulo $F$ and since $\F$ is freely indecomposable modulo $F$ we
  have $\FRSP = A*B$ with $\F\leq A$ which, as we saw, is impossible.
\end{proof}

\begin{lem}\label{lem:add-a-root} If $\FRSP$ as in Lemma
  \ref{lem:strict-split} has a one edge cyclic splitting $D$ and if
  $\F$ has an induced one edge cyclic splitting $D_\F$, then $\FRSP$
  can be obtained from $\F$ by adding an element $\sqrt[n]{\eta}$: an
  $n^{th}$ root of the generator $\eta$ of an edge group of $D_\F$.
\end{lem}

\begin{proof}
  Consider the action of $\FRSP$ and $\F$ on the Bass-Serre tree $T$
  corresponding to the splitting $D$. Abusing notation we use
  $t$ to denote the image of $\pi(t)$ in $\FRSP$. 

  We have hyperbolic elements $\beta,\beta' \in \F$ such that
  $\beta^t=\beta'$. This means that $t \axis(\beta') =
  \axis(\beta)$. Which means that there is are edges $e',e$ $\in
  \axis(\beta'), \axis(\beta)$ respectively such that $t e'=e$. On
  the other hand $\axis(\beta),\axis(\beta') \subset T(\F)$, the
  minimal $\F$-invariant subtree. 

  $D_\F$ has only one edge and $\F$ is
  transitive on the set of edges in $T(\F)$. Let $g \in \F$ be such
  that $ge = e'$. Then $tg \in \stab_\F(e)$. Let
  $\stab_{\FRSP}(e)\cap{\F} = \brakett{\eta}$, then we have that
  $\brakett{\eta,t g } \leq \stab_{\FRSP}(e)$ which is cyclic, so
  $\brakett{\eta, t g} = \brakett{\sqrt[n]{\eta}}$.
\end{proof}

\begin{cor}\label{cor:betti-drop}
  Suppose the JSJ of $\FRS$ has one edge and one vertex and let $\pi:\FRS
  \rightarrow \FRSP$ be a strict epimorphism. If there is a one edge
  cyclic splitting of $\FRSP$ such that the induced splitting of $\F$
  only has one edge then $b_1(\FRSP) < b_1(\FRS)$.
\end{cor}
\begin{proof}
  By Lemma \ref{lem:add-a-root}, looking at abelianizations we
  immediately see that $b_1(\FRSP) \leq b_1(\F)$. Let $\bk{\beta},
  \bk{\beta'}$ be boundary subgroups in $\F$. Consider the
  mapping: \begin{eqnarray*}
    \F & \rightarrow & \F/[\F,\F] \otimes_\Z \mathbb{Q}\\
    g & \mapsto &[g].
  \end{eqnarray*} We consider two cases,  either
  $[\beta] - [\beta'] = 0$ or $[\beta] - [\beta'] \neq 0$. If $[\beta]
  = [\beta']$ then $b_1(\FRS) = b_1(\F) + 1$, which implies
  $b_1(\FRSP) < b_1(\FRS).$
  
  If $[\beta] \neq [\beta']$ then $b_1(\FRS) = b_1(\F)$, but by Lemma
  \ref{lem:add-a-root} we have ${\FRSP/[\FRSP,\FRSP] \otimes_\Z
    \mathbb{Q}}$ is obtained from $\F/[\F,\F] \otimes_\Z \mathbb{Q}$
  by adding $[\eta]/n$ (which doesn't increase the rank) and then
  forcing $[\beta] = [\beta']$ (since $\pi(t)^\mo \beta \pi(t) =
  \beta'$ in $\FRSP$) which decreases the rank by 1. So $b_1(\FRSP) <
  b_1(\F) = b_1(\FRS)$.
\end{proof}

\begin{lem}\label{lem:1v-1e-b1-N+1}
  Suppose the JSJ of $\FRS$ has one vertex and one edge, and suppose
  $b_1(\FRS) =N+1$. Then $\FRS$ has uniform hierarchical depth relative
  to $F$ at most 2. Moreover if $\F$ is freely indecomposable modulo
  $F$ its JSJ has more than one vertex group.
\end{lem}
\begin{proof}
  Abelianizing relative presentations, we immediately see that $b_1(\F) \leq
  b_1(\FRS)$. Corollary \ref{lem:Ftilde-not-F} and
  Proposition \ref{prop:b1-complexity} imply that $b_1(\F)=N+1$. 
  
  If $\F$ is freely decomposable modulo $F$ then $\uhdf(\F) = 1$ and
  the result holds.

  Let $\pi:\FRS \rightarrow \FRSP$ be a strict epimorphism, then
  $b_1(\FRSP) = m\leq N+1$. If $m = N$ then by Proposition
  \ref{prop:b1-complexity} $\FRSP =F$ and since $\pi$ is injective on
  $\F$ we have a contradiction to Corollary \ref{lem:Ftilde-not-F}.

  Consider the case where $\FRSP$ is freely indecomposable modulo $F$
  and its JSJ has at least two vertex groups. Corollaries
  \ref{cor:abelian-low-betti-uhd} and \ref{cor:nonab-uhd} imply
  $\uhdf(\FRSP) \leq 1$ so by Corollary \ref{cor:uhdf-monotonic},
  $\uhdf(\F) \leq 1$ and the result holds. From this it also follows
  that unless the JSJ of $\F$ has only one vertex group, then
  $\uhdf(\F) \leq 1$.
  
  Otherwise the JSJ of $\FRSP$ has one vertex group. By Lemmas
  \ref{lem:1v-2e-F} and \ref{lem:1v-1e-2genmodF} the centralizers of
  the edge groups of $\FRSP$ are cyclic and distinct edges have
  distinct edge classes. By Lemma \ref{lem:strict-split} any one edge
  cyclic splitting of $\FRSP$ modulo $F$ induces a cyclic of $\F$
  modulo $F$, and this splitting of $\FRSP$ is as an HNN extension
  with cyclic edge stabilizers. Since the JSJ of $\F$ also has only
  one vertex, the edge groups do not lie in non-cyclic abelian
  subgroups the vertex groups, so the the only possible induced
  splitting of $\F$ (considering the argument used for Lemma
  \ref{lem:1v-2e-not-one-vertex}) is of the form
  :
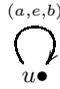
\[\xymatrix{u\bullet\ar@(ul,ur)^{(a,e,b)}}\] Corollary
  \ref{cor:betti-drop} therefore applies and gives $b_1(\FRSP) = N$ so
  $\F \leq \FRSP = F$ -- contradiction.
\end{proof}

We can now prove the following:

\begin{proof}[Proof of Proposition \ref{prop:abelian-uhd}]
  We have that $m = b_1(\F) < b_1(\FRS)$. If $m = N$ there is nothing
  to show, otherwise $m=N+1$. If the JSJ of $\F$ has more than two
  vertex groups Corollaries \ref{cor:abelian-low-betti-uhd} and
  \ref{cor:nonab-uhd} imply $\uhdf(\F) \leq 1$. Otherwise $\F$ has
  one vertex group and Lemma \ref{lem:1v-1e-b1-N+1} and Corollary
  \ref{cor:1v-ev-uhd} imply that $\uhdf(\F) \leq 2$.
\end{proof}

\begin{prop}\label{prop:1v-1e-uhd}
  If the JSJ of $\FRS$ has one edge and one vertex then it has uniform
  hierarchical depth relative to $F$ at most 4.
\end{prop}
\begin{proof}
  If $\F$ is freely decomposable modulo $F$ or if the JSJ of $\F$
  doesn't have one vertex and one edge then Corollaries
  \ref{cor:nonab-uhd} and Corollary \ref{cor:1v-ev-uhd} and Proposition
  \ref{prop:abelian-uhd} imply that $\uhdf(\F) \leq 3$.

  By Corollary \ref{cor:uhdf-monotonic}, the same bound holds if
  $\pi:\FRS \rightarrow \FRSP$ is a strict epimorphism and the JSJ of
  $\FRSP$ doesn't have one edge and one vertex.

  The remaining possibility is that the JSJs of $\F$ and $\FRSP$ both
  only have one edge and one vertex. But then as in the proof Lemma
  \ref{lem:1v-1e-b1-N+1} the induced splitting of $\F$ has only one
  edge and  by Corollary \ref{cor:betti-drop} $b_1(\FRSP) \leq N+1$ which
  by Lemma \ref{lem:1v-1e-b1-N+1} and Corollary
  \ref{cor:uhdf-monotonic} imply that $\uhdf(\F) \leq 2.$
\end{proof}

\begin{proof}[Proof of Proposition \ref{prop:1v}]
  The result follows immediately from Corollary \ref{cor:1v-ev-uhd}
  and Proposition \ref{prop:1v-1e-uhd}.
\end{proof}

\bibliographystyle{alpha} \bibliography{biblio.bib}
\end{document}